\documentclass[reqno,a4paper,11pt]{amsart}
\usepackage{hyperref}
\usepackage{tikz}
\usetikzlibrary{arrows.meta,calc,decorations.pathreplacing}
\usepackage[active]{srcltx}
\usepackage[curve]{xypic}
\usepackage{graphicx}
\usepackage{mathrsfs}
\usepackage{geometry}
\usepackage{color}

\usepackage[active]{srcltx}
\usepackage{epsfig,amsmath}
\usepackage[english]{babel}
\usepackage{amsmath,amsfonts,amssymb,amscd,color,epsfig,amsthm}
\providecommand{\th}{Theorem}
\newcommand{\Sone}{\mathbb R/\mathbb Z}
\newcommand{\sig}{\sigma_d}

\newcommand{\Hd}{\operatorname{dim}_{\!H}}

\newcommand{\CH}{\operatorname{CH}}

\newcommand{\Imp}{\operatorname{Imp}}

\newtheorem{theorem}{Theorem}[section]
\newtheorem{lemma}[theorem]{Lemma}
\newtheorem{proposition}[theorem]{Proposition}
\newtheorem{corollary}[theorem]{Corollary}

\newtheorem{definition}[theorem]{Definition}

\theoremstyle{remark}

\theoremstyle{plain}

\numberwithin{equation}{section}

\newcommand{\wh}{\widehat}

\def\EEE{{\cal E}}
\def\FFF{{\cal F}}

\def\GGG{{\cal G}}

\def\JJJ{{\cal J}}

\def\LLL{{\cal L}}

\def\RRR{{\cal R}}
\def\SSS{{\cal S}}

\def\WWW{{\cal W}}

\def\PPPP{\mathscr P}

\def\CCCC{\mathscr C}

\def\g{\gamma}
\def\G{\Gamma}

\def\De{\Delta}
\def\de{\delta}
\def\x{\xi}

\def\R{\mbox{$\mathbb R$}}

\def\C{\mbox{$\mathbb C$}}

\def\D{\mathbb D}

\def\Z{\mbox{$\mathbb Z$}}

\def\lv{ \left(\begin{matrix} }
 \def\rv{\end{matrix}\right)}

\def\cal{\mathcal}

\def\dw{{\dw}}

\def\ov{\overline}

\usepackage{tikz}
\usetikzlibrary{arrows}
\tikzstyle{every picture}=[> = to]
\tikzset{cdlabel/.style={execute at begin node=$\scriptstyle,execute at end node=$}}
\tikzset{implication/.style={double equal sign distance, -implies}}
\tikzset{biimplication/.style={double equal sign distance, implies-implies}}

\def\dw{{\ \mathrm{d}w}}
\newcommand{\PaperNeedspace}[1]{\par\begingroup
  \dimen0=#1\relax\advance\dimen0 by\pagetotal
  \ifdim\dimen0>\pagegoal\newpage\fi\endgroup}
\hypersetup{hidelinks,pdftitle={Continuity of biaccessible dimension for polynomials}}

\title{Continuity of  biaccessibility dimension for polynomials}
\date{}
\author{Yan Gao}
\address{Yan Gao, School of Mathematical Sciences, Shenzhen University, Shenzhen 518061, China}
\email{gyan@szu.edu.cn}

\begin{document}
\begin{abstract}
For a polynomial \(f\) of degree \(d\ge 2\), its biaccessible set \(B(f)\) consists of angles whose external rays are smooth and land at a point receiving at least two smooth external rays. We prove that the function \(f\mapsto \Hd(B(f))\) is continuous on the entire space of degree-\(d\) polynomials. Therefore $(\log d)\Hd B(f)$ gives a continuous extension of classical core entropy from the postcritically finite case to all  polynomials of degree $d$.
\end{abstract}
\maketitle

\section{Introduction}
\label{sec:introduction}

Let \(\PPPP_d \simeq \mathbb{C}^{d-1}\) denote the space of monic centered polynomials of degree \(d \ge 2\). A central problem in complex dynamics is to study how the dynamics changes as polynomials vary in (canonical subsets of) $\PPPP_d$.
 A  subset of \(\PPPP_d\) that has attracted particular interest is the \emph{connectedness locus}
\[
\CCCC_d = \{ f \in \PPPP_d : \text{the Julia set of } f \text{ is connected} \}.
\]
When $d=2$, the connectedness locus $\CCCC_2$ is the celebrated Mandelbrot set, which has been extensively studied. For $d\geq 3$, however, the structure of $\CCCC_d$ becomes considerably more intricate and remains much less understood.


To gain a qualitative understanding of the connectedness locus, W.~Thurston introduced a new invariant, known as the \emph{core entropy} $h(f)$, for each postcritically finite polynomial $f$. It is defined as the topological entropy of $f$ restricted to its Hubbard tree (see Definition~\ref{def:Hubbard}). In \cite{TG}, Thurston laid the foundations for the theory of core entropy and, in particular, developed an algorithm for computing core entropy from critical angle data (see also \cite{G}).

Another quantity closely related to core entropy is the
\emph{biaccessibility dimension}. For any $f\in\PPPP_d$,
let $B(f)$ denote the set of angles whose external rays are smooth and land at a point receiving at least two smooth external rays. The biaccessibility dimension of $f$ is $\dim_H B(f)$. For postcritically finite polynomials,
it satisfies
\begin{equation}\label{eq:0}
h(f)=\log d\cdot \Hd(B(f)).
\end{equation}
See \cite[Section~1.1]{GT} and the references therein.
Unlike core entropy defined via Hubbard trees, the
biaccessibility dimension remains well-defined even when
the Julia set is disconnected or fails to be locally connected.

One main result of this paper establishes the continuity of the biaccessibility dimension as a function on \(\PPPP_d\).

\begin{theorem}\label{thm:main}
The function \(f\mapsto \Hd(B(f))\) is continuous on \(\PPPP_d\).
\end{theorem}
Consequently, $(\log d)\dim_H B(f)$ provides a continuous extension of the classical core entropy to all polynomials of degree $d$.

Our proof of Theorem~\ref{thm:main} relies neither on the structure of the parameter space nor on the continuity of core entropy on the space of critical portraits established in \cite{GT}. Instead, it rests on the following combinatorial entropy-dimension formula, which extends \eqref{eq:0} to {general laminational equivalence relations}.

Let $\sim$ be a \emph{laminational equivalence relation}, and
write $\LLL=\LLL_\sim$ and $g=g_\sim$ for the associated
invariant lamination and topological polynomial, respectively.
The \emph{biaccessible set} $B(\LLL)$ is the union of all
non-singleton $\sim$-classes.

A \emph{marking portrait} $\xi=\xi_\sim$ specifies the Julia critical
classes and dynamically compatible choices of angles in
critical Fatou gaps, following  the critical marking framework
of \cite{Poi,Kiwi2,Z}.
Let $h(\xi)$ denote the logarithmic growth rate of closed
paths in the \emph{separation graph} $\Gamma_\xi$ arising from
Thurston's entropy algorithm. It is called the \emph{(combinatorial) core entropy} of the critical portrait $\xi$.

\begin{theorem}\label{thm:equal}
For every marking portrait $\xi$ of $\LLL$,
\[
h(\xi)=(\log d)\Hd(B(\LLL)).
\]
\end{theorem}

In particular, this core entropy is independent  of the marking. 

The key idea in the proof of Theorem~\ref{thm:equal} is to
convert closed paths in $\G_\xi$ based at a vertex into
a monotone horseshoe on a postcritical arc in the
reduction of $g$. This serves as a substitute for
a finite Hubbard-tree Markov partition.

For Theorem \ref{thm:main}, the main new ingredient is a
\emph{uniform realization procedure} that realizes
combinatorial interval horseshoes as conformal strip
systems in the polynomial dynamical plane, with arbitrarily small
entropy loss and uniform control of return times and
boundary-angle periods.

\subsection{Relation to previous work}\label{sec:related}

The study of biaccessible angles predates the continuity
problem for core entropy. Early work of Zakeri, Zdunik
and Smirnov investigated the measure-theoretic and
dynamical properties \cite{Zakeri,Zdunik,Smirnov}.
Meerkamp and Schleicher proved that the biaccessibility
dimension of a polynomial in $\CCCC_d$ is strictly
less than one unless its Julia set is an interval
\cite{MS}. Bruin and Schleicher established refined
dimension estimates in the quadratic case \cite{BS}.

Tiozzo proved the continuity of quadratic core entropy
in the external angle \cite{Ti}. Dudko and Schleicher
independently established continuity on the Mandelbrot
set \cite{DS}. In an appendix, Jung identified this
entropy with $(\log 2)\dim_H B(f)$ {using the combinatorics
of the Mandelbrot set}, thereby
establishing the continuity of the biaccessibility
dimension \cite{DS}.

In higher degrees, Gao and Tiozzo established the
continuity of combinatorial core entropy on the space
of critical portraits and of core entropy on the
postcritically finite locus \cite{GT}.
Theorem~\ref{thm:main} extends the continuity of the
biaccessibility dimension to all polynomials of any
degree.

Other extensions of core entropy use different invariant sets.
Generalized core entropy defined via
maximal dendrite factors need not be continuous, although
its lower envelope is continuous on the Mandelbrot set
\cite[Theorem~III]{LTYY}. Hubbard-forest entropy for
partially postcritically finite polynomials also exhibits
different continuity behavior
\cite[Proposition~1.2 and Theorem~1.6]{G2}.

\subsection{Outline of the proof}\label{sec:proof}
The proof of Theorem \ref{thm:main} is divided into four steps.\vspace{5pt}

\emph{$\bullet$ The entropy-dimension formula (Theorem \ref{thm:equal})}.
This proof relates three descriptions
of the dynamics: paths in $\Gamma_\xi$, returns of postcritical
arcs under the reduced map $\overline g_\sim$, and angular
itineraries under $\sigma_d(\theta)=d\theta\pmod 1$.
Collapsing the finite marked forest yields an arc rule
that assigns a monotone branch and an unlinked angular
class to each edge. Iterated pullbacks then associate
finite paths with subarcs and angular cylinders, while
closed paths yield interval returns.

Two issues require special care: degenerate arcs may
cause spurious exponential growth, and paths escaping
every finite subgraph must also be accounted for.
We show that degenerate arcs contribute no exponential
growth and that escaping paths correspond to a set
of angles of Hausdorff dimension zero.
For recurrent paths, closed-path counts and contraction
of angular inverse branches yield the upper dimension bound.
For the lower bound, closed paths based at a non-degenerate
arc produce horseshoes with entropy arbitrarily close
to $h(\xi)$. Their compact invariant sets lift to
compact $\sigma_d^n$-invariant angle sets without entropy
loss. After removing countably many exceptional angles
the classical entropy--dimension formula for the circle
map then gives the lower bound.\vspace{5pt}

\textit{$\bullet$ From the model to actual rays.}
For an arbitrary polynomial $f$, let
$\mathcal L_{\rm bi}^{\rm p}(f)$ be the perfect part of
the lamination generated by finite smooth landing classes.
It is a proper invariant lamination and thus defines a
laminational model. The discarded leaves have at most
countably many endpoints, while angles in infinite landing
classes form a set of Hausdorff dimension zero.
Since the construction uses actual landing data, it
applies to both connected and disconnected Julia sets.

Periodic classes of the abstract model need not correspond
to actual landing classes. We establish a realization
result for the finite periodic classes used as return
boundaries. This allows suitable interval returns in
the dendrite quotient to be realized as conformal strip
systems bounded by pairs of smooth landing rays of $f$.
Contraction of inverse branches then yields smooth
biaccessible rays throughout the associated non-escaping
sets. In particular, a system with $M$ branches and
return time $m$ produces a compact
$\sigma_d^m$-invariant set $\Theta\subset B(f)$ with
\[
\dim_H\Theta=\frac{\log M}{m\log d}.
\]
Such sets approximate the Hausdorff dimension of $B(f)$.
The realization is local to the chosen returns and
requires no global landing map.\vspace{3pt}

\textit{$\bullet$ Uniform pruning.}
The key quantitative step is to make this realization
uniform. Given $M_0\ge 2$ closed paths of length $m_0$
based at a vertex corresponding to a non-degenerate arc,
we construct, for any $\delta>0$ and finite exceptional
angle set $E$, a conformal strip system satisfying
\[
\frac{\log M}{m}>
\frac{\log M_0}{m_0}-\delta.
\]
Both $m$ and $M$ depend only on
$d,m_0,M_0,\delta,\#E$. The pruning procedure removes
critical-value obstructions and unsuitable ray boundaries
while forcing the boundary angles to lie in the finite set
\[
\sigma_d^{-m}\bigl(\operatorname{Fix}(\sigma_d^{2m})\bigr).
\]
This provides the uniform control needed as $f$ varies.\vspace{3pt}

\emph{$\bullet$ Continuity.} For lower semicontinuity, we choose such a system with
repelling boundary points and perturb it.
For upper semicontinuity, let $f_n\to f$.
Graph-growth continuity  provides a common finite
family of closed paths whose growth approximates the
limiting dimensions. Uniform pruning realizes these paths
for the $f_n$ with a common return time and boundary angles
in a finite set. The exceptional set is chosen using $f$
so that the limiting boundary rays are smooth, land at
repelling points, and admit smooth lifts of the required
orders. Passing to a subsequence fixes the boundary angles.
Convergence and separation of these ray pairs ensure
that their landing points remain distinct, yielding a
non-degenerate return system for $f$ and the desired
upper bound.

The limiting graph serves only to select finite paths;
it need not be a marking graph of $f$. Thus the proof
requires neither a continuous choice of critical markings
nor continuity of entropy on the space of critical portraits.\vspace{5pt}

The paper is organized as follows.
Sections~2 and~\ref{sec:invariant} review preliminaries
on polynomial dynamics and invariant laminations.
Section~\ref{sec:partition} develops critical markings
and reduced dynamics for laminational equivalence relations.
Section~\ref{sec:paths} introduces separation graphs and
establishes new path-counting estimates.
Section~\ref{sec:dimension} proves the combinatorial
entropy--dimension formula (Theorem~\ref{thm:equal}).
Finally, Section~\ref{sec:continuity1} establishes
the analytic realization and proves the continuity
of biaccessibility dimension (Theorem~\ref{thm:main}).

\subsection{Notations}

We denote the complex plane by $\C$ and set
$\mathbb D=\{z\in\C:|z|<1\}$.
We identify $\mathbb S=\mathbb R/\mathbb Z$ with
$\partial\mathbb D$ via $\theta\mapsto e^{2\pi i\theta}$.
The circle length is normalized to one, and we write
$\sigma_d(\theta)=d\theta\pmod 1$.
All logarithms are natural.

For a polynomial or topological polynomial $g$, let
$\operatorname{Crit}(g)$ and
$\operatorname{CritV}(g)=g(\operatorname{Crit}(g))$
denote the sets of its critical points and critical values, respectively.
Its postcritical set is defined by
\begin{equation}\label{eq:-3}
\operatorname{Post}(g)
=\bigcup_{k\ge1}g^k(\operatorname{Crit}(g)),
\end{equation}
without taking the closure.

An \emph{arc} (or \emph{interval}) in $\C$ is a
homeomorphic image of $[0,1]$ or a singleton.
It is \emph{non-degenerate} if it is not a singleton.
By the interior and boundary of an arc, we mean the
arc without its endpoints and the set of its endpoints,
respectively.

\vspace{10pt}

\noindent{\bf Acknowledgement.} The author thanks M. Hassler for helpful discussions. The  author is supported by the  National Key R\&D Program of China (Grant no. 2021YFA1003203), the NSFC (Grant no. 12322104, 12671097).

\section{Polynomial dynamics}
We recall some classical results of polynomial dynamics for use in Section~\ref{sec:continuity1}.

\subsection{External radii and rays}
Let \(f\) be a polynomial of degree \(d\ge 2\). The \emph{basin of infinity} of \(f\) is
\[
\Omega_f(\infty)=\Omega_f=\{z\in\mathbb{C}: f^{\circ n}(z)\to\infty \text{ as } n\to\infty\}.
\]
The boundary of \(\Omega_f\) is the \emph{Julia set} \(J_f\), the complement of \(\Omega_f\) is the \emph{filled Julia set} \(K_f\), and \(\mathbb{C}\setminus J_f\) is the \emph{Fatou set} \(F_f\).

To study the dynamics of \(f\), we use the Green function \(G_f\) and the B\"{o}ttcher map \(\chi_f\) (e.g.\ see \cite[Section 18]{Mil}).
The Green function
\[
G_f:\mathbb{C}\to\mathbb{R}_{\ge0},\qquad
z\mapsto \lim_{n\to\infty}\frac{\log_+|f^{\circ n}(z)|}{d^n}
\]
measures the escape rate of points to \(\infty\). It vanishes at the filled Julia set \(K_f\) and satisfies \(G_f(f(z))=d\cdot G_f(z)\). The number $G_f(z)$ is called the \emph{potential} of $z$.
On $\Omega_f$, the gradient vector field $\nabla_f=\nabla G_f$ vanishes exactly at the escaping critical and precritical points of $f$. 

The Böttcher map \(\chi_f\) conjugates \(f\) to \(z\mapsto z^d\) near \(\infty\). Its germ at \(\infty\) is unique up to \(z\mapsto \zeta z\), where \(\zeta^{d-1}=1\). For monic \(f\), we normalize \(\chi_f\) so that 
\[\chi_f(z)/z\to1 \quad\text{as}\quad z\to\infty\]

Starting near infinity, continue each gradient trajectory towards lower
potential until it first reaches a zero of $\nabla G_f$, or until its
potential tends to zero. The union of these maximal non-singular
trajectories is the \emph{reduced basin} $\Omega_f^*$.
The normalized B\"ottcher map extends conformally to $\Omega_f^*$,
with image a radial slit domain $\Delta_f$ containing a neighbourhood
of infinity; see \cite[Section~2.2]{PZ1}.\vspace{2pt}

For \(\theta\in\R/\Z\), let \((r_\theta,\infty)e^{2\pi i \theta}\) be the maximal portion of \((1,\infty)e^{2\pi i \theta}\) contained in \(\De_f\). The \emph{external radius} \(R^*_f(\theta)\) at argument \(\theta\) is
\[
R^{*}_f(\theta)=\chi_f^{-1}((r_\theta,\infty)e^{2\pi i \theta}).
\]
In the case \(r_\theta>1\), the external radius \(R^*_f(\theta)\) terminates at some point \(z\) where \(\nabla_f(z)=0\). Otherwise, \(r_\theta=1\) and \(R^{*}_f(\theta)\) is a \emph{smooth external ray}. In this case, we simply write $R_f^*(\theta)=R_f(\theta)$, and  call $\theta$ a  \emph{smooth angle}.  

Let $\Sigma_0(f)$ denote the arguments of the external radii that terminate at critical points of \(f\). Since zeros of $\nabla_f$ are exactly the (pre)critical points of $f$ in $\Omega_f$, the external rays are well-defined for all arguments in 
\[
\mathbb S_f:=\mathbb S\setminus\Sigma(f), \quad\text{where}\ \Sigma(f):= \bigcup_{n=0}^{\infty}\sigma_d^{-n}(\Sigma_0(f)).
\]

A smooth external ray \emph{lands} if the intersection of its closure with the Julia set is a singleton. The following is a classic result, see e.g.  \cite[Theorem~2.1(i)]{PZ1}.

\begin{lemma}\label{lem:PZ1}
If $\theta\in\mathbb S_f$ is periodic, then $R_f(\theta)$ lands at a repelling or parabolic periodic point of $f$. 
\end{lemma}

\begin{lemma}\label{lem:stable}
Let \(\theta\) be a preperiodic angle. If \(R_f(\theta)\) is smooth and lands at a pre-repelling point \(z\) whose orbit avoids the critical points of \(f\), then every \(g\) sufficiently close to \(f\) in \(\PPPP_d\) has a smooth ray \(R_g(\theta)\) landing at a pre-repelling point \(z_g\), such that
$
z_g\to z$ and $ \overline{R_g(\theta)}\to \overline{R_f(\theta)}
$
as \(g\to f\).
\end{lemma}
This is also  a well-known result, see for instance {\cite[Proposition~8.5]{DH1}}.\vspace{2pt}

Note that smooth external rays do not necessarily admit smooth lifts. The following simple observation gives a condition guaranteeing a smooth lift.

For every $N\geq1$, we define a set
\begin{equation}\label{eq:-2}
\Sigma_N(f):=\bigcup_{k=1}^N\sigma_d^k(\Sigma_0(f)).
\end{equation}

\begin{lemma}\label{lem:lift}
Fix any integer $N\geq1$. The following two statements hold.
\begin{enumerate}
\item If  \(\theta\notin \Sigma(f)\cup \Sigma_N(f)\), then all lifts of \(R_f(\theta)\) of order at most \(N\) are smooth.
\item Let $\Theta$ be a finite set of angles such that $\sigma_d^N(\Theta)=\Theta$ and \(\Theta\cap\Sigma_N(f)=\emptyset\).  Then $R_f(\theta)$  is smooth for every $\theta\in\Theta$.
\end{enumerate} 
\end{lemma}
\begin{proof}
(1) Otherwise, let \(k\leq N\) be the minimal integer such that some \(\alpha\in\sigma_d^{-k}(\theta)\) is not smooth. Then \(\alpha\in\Sigma_0(f)\), so \(\theta\in\Sigma_k(f)\subset\Sigma_N(f)\), a contradiction. \vspace{2pt}

(2)
If  an angle $\theta\in\Theta$ lies in $\Sigma(f)$, there would be
an integer  $k\ge0$ with
$\sigma_d^k(\theta)=\eta\in\Sigma_0(f)$. Choose $1\le j\le N$ such that
 $k+j$ is a multiple of $N$. Then
$\theta':=\sigma_d^j(\eta)=\sigma_d^{k+j}(\theta)\in \Sigma_{N}(f)\cap \Theta$, a contradiction.
\end{proof}

\subsection{Impressions of smooth angles}
 For an arbitrary angle
$\eta\in \mathbb S_f$ and $t>0$,  we write
$R_f(\eta,t)$ for the point on $R_f(\eta)$ with potential $t$.
Fix an angle \(\theta\in \mathbb{S}_f\) and define its \emph{impression} as
\[{\rm Imp}_f(\theta)=\bigcap_{\de>0}\overline{\{R_f(\eta,t):\eta\in\mathbb S_f,\,|\eta-\theta|<\de,\,0<t<\de\}}.\]

By a \emph{continuum}, we mean a connected and compact set in $\C$.

\begin{lemma}\label{lem:star-impressions}
The following statements hold.
\begin{enumerate}
\item For any $\theta\in\mathbb{S}_f$, its impression is a non-empty continuum in $J_f$, and it holds that $f({\rm Imp}_f(\theta))\subset {\rm Imp}_f(\sigma_d(\theta))$.
\item The impressions $\{{\rm Imp}_f(\theta),\theta\in\mathbb{S}_f\}$ form an upper semicontinuous family, i.e.,
if $\theta_k\to \theta,\,z_k\in{\rm Imp}_f(\theta_k)$ and $z_k\to z$, then $z\in{\rm Imp}_f(\theta)$. 
\end{enumerate} 
\end{lemma}

\begin{proof}

 For an arbitrary angle
$\eta\in\mathbb S$, let $t_\eta\geq 0$ be its stopping potential.
Thus $\eta\in \mathbb S_f$ if and only if $t_\eta=0$. We also write
$R_f^*(\eta,t)$ for the point on the maximal smooth field line at
angle $\eta$, whenever $t>t_\eta$. It is known that the function $\eta\mapsto t_\eta$ is upper semicontinuous, and
$(\eta,t)\mapsto R_f^*(\eta,t)$ is continuous on
$\{(\eta,t):t>t_\eta\}$ (see e.g. \cite[Section 2.2]{PZ1}).

Fix $\theta\in \mathbb S_f$. Choose decreasing sequences
$\varepsilon_j,\delta_j\to 0^+$,  such that
\[
|\eta-\theta|<\delta_j
\quad\Longrightarrow\quad
t_\eta<\frac{\varepsilon_j}{2}.
\]
This is possible because $t_\theta=0$ and $\eta\mapsto t_\eta$ is semicontinuous. Consider the sets
\[
V_j=
\left\{
R_f^*(\eta,t):
|\eta-\theta|<\delta_j,\
t_\eta<t<\varepsilon_j
\right\}.
\]
Each $V_j$ is connected, since  every point can be joined along
its field line to potential $3\varepsilon_j/4$, and the points at
this potential form an arc parametrized by
$|\eta-\theta|<\delta_j$.

Since smooth angles are dense and the ray parametrization is
continuous at positive potential, we have
\[
\overline{V_j}
=
\overline{
\left\{
R_f(\eta,t):
\eta\in \mathbb S_f,\
|\eta-\theta|<\delta_j,\
0<t<\varepsilon_j
\right\}
}.
\]
These closures are nonempty, compact, connected and nested.
 Moreover,
\[
{\rm Imp}_f(\theta)=\bigcap_{j\geq 1}\overline{V_j}.
\]
Thus ${\rm Imp}_f(\theta)$ is a nonempty continuum in $J_f$

Note that smooth rays map to smooth rays, and
$
f(R_f(\eta,t))=R_f(\sigma_d(\eta),dt)
$
for any $\eta\in \mathbb S_f,t>0$.
Hence
$f({\rm Imp}_f(\theta))
\subset {\rm Imp}_f(\sigma_d(\theta)).
$

To show the upper semicontinuity, suppose that
$\theta_k,\theta\in\mathbb S_f$, $\theta_k\to\theta$, and
$z_k\in{\rm Imp}_f(\theta_k)$ with $z_k\to z$.
By the definition of impression, we can choose $\eta_k\in \mathbb S_f$
and $0<t_k<1/k$ such that
\[
|\eta_k-\theta_k|<\frac1k,
\qquad
|R_f(\eta_k,t_k)-z_k|<\frac1k.
\]
Then $\eta_k\to\theta$, $t_k\to0$, and
$R_f(\eta_k,t_k)\to z$. Hence $z\in{\rm Imp}_f(\theta)$.
\end{proof}

\section{Invariant laminations}\label{sec:invariant}

In this section, we introduce basic definitions and collect some known results on (invariant) geodesic laminations that will be used below.

\subsection{Invariant geodesic laminations}

A \emph{chord} is a closed segment in $\overline{\D}$ connecting two points of $\partial \D$.
Let $\mathcal{R}$ be a collection of chords. We call $$\mathcal{R}^\cup:=(\bigcup_{\ell\in\RRR} \ell)\bigcup \mathbb{S}$$ the \emph{solid} of $\mathcal{R}$.
We are interested in collections of chords with specific properties.

\begin{definition}[Geodesic laminations]
A \emph{(geodesic) prelamination} is a collection \(\mathcal{L}\) of chords, called \emph{leaves}, together with all points of \(\mathbb{S}\), viewed as \emph{degenerate leaves}, such that the leaves are pairwise disjoint in \(\mathbb{D}\). If \(\mathcal{L}^{\cup}\) is closed, we call \(\mathcal{L}\) a \emph{(geodesic) lamination}. It is \emph{trivial} if \(\mathcal{L}^{\cup}=\mathbb{S}\). \emph{Gaps} of \(\mathcal{L}\) are the closures of the connected components of \(\mathbb{D}\setminus\mathcal{L}^{\cup}\).
\end{definition}

We endow the space of chords, including degenerate ones, with the Hausdorff metric by regarding them as compact subsets of \(\overline{\mathbb{D}}\). In the following, all limits of leaves are taken in this compact metric space.

Let \(\mathcal{L}\) be a lamination. We extend \(\sigma_d\) over the leaves of \(\mathcal{L}\) so that the restriction of the extension to each leaf is affine. This extension is continuous on \(\mathcal{L}^{\cup}\) and well-defined. We use the same notation for \(\sigma_d\) and all its extensions.

Let us introduce the notion of a (sibling) $\sigma_{d}$-invariant lamination given in \cite{BMOV}, which is a slight modification of the notion of an invariant lamination introduced by Thurston \cite{ThLam}.

\begin{definition}[Invariant laminations]\label{def:invariant-lamination}
A  lamination $\mathcal{L}$ is  $(\sigma_{d})$-invariant if:
\begin{enumerate}
\item for each $\ell \in \mathcal{L}$ we have $\sigma_{d}(\ell) \in \mathcal{L}$;
\item for each $\ell \in \mathcal{L}$ there exists $\ell^{*} \in \mathcal{L}$ so that $\sigma_{d}(\ell^{*}) = \ell$;
\item for each $\ell \in \mathcal{L}$ such that $\sigma_{d}(\ell)$ is not a singleton, there exist $d$ pairwise disjoint leaves $\ell_{1},\ldots,\ell_{d}$ in $\mathcal{L}$ such that $\ell_{1} = \ell$ and $\sigma_{d}(\ell_{i}) = \sigma_{d}(\ell)$ for all $i=2,\ldots,d$.
\end{enumerate}
\end{definition}

Note that the pairwise disjoint leaves in part (3) of the above definition cannot intersect even at an endpoint on the unit circle.

We call the leaf \(\ell^*\) in (2) a \emph{pullback} of \(\ell\), and the leaves \(\ell_2,\ldots,\ell_d\) in (3) \emph{sibling leaves}, or simply \emph{siblings}, of \(\ell=\ell_1\). In general, a sibling of \(\ell\) is a leaf with the same image as \(\ell\) but distinct from it.

In \cite{BMOV}, the authors proved that every invariant lamination is gap invariant. For any compact set \(A\subset \mathbb{S}\), we use \(\operatorname{CH}(A)\) to denote the convex hull of \(A\) in \(\overline{\mathbb{D}}\).
A continuous map is \emph{monotone} if every point-preimage is
connected.

\begin{proposition}
Let $\LLL$ be an invariant lamination.
 If $G$ is a gap of $\mathcal{L}$ and $H=\sigma_d(G):=\mathrm{CH}(\sigma_{d}(G\cap \mathbb{S}))$, then $H$ is a point, a leaf of $\mathcal{L}$ or a gap of $\mathcal{L}$. In the latter case, the map $\sigma_{d}|_{\partial G} : \partial G \to \partial H$  is a positively oriented composition of a monotone map and a covering map.
\end{proposition}

A gap $G$ is said to be \emph{infinite} (\emph{finite}, \emph{uncountable}) if $G\cap \mathbb{S}$ is infinite (finite, uncountable). Uncountable gaps are also called \emph{Fatou gaps}.  Straight segments from $\partial G$ are called \emph{edges} of $G$, and the points in $G\cap\mathbb S$  are  \emph{vertices} of $G$.

\begin{theorem}[\cite{Kiwi1},\,Theorem 1.1]\label{thm:kiwi}
Let \(A\subset \mathbb S\) be a finite set. Suppose that for any \(m \neq n\) we have
\[
\operatorname{CH}(\sigma_d^m(A)) \cap \operatorname{CH}(\sigma_d^n(A)) = \emptyset.
\]
Then $\# A\leq 2^d$. As a consequence, every infinite gap of $\LLL$ is preperiodic.
\end{theorem}

Let $G$ be a gap or a leaf. If $\sigma_{d}(G)$ is a singleton, then the \emph{degree} of $G$ is the cardinality of $G \cap \mathbb{S}$. Otherwise, the \emph{degree} of $G$ equals the number of components in the preimage of a point $z \in \sigma_{d}(\partial G)$ under the map $\sigma_{d}|_{\partial G}$. 

Leaves or gaps of degree greater than \(1\) are called \emph{critical leaves} or \emph{critical gaps} of \(\LLL\), respectively; collectively, they are called \emph{critical sets}.

A periodic Fatou gap $G$ is \emph{hyperbolic} or \emph{Siegel} according to ${\rm deg}(\sigma_d^n|_G)>1$ or ${\rm deg}(\sigma_d^n|_G)=1$, where $n$ is the period of $G$.

\subsection{Laminational equivalence relations and topological polynomials}\label{sec:laminations}

A \emph{laminational equivalence relation} is an equivalence relation $\sim$ in $\mathbb{S}$ such that:
\begin{itemize}
\item [(R1)] The relation $\sim$ is closed in $\mathbb{S}\times \mathbb{S}$;
\item [(R2)] Each $\sim$-class is a finite subset of $\mathbb{S}$;
\item [(R3)] $\sim$-classes have pairwise disjoint convex hulls.
\item [(R4)] If $A$ is  a $\sim$-class, then $\sigma_d(A)$ is a $\sim$-equivalence class.
\item [(R5)] If $A$ is a $\sim$-class, then $\sigma_d:A\to\sigma_d(A)$ is \emph{consecutive preserving}, i.e., for any component $(a,b)$ of $\mathbb S\setminus A$,  $(\sigma_d(a),\sigma_d(b))$ is a component of $\mathbb S\setminus \sigma_d(A)$.
\end{itemize}

The collection $\mathcal{L}_{\sim}$ of chords of $\overline{\mathbb{D}}$ that are edges of convex hulls of $\sim$-classes is called the \emph{(geodesic) lamination generated by $\sim$}.

\begin{lemma}[Theorem 3.21,\,\cite{BMOV}]
Laminations $\mathcal{L}_{\sim}$ generated by a laminational equivalence {relation} $\sim$ are (sibling) invariant.
\end{lemma}

Suppose that $\sim$ is a laminational equivalence relation. Write
$
J_\sim={\mathbb S}/\!\sim
$
and let
$
\pi_\sim:{\mathbb S}\to J_\sim
$
be the quotient map. The induced map
$
f_\sim:J_\sim\to J_\sim
$
satisfying
\[
f_\sim\circ \pi_\sim=\pi_\sim\circ \sigma_d
\]
is well defined. We call  \(f_\sim\) the  \emph{topological polynomial} associated to $\sim$, and $J_\sim$ the \emph{Julia set} of $f_\sim$.

To realize this quotient in the plane, we extend \(\pi_\sim\) to \(\mathbb{C}\) by collapsing the convex hull of each \(\sim\)-equivalence class to a point and leaving all other points as singletons. Moore's theorem identifies the resulting quotient with the plane. In this realization, \(f_\sim\) extends to a branched covering of \(\mathbb{C}\) of degree \(d\) as follows.

On $\C\setminus\overline{\D}$, $\sigma_d$ denotes the radial extension
$\sigma_d(re^{2\pi i\theta})=r^d e^{2\pi i d\theta}$ for $r\ge1$.
The unbounded component
\[
\Omega_{\sim}(\infty):=\pi_\sim(\mathbb{C}\setminus\overline{\mathbb{D}})
\]
of \(\mathbb{C}\setminus J_\sim\) is the \emph{basin of infinity}, and we have
\begin{equation}\label{eq:11}
\pi_\sim\circ \sigma_d=f_\sim\circ \pi_\sim
\quad\text{on }\mathbb{C}\setminus\mathbb{D}.
\end{equation}
The image under \(\pi_\sim\) of \(\{re^{2\pi i\theta}:r>1\}\) is called the \emph{external ray} of \(f_\sim\) with argument \(\theta\), denoted by \(R_\sim(\theta)\). It follows from \eqref{eq:11} that
\[
f_\sim(R_\sim(\theta))=R_\sim(\sigma_d(\theta)).
\]
Since \(J_\sim\) is connected and locally connected, every external ray lands on \(J_\sim\).

Every bounded component $U$ of $\C\setminus J_\sim$ is called a \emph{Fatou domain}, which corresponds to a unique Fatou gap $G_U$ of $\LLL_\sim$ such that $\pi_\sim(G_U)=\overline{U}$. So every Fatou domain of $f_\sim$ is (pre)periodic by Theorem \ref{thm:kiwi}. A periodic Fatou domain $U$ has
 a return map of degree greater than one (\emph{attracting domain}) or a degree-one rotation
type return map (\emph{Siegel disk}).

We choose extensions of $f_\sim$ over the bounded Fatou
domains, together with compatible disk coordinates, so
that maps between successive domains preserve radial
segments. In these coordinates, the first-return maps
on periodic attracting and Siegel domains are power maps
and irrational rotations, respectively. The preimages
of {radial segments} in $\D$ under these coordinates are
 \emph{internal rays}.

For two points $x,y\in K_\sim:=\C\setminus \Omega_\sim(\infty)$, there exists a unique arc in $K_\sim$ joining $x,y$ such that its intersection with any Fatou domain consists of subarcs of at most two internal rays. We call this arc the \emph{regulated arc} joining $x,y$ and denote it by $[x,y]_{\sim}$. Generally, given finitely many points $x_1,\ldots, x_m\in K_\sim$, its regulated convex hull $[x_1,\ldots,x_m]_\sim$ is defined as the union of $[x_i,x_j]_\sim$ for all $i,j\in\{1,\ldots,m\}$.

\begin{lemma}\label{lem:regulated}
The following statements hold
\begin{enumerate}
\item The regulated hull $[x_1,\ldots,x_m]_\sim$ is a finite tree, called a \emph{regulated tree}, with endpoints contained in $\{x_1,\ldots,x_m\}$.
\item The intersection of two regulated trees is either empty or  a regulated tree, possibly a singleton.
\item The image of $[x_1,\ldots,x_m]_\sim$ by $f_\sim$ is still a regulated tree, with endpoints contained in $\{f_\sim(x_1),\ldots,f_\sim(x_m)\}\cup {\rm CritV}(f_\sim)$.
\item Every component of the preimage of a regulated tree is a regulated tree.
\end{enumerate}
\end{lemma}
This result is well-known  and the proof can be found in \cite[Section 2]{Poi2}. 

\begin{definition}\label{def:Hubbard}
The  topological polynomial \(f_\sim\) is called \emph{postcritically finite} if every critical point of \(f_\sim\) has finite forward orbit. Its \emph{Hubbard tree} $H_{f_\sim}$ is defined as the regulated convex hull of the forward orbits of the critical points of \(f_\sim\).
\end{definition}

By Lemma \ref{lem:regulated}, we obtain $f(H_{f_\sim})\subset H_{f_\sim}$. 

\subsection{Proper lamination and perfect lamination}\label{sec:proper}
Given an invariant  lamination $\mathcal{L}$, define an equivalence relation $\sim_{\mathcal{L}}$ by declaring that $x \sim_{\mathcal{L}} z$ if and only if there exists a finite concatenation of leaves of $\mathcal{L}$ connecting $x$ and $z$.
Note that $\sim_{\mathcal{L}}$ is not necessarily a laminational equivalence relation because some $\sim_\LLL$-class is perhaps an infinite set. The following result {characterizes} $\LLL$ such that $\sim_\LLL$ is a laminational equivalence.

\begin{definition}[Proper lamination]\label{def:proper}
An invariant lamination \(\mathcal{L}\) is called \emph{proper} if no critical leaf has a periodic endpoint and no two leaves sharing a periodic endpoint have the same image.
 \end{definition}

\begin{theorem}[Theorem 4.9,\,\cite{BMOV}]
Let $\mathcal{L}$ be an  invariant lamination. Then $\sim_{\mathcal{L}}=\sim$ is a   laminational equivalence relation if and only if $\LLL$ is proper.
\end{theorem}

By this theorem, every leaf in \((\LLL\setminus\LLL_\sim)\cup(\LLL_\sim\setminus\LLL)\) lies in the convex hull of a finite \(\sim\)-class with at least three angles. There are at most countably many such classes, since their convex hulls have pairwise disjoint nonempty interiors, each containing a point of \(\mathbb{Q}^2\). Therefore,

\begin{corollary}\label{cor:differ}
Let \(\LLL\) be a proper invariant lamination. Then \(\LLL\) and \(\LLL_\sim\) differ in only countably many leaves, where \(\sim:=\sim_\LLL\).
\end{corollary}

We recall the notion of a perfect lamination and some
of its basic properties. A lamination is
called \emph{perfect} if it has no isolated leaves.
In particular, the trivial lamination is perfect.

Furthermore, for any lamination $\LLL$, its \emph{perfect part}
$\LLL^{\rm p}$ is the maximal perfect subset of $\LLL$,
viewed as a closed subset of the compact metrizable
space of chords. Since the points of $\mathbb S$
form a perfect subset of $\LLL$, they all belong to
$\LLL^{\rm p}$. The Cantor--Bendixson theorem
immediately yields the following.

\begin{proposition}\label{pro:countable}
Every leaf of $\LLL^{\rm p}$ can be approximated by
distinct leaves of $\LLL$, and
$\LLL\setminus\LLL^{\rm p}$ is countable.
\end{proposition}

Although defined purely topologically, the perfect
part of an invariant lamination is itself invariant,
as shown in the following result.

\begin{lemma}[\cite{BOPT2}, Lemma 3.12]\label{lem:perfect}
The perfect part $\LLL^{\rm p}$ of any invariant
lamination $\LLL$ is a perfect invariant lamination.
\end{lemma}

Moreover, the perfect part enjoys the following
geometric properties.

\begin{lemma}\label{lem:proper}
For any invariant lamination $\LLL$, its perfect
part $\LLL^{\rm p}$ is proper, its gaps are pairwise
disjoint, and no gap meets a leaf other than one
of its edges.
\end{lemma}
{Indeed, properness follows from \cite[Lemma~4.2]{BOTW}.
Let $\sim$ be the relation induced by $\LLL^{\rm p}$.
Corollary~\ref{cor:differ} shows that $\LLL^{\rm p}$ and $\LLL_\sim$
differ by at most countably many leaves, so
$\LLL^{\rm p}=(\LLL_\sim)^{\rm p}$.
The remaining assertions now follow from
\cite[Lemma~3.14]{BOPT2}, applied to the perfect part of $\LLL_\sim$.}\vspace{2pt}

By Lemma \ref{lem:proper}, the perfect part of an invariant lamination induces a topological polynomial whose Fatou domains are all attracting and have pairwise disjoint closures. Collapsing these closures to points yields a planar \emph{dendrite}, i.e., a locally connected continuum contained in $\C$ in which any two points are joined by a unique arc. This construction will be used in Section~\ref{sec:continuity1}.

\section{Critical markings and the induced partitions}\label{sec:partition}

 In this section, we fix a  laminational equivalence relation $\sim$. Let \(\LLL_\sim\) be the lamination generated by \(\sim\) and \(f=f_\sim\) the associated topological polynomial. Denote \(p_f=\pi_\sim|_{\mathbb S}\) the circle projection.

\subsection{Critical markings of topological polynomials}\label{sec:marking}
Recall that a critical portrait of degree $d$ is a collection
\[\xi=\{Q_1,\ldots,Q_s\}\]
of chords and ideal polygons in $\ov{\D}$ fulfilling the following conditions:
\begin{enumerate}
\item any two distinct elements $Q_k$ and $Q_l$  are disjoint in $\D$;
\item the edges of each $Q_k$ are all critical chords;
\item \(\mathbb{D}\setminus\bigcup_{i=1}^s Q_i\) has \(d\) components, each meeting \(\mathbb{S}\) in total length \(1/d\).
\end{enumerate}
\vspace{3pt}

Following \cite{Poi,Kiwi2,Z}, we  associate critical portraits to the topological polynomial \(f\), or equivalently, to the  lamination \(\LLL_\sim\), as follows.
\vspace{5pt}

For a Fatou gap \(G\), let \(\Omega_G\) be its associated Fatou domain. At a point \(z\in\partial\Omega_G\), the external rays landing at \(z\) bound a sector containing \(\Omega_G\). The first such ray in the positive cyclic order is called the \emph{left-supporting ray} of \(\Omega_G\) at \(z\) (if only one ray lands at \(z\), it is the left-supporting ray). Its angle is a vertex of \(G\), called a \emph{left-supporting vertex} of \(G\). Let \(V_L(G)\) denote the set of all left-supporting vertices of \(G\).
Every edge of \(G\) has exactly one endpoint in \(V_L(G)\). Moreover, since \(\sigma_d\) is orientation-preserving on \(\partial G\), we have
\[
\sigma_d(V_L(G))=V_L(\sigma_d(G)).
\]

(1) \emph{Attracting marking.}
\vspace{5pt}

Let $\GGG_A$ be the collection of critical attracting gaps and their iterated images.
This is a finite set. For each gap $G\in\GGG_A$, mark a left-supporting vertex $\beta_G$ of $G$ such that $\sigma_d(\beta_G)=\beta_{\sigma_d(G)}$. Then $\beta_G$ has the same preperiod and period as those of $G$. For any critical gap $G\in\GGG_A$ of degree $k>1$, its marking $\Theta_G$ is defined as 
\[
\Theta_G
=
V_L(G)\cap
\sigma_d^{-1}\bigl(\beta_{\sigma_d(G)}\bigr).
\]
Then  $Q_G:={\rm CH}(\Theta_G)$ is a critical $k$-gon contained in $G$. Define
\[
\mathcal F_A=\{Q_G:\,G\in\GGG_A,\ \deg(G)>1\}.
\]

(2) \emph{Siegel marking}
\vspace{3pt}

Let \(\GGG_S\) be the finite set of critical Siegel gaps, all strictly preperiodic. For each \(G\in\GGG_S\), we mark a vertex $\alpha_G$ of $G$ such that
\begin{itemize}
\item if $\sigma_d^j(G)\in\GGG_S$ for some $j\geq1$, then $\sigma_d^j(\alpha_G)=\alpha_{\sigma_d^j(G)}$;
\item the angles in the $\sigma_d$-grand orbit of $\alpha_G$  avoid the edges of every Fatou gap.
\end{itemize}
This excludes only countably many choices, so compatible choices can still be made backwards along the finite collection of critical Siegel gaps.\vspace{2pt}

For any gap $G\in\GGG_S$
of degree \(k>1\), its marking $\Theta_G$ is defined as 
\[
\Theta_G
=
\partial G\cap
\sigma_d^{-1}\bigl(\sigma_d(\alpha_{G})\bigr).
\]
Then  $Q_G:={\rm CH}(\Theta_G)$ is a critical $k$-gon contained in $G$. Define
\[
\mathcal F_S=\{Q_G:\,G\in \GGG_S\}.
\]

(3) \emph{Julia marking.}
\vspace{3pt}

For every finite critical set $G$ of $\mathcal L_\sim$, including
critical leaves, let $k>1$ be its degree. Choose a \(k\)-element set \(\Theta_G\) of vertices of \(G\) with \(\sigma_d(\Theta_G)\) a singleton, and set \(Q_G:=\operatorname{CH}(\Theta_G)\). Define
\[
\mathcal J
=
\{Q_G: G\text{ is a finite critical set of $\LLL_\sim$}\}.
\]

We call $\FFF=\FFF_A\cup\FFF_S$, $\JJJ$ and $\{\FFF,\,\JJJ\}$ the \emph{Fatou marking}, \emph{Julia marking} and \emph{critical marking} of $f$, respectively.  As a straightforward consequence of this construction, $\xi=\FFF\cup\JJJ$ is a critical portrait of degree $d$, called a \emph{marking portrait} of $f$.  The angles occurring in $\xi$ are called \emph{marking angles}, see Figure \ref{fig:supporting-marking}. 

\begin{figure}[htbp]
\centering
\begin{tikzpicture}[x=1cm,y=1cm,line cap=round,line join=round,font=\small]
  \begin{scope}[shift={(1.8,0)}]
    \node at (0,2.45) {(a) Marking portrait};
    \draw[black,thin] (0,0) circle (1.35);
    \draw[blue!70!black,line width=1.25pt] (0:1.35)--(240:1.35);
    \draw[orange!85!black,line width=1.25pt] (120:1.35)--(240:1.35);
    \foreach \t in {0,120,240} \fill[black] (\t:1.35) circle(1.5pt);
    \node[right] at (0:1.35) {$0$};
    \node[above left] at (120:1.35) {$1/3$};
    \node[below left] at (240:1.35) {$2/3$};
    \draw[blue!70!black,line width=1.25pt] (-1.2,-1.85)--(-0.75,-1.85);
    \node[anchor=west,text=red] at (-0.6,-1.85) {$\{0,2/3\}\in\mathcal F$};
    \draw[orange!85!black,line width=1.25pt] (-1.2,-2.25)--(-0.75,-2.25);
    \node[anchor=west] at (-0.6,-2.25) {$\{1/3,2/3\}\in\mathcal J$};
  \end{scope}
  \node at (6.9,2.45) {(b) Julia set and marked rays};
  \begin{scope}[shift={(7.3,0)},scale=1.2]
\path[fill=gray!12,draw=black,line width=0.22pt] (-1.04875,-0.07662) -- (-1.04695,-0.07310) -- (-1.05327,-0.07046) -- cycle;
\path[fill=gray!12,draw=black,line width=0.22pt] (0.50760,0.20168) -- (0.51122,0.18935) -- (0.51845,0.18935) -- (0.51303,0.18583) -- (0.51303,0.17350) -- (0.52026,0.16821) -- (0.51664,0.16117) -- (0.52387,0.14532) -- (0.52658,0.14267) -- (0.53382,0.14972) -- (0.54828,0.15324) -- (0.55822,0.14355) -- (0.54828,0.13387) -- (0.53743,0.13387) -- (0.52839,0.14091) -- (0.52207,0.13475) -- (0.51664,0.11361) -- (0.51664,0.09952) -- (0.52207,0.09600) -- (0.51484,0.09424) -- (0.51303,0.06958) -- (0.51574,0.06165) -- (0.52297,0.06517) -- (0.52930,0.05901) -- (0.51484,0.05901) -- (0.50760,0.03963) -- (0.50580,0.01145) -- (0.50037,0.00440) -- (0.51122,-0.05196) -- (0.51574,-0.05989) -- (0.52116,-0.05637) -- (0.52930,-0.05901) -- (0.52297,-0.06517) -- (0.51574,-0.06165) -- (0.51303,-0.06958) -- (0.51484,-0.09424) -- (0.52207,-0.09600) -- (0.51664,-0.09952) -- (0.52207,-0.13475) -- (0.52839,-0.14091) -- (0.53743,-0.13387) -- (0.54828,-0.13387) -- (0.55822,-0.14355) -- (0.54828,-0.15324) -- (0.54105,-0.15324) -- (0.52658,-0.14267) -- (0.51664,-0.16117) -- (0.52026,-0.16821) -- (0.51303,-0.17350) -- (0.51303,-0.18583) -- (0.51845,-0.18935) -- (0.51212,-0.18847) -- (0.50760,-0.19992) -- (0.51122,-0.20697) -- (0.50760,-0.20873) -- (0.50941,-0.23163) -- (0.50218,-0.24572) -- (0.50580,-0.25276) -- (0.50037,-0.25452) -- (0.50037,-0.27214) -- (0.50760,-0.27566) -- (0.50128,-0.27478) -- (0.49857,-0.27918) -- (0.49676,-0.29680) -- (0.50218,-0.30032) -- (0.49676,-0.30384) -- (0.50037,-0.31441) -- (0.49676,-0.32146) -- (0.50128,-0.33995) -- (0.53020,-0.32762) -- (0.55189,-0.33467) -- (0.56726,-0.35140) -- (0.56545,-0.36197) -- (0.57539,-0.36637) -- (0.57629,-0.37254) -- (0.56635,-0.36813) -- (0.54466,-0.38046) -- (0.54285,-0.38751) -- (0.53924,-0.38046) -- (0.53382,-0.38223) -- (0.51303,-0.36902) -- (0.49947,-0.34171) -- (0.48320,-0.35228) -- (0.45518,-0.38135) -- (0.45338,-0.38839) -- (0.45880,-0.39367) -- (0.44705,-0.39103) -- (0.42988,-0.41129) -- (0.42445,-0.42890) -- (0.39734,-0.45180) -- (0.37023,-0.49584) -- (0.36480,-0.51873) -- (0.33679,-0.53195) -- (0.29611,-0.56805) -- (0.28165,-0.58567) -- (0.27804,-0.59800) -- (0.25183,-0.61297) -- (0.23285,-0.63147) -- (0.22742,-0.64380) -- (0.20573,-0.65965) -- (0.18223,-0.68783) -- (0.18585,-0.69311) -- (0.17952,-0.69223) -- (0.17320,-0.69840) -- (0.17320,-0.70368) -- (0.15241,-0.72746) -- (0.11987,-0.71865) -- (0.10541,-0.72218) -- (0.10541,-0.71689) -- (0.09276,-0.71337) -- (0.08553,-0.71337) -- (0.08010,-0.72218) -- (0.07649,-0.71161) -- (0.05480,-0.70809) -- (0.04214,-0.70985) -- (0.04305,-0.71954) -- (0.03672,-0.72218) -- (0.03401,-0.71777) -- (0.03762,-0.71073) -- (0.03130,-0.70633) -- (0.00780,-0.70280) -- (0.00418,-0.70985) -- (0.00238,-0.70280) -- (-0.01751,-0.70280) -- (-0.03107,-0.70721) -- (-0.02745,-0.72306) -- (-0.03558,-0.73098) -- (-0.04191,-0.71777) -- (-0.03287,-0.70721) -- (-0.04643,-0.69928) -- (-0.06451,-0.69576) -- (-0.06631,-0.70104) -- (-0.06993,-0.69400) -- (-0.08620,-0.69223) -- (-0.09433,-0.69488) -- (-0.09252,-0.70192) -- (-0.09704,-0.70456) -- (-0.09885,-0.69400) -- (-0.12235,-0.69047) -- (-0.16393,-0.69752) -- (-0.16664,-0.70016) -- (-0.15941,-0.71073) -- (-0.15760,-0.72482) -- (-0.16393,-0.73627) -- (-0.17658,-0.73979) -- (-0.18833,-0.73010) -- (-0.18833,-0.71777) -- (-0.16844,-0.69840) -- (-0.19465,-0.67990) -- (-0.20189,-0.67814) -- (-0.20550,-0.68343) -- (-0.20460,-0.67726) -- (-0.21092,-0.67286) -- (-0.22719,-0.66934) -- (-0.23442,-0.67990) -- (-0.23713,-0.67550) -- (-0.23352,-0.67022) -- (-0.25069,-0.66053) -- (-0.29407,-0.66053) -- (-0.31757,-0.64644) -- (-0.32661,-0.64996) -- (-0.32571,-0.64556) -- (-0.33565,-0.64115) -- (-0.35372,-0.64115) -- (-0.35734,-0.65348) -- (-0.36186,-0.65084) -- (-0.35824,-0.64203) -- (-0.37722,-0.63411) -- (-0.38988,-0.63411) -- (-0.39168,-0.64115) -- (-0.39259,-0.63499) -- (-0.39892,-0.63235) -- (-0.41157,-0.63235) -- (-0.41338,-0.63763) -- (-0.41518,-0.63235) -- (-0.44953,-0.63587) -- (-0.46489,-0.61033) -- (-0.49653,-0.57422) -- (-0.51099,-0.56365) -- (-0.52003,-0.57246) -- (-0.52454,-0.56982) -- (-0.51731,-0.56101) -- (-0.51912,-0.55396) -- (-0.53629,-0.53723) -- (-0.54352,-0.54075) -- (-0.54081,-0.53635) -- (-0.54714,-0.52842) -- (-0.57245,-0.51785) -- (-0.57968,-0.53723) -- (-0.58962,-0.53283) -- (-0.58962,-0.52402) -- (-0.57606,-0.51785) -- (-0.57516,-0.51169) -- (-0.59414,-0.48791) -- (-0.60046,-0.49055) -- (-0.59866,-0.48351) -- (-0.60679,-0.47558) -- (-0.62487,-0.46677) -- (-0.64294,-0.44740) -- (-0.65198,-0.44388) -- (-0.65560,-0.44916) -- (-0.65469,-0.44123) -- (-0.66283,-0.43507) -- (-0.70169,-0.42186) -- (-0.71434,-0.38487) -- (-0.72429,-0.37342) -- (-0.73242,-0.37606) -- (-0.72519,-0.37078) -- (-0.73242,-0.35669) -- (-0.75230,-0.33907) -- (-0.75773,-0.32322) -- (-0.76496,-0.32146) -- (-0.76496,-0.31265) -- (-0.77309,-0.30649) -- (-0.78032,-0.31001) -- (-0.77761,-0.30032) -- (-0.78936,-0.28887) -- (-0.79298,-0.29239) -- (-0.79659,-0.28183) -- (-0.81738,-0.27214) -- (-0.83184,-0.23163) -- (-0.84449,-0.21930) -- (-0.84992,-0.20344) -- (-0.86076,-0.19992) -- (-0.85895,-0.19287) -- (-0.88245,-0.17526) -- (-0.88426,-0.16293) -- (-0.90053,-0.13475) -- (-0.91951,-0.11625) -- (-0.92674,-0.12506) -- (-0.93307,-0.12066) -- (-0.92222,-0.11361) -- (-0.92764,-0.09776) -- (-0.99904,-0.00176) -- (-1.05056,-0.06605) -- (-1.07587,-0.09071) -- (-1.08129,-0.10304) -- (-1.10479,-0.11890) -- (-1.12919,-0.15148) -- (-1.15450,-0.16381) -- (-1.18975,-0.19640) -- (-1.19969,-0.21137) -- (-1.21596,-0.21313) -- (-1.22319,-0.22018) -- (-1.22680,-0.21665) -- (-1.24850,-0.22722) -- (-1.26476,-0.24131) -- (-1.30272,-0.25893) -- (-1.32532,-0.27918) -- (-1.30544,-0.29680) -- (-1.30363,-0.30913) -- (-1.31628,-0.31793) -- (-1.31538,-0.32410) -- (-1.33255,-0.31265) -- (-1.33436,-0.29680) -- (-1.32713,-0.27918) -- (-1.34249,-0.27478) -- (-1.36057,-0.27478) -- (-1.36238,-0.28006) -- (-1.36599,-0.27478) -- (-1.37503,-0.27830) -- (-1.37684,-0.27478) -- (-1.38678,-0.27918) -- (-1.38768,-0.28887) -- (-1.38949,-0.27830) -- (-1.40937,-0.27830) -- (-1.41118,-0.28359) -- (-1.42022,-0.28006) -- (-1.43558,-0.28799) -- (-1.43016,-0.29680) -- (-1.43468,-0.30296) -- (-1.44281,-0.29680) -- (-1.43739,-0.28799) -- (-1.44914,-0.28359) -- (-1.48168,-0.28535) -- (-1.48529,-0.29592) -- (-1.48620,-0.28799) -- (-1.49252,-0.28535) -- (-1.50879,-0.28711) -- (-1.51060,-0.29239) -- (-1.51964,-0.28711) -- (-1.54946,-0.29504) -- (-1.53862,-0.31265) -- (-1.53862,-0.33379) -- (-1.53500,-0.33731) -- (-1.54223,-0.34083) -- (-1.54675,-0.35228) -- (-1.55941,-0.35933) -- (-1.57387,-0.35933) -- (-1.58742,-0.33731) -- (-1.59465,-0.33731) -- (-1.58742,-0.33379) -- (-1.58381,-0.32146) -- (-1.58742,-0.31793) -- (-1.55218,-0.29592) -- (-1.55308,-0.28975) -- (-1.58652,-0.25717) -- (-1.61725,-0.23427) -- (-1.62990,-0.23074) -- (-1.65159,-0.20608) -- (-1.67871,-0.18671) -- (-1.69679,-0.18143) -- (-1.69769,-0.19287) -- (-1.70582,-0.20256) -- (-1.71757,-0.19640) -- (-1.71757,-0.18759) -- (-1.70944,-0.17966) -- (-1.69769,-0.17878) -- (-1.69950,-0.16821) -- (-1.72480,-0.12594) -- (-1.72932,-0.12682) -- (-1.73023,-0.11890) -- (-1.74469,-0.10833) -- (-1.75372,-0.08191) -- (-1.80615,-0.00088) -- (-1.77361,0.04844) -- (-1.76638,0.06781) -- (-1.75372,0.08191) -- (-1.74469,0.10833) -- (-1.73023,0.11890) -- (-1.72932,0.12682) -- (-1.72480,0.12594) -- (-1.69950,0.16821) -- (-1.69769,0.17878) -- (-1.70944,0.17966) -- (-1.71757,0.18759) -- (-1.71757,0.19640) -- (-1.70582,0.20256) -- (-1.69769,0.19287) -- (-1.69679,0.18143) -- (-1.67871,0.18671) -- (-1.65159,0.20608) -- (-1.62990,0.23074) -- (-1.61725,0.23427) -- (-1.58652,0.25717) -- (-1.55308,0.28975) -- (-1.55218,0.29592) -- (-1.58742,0.31793) -- (-1.58381,0.31970) -- (-1.58742,0.33379) -- (-1.59465,0.33731) -- (-1.58742,0.33731) -- (-1.57387,0.35933) -- (-1.55941,0.35933) -- (-1.54675,0.35228) -- (-1.54223,0.34083) -- (-1.53500,0.33731) -- (-1.53862,0.33379) -- (-1.53862,0.31265) -- (-1.54946,0.29504) -- (-1.51964,0.28711) -- (-1.51060,0.29239) -- (-1.50879,0.28711) -- (-1.49252,0.28535) -- (-1.48620,0.28799) -- (-1.48529,0.29592) -- (-1.48168,0.28535) -- (-1.44914,0.28359) -- (-1.43739,0.28799) -- (-1.44281,0.29680) -- (-1.43468,0.30296) -- (-1.43016,0.29680) -- (-1.43558,0.28799) -- (-1.42022,0.28006) -- (-1.41118,0.28359) -- (-1.40937,0.27830) -- (-1.38949,0.27830) -- (-1.38768,0.28887) -- (-1.38678,0.27918) -- (-1.37684,0.27478) -- (-1.37503,0.27830) -- (-1.36599,0.27478) -- (-1.36418,0.28006) -- (-1.36057,0.27478) -- (-1.34249,0.27478) -- (-1.32713,0.27918) -- (-1.33436,0.29680) -- (-1.33255,0.31265) -- (-1.31538,0.32410) -- (-1.31628,0.31793) -- (-1.30363,0.30913) -- (-1.30544,0.29680) -- (-1.32532,0.27918) -- (-1.30272,0.25893) -- (-1.26476,0.24131) -- (-1.24850,0.22722) -- (-1.22680,0.21665) -- (-1.22319,0.22018) -- (-1.22138,0.21489) -- (-1.19969,0.21137) -- (-1.18975,0.19640) -- (-1.15450,0.16381) -- (-1.12919,0.15148) -- (-1.10479,0.11890) -- (-1.08129,0.10304) -- (-1.07587,0.09071) -- (-1.05056,0.06605) -- (-0.99904,0.00176) -- (-0.92764,0.09776) -- (-0.92222,0.11361) -- (-0.93307,0.12066) -- (-0.92674,0.12506) -- (-0.91951,0.11625) -- (-0.90053,0.13475) -- (-0.88426,0.16293) -- (-0.88245,0.17526) -- (-0.85895,0.19287) -- (-0.86076,0.19992) -- (-0.84992,0.20344) -- (-0.84449,0.21930) -- (-0.83184,0.23163) -- (-0.81738,0.27214) -- (-0.79659,0.28183) -- (-0.79207,0.28623) -- (-0.79388,0.29151) -- (-0.78936,0.28887) -- (-0.77761,0.30032) -- (-0.78032,0.31001) -- (-0.77309,0.30649) -- (-0.76496,0.31265) -- (-0.76496,0.32146) -- (-0.75773,0.32322) -- (-0.75230,0.33907) -- (-0.73242,0.35669) -- (-0.72519,0.36902) -- (-0.73242,0.37606) -- (-0.72429,0.37342) -- (-0.71434,0.38487) -- (-0.70169,0.42186) -- (-0.66283,0.43507) -- (-0.65469,0.44123) -- (-0.65560,0.44916) -- (-0.65198,0.44388) -- (-0.64294,0.44740) -- (-0.62487,0.46677) -- (-0.59866,0.48351) -- (-0.59956,0.49143) -- (-0.59414,0.48791) -- (-0.58600,0.49584) -- (-0.57516,0.51697) -- (-0.58962,0.52402) -- (-0.58691,0.53547) -- (-0.57877,0.53635) -- (-0.57245,0.51785) -- (-0.54714,0.52842) -- (-0.54081,0.53635) -- (-0.54352,0.54075) -- (-0.53629,0.53723) -- (-0.51912,0.55396) -- (-0.51731,0.56101) -- (-0.52454,0.56982) -- (-0.52003,0.57246) -- (-0.51099,0.56365) -- (-0.49653,0.57422) -- (-0.46489,0.61033) -- (-0.44953,0.63587) -- (-0.41518,0.63235) -- (-0.41338,0.63763) -- (-0.41157,0.63235) -- (-0.39892,0.63235) -- (-0.39259,0.63499) -- (-0.39168,0.64115) -- (-0.38988,0.63411) -- (-0.37722,0.63411) -- (-0.36096,0.63939) -- (-0.35824,0.64380) -- (-0.36186,0.65084) -- (-0.35734,0.65348) -- (-0.35372,0.64115) -- (-0.33565,0.64115) -- (-0.32571,0.64556) -- (-0.32661,0.64996) -- (-0.31757,0.64644) -- (-0.29407,0.66053) -- (-0.25069,0.66053) -- (-0.23352,0.67022) -- (-0.23713,0.67550) -- (-0.23442,0.67990) -- (-0.23171,0.67198) -- (-0.22538,0.66934) -- (-0.20731,0.67462) -- (-0.20550,0.68343) -- (-0.20189,0.67814) -- (-0.19104,0.68167) -- (-0.16844,0.69840) -- (-0.18833,0.71777) -- (-0.18833,0.73010) -- (-0.17658,0.73979) -- (-0.16393,0.73627) -- (-0.15760,0.72482) -- (-0.15941,0.71073) -- (-0.16664,0.70016) -- (-0.16393,0.69752) -- (-0.12235,0.69047) -- (-0.09885,0.69400) -- (-0.09704,0.70456) -- (-0.09252,0.70192) -- (-0.09433,0.69488) -- (-0.08620,0.69223) -- (-0.06993,0.69400) -- (-0.06631,0.70104) -- (-0.06451,0.69576) -- (-0.04643,0.69928) -- (-0.03287,0.70721) -- (-0.04191,0.71777) -- (-0.03558,0.73098) -- (-0.02745,0.72306) -- (-0.03107,0.70721) -- (-0.01751,0.70280) -- (0.00238,0.70280) -- (0.00418,0.70985) -- (0.00780,0.70280) -- (0.03130,0.70633) -- (0.03762,0.71073) -- (0.03401,0.71777) -- (0.03672,0.72218) -- (0.04305,0.71954) -- (0.04214,0.70985) -- (0.05480,0.70809) -- (0.07649,0.71161) -- (0.08010,0.72218) -- (0.08553,0.71337) -- (0.09276,0.71337) -- (0.10541,0.71689) -- (0.10541,0.72218) -- (0.11987,0.71865) -- (0.15241,0.72746) -- (0.17320,0.70368) -- (0.17320,0.69840) -- (0.17952,0.69223) -- (0.18585,0.69311) -- (0.18223,0.68783) -- (0.20573,0.65965) -- (0.22742,0.64380) -- (0.23285,0.63147) -- (0.25183,0.61297) -- (0.27804,0.59800) -- (0.28165,0.58567) -- (0.29611,0.56805) -- (0.33679,0.53195) -- (0.36480,0.51873) -- (0.37023,0.49584) -- (0.39734,0.45180) -- (0.42445,0.42890) -- (0.42988,0.41129) -- (0.44705,0.39103) -- (0.45880,0.39367) -- (0.45338,0.38839) -- (0.45518,0.38135) -- (0.48320,0.35228) -- (0.49947,0.34171) -- (0.51303,0.36902) -- (0.53382,0.38223) -- (0.53924,0.38046) -- (0.54285,0.38751) -- (0.54466,0.38046) -- (0.56635,0.36813) -- (0.57629,0.37254) -- (0.57629,0.36725) -- (0.56545,0.36197) -- (0.56726,0.35140) -- (0.55189,0.33467) -- (0.53020,0.32762) -- (0.50309,0.33995) -- (0.49857,0.33555) -- (0.49676,0.31617) -- (0.50037,0.31441) -- (0.49676,0.30384) -- (0.50218,0.29856) -- (0.49676,0.29680) -- (0.49676,0.28799) -- (0.50037,0.27566) -- (0.50760,0.27566) -- (0.50037,0.27214) -- (0.50037,0.25452) -- (0.50580,0.25276) -- (0.50218,0.24572) -- (0.50941,0.23163) -- (0.50760,0.20873) -- (0.51122,0.20697) -- cycle;
\path[fill=gray!12,draw=black,line width=0.22pt] (-0.52093,-0.69135) -- (-0.51822,-0.69576) -- (-0.51189,-0.69311) -- (-0.50647,-0.70897) -- (-0.49382,-0.72482) -- (-0.50105,-0.73187) -- (-0.49833,-0.73979) -- (-0.49201,-0.73891) -- (-0.48930,-0.72570) -- (-0.46399,-0.72394) -- (-0.45495,-0.71865) -- (-0.44139,-0.69840) -- (-0.43959,-0.67902) -- (-0.43416,-0.67726) -- (-0.43959,-0.67374) -- (-0.44501,-0.64380) -- (-0.45134,-0.63587) -- (-0.48026,-0.64291) -- (-0.50466,-0.66669) -- (-0.51189,-0.68959) -- cycle;
\path[fill=gray!12,draw=black,line width=0.22pt] (-1.82422,-0.02026) -- (-1.85405,-0.04580) -- (-1.87755,-0.05460) -- (-1.87935,-0.06341) -- (-1.88297,-0.05460) -- (-1.89924,-0.05460) -- (-1.90105,-0.05989) -- (-1.90466,-0.05460) -- (-1.92274,-0.05460) -- (-1.94262,-0.03523) -- (-1.95256,-0.03435) -- (-1.94895,-0.02906) -- (-1.96612,-0.00176) -- (-1.97697,-0.00881) -- (-1.98781,-0.00881) -- (-1.99414,0.00088) -- (-1.98781,0.00881) -- (-1.97697,0.00881) -- (-1.96612,0.00176) -- (-1.94895,0.02906) -- (-1.95256,0.03435) -- (-1.94262,0.03523) -- (-1.92274,0.05460) -- (-1.90466,0.05460) -- (-1.90105,0.05989) -- (-1.89924,0.05460) -- (-1.88297,0.05460) -- (-1.87755,0.06341) -- (-1.87755,0.05460) -- (-1.85405,0.04580) -- (-1.81338,0.01145) -- (-1.80795,-0.00264) -- cycle;
\path[fill=gray!12,draw=black,line width=0.22pt] (0.42445,0.43419) -- (0.42536,0.42978) -- (0.43078,0.42978) -- (0.43530,0.43947) -- (0.43078,0.44211) -- cycle;
\path[fill=gray!12,draw=black,line width=0.22pt] (-0.95295,-0.07662) -- (-0.95385,-0.08102) -- (-0.94753,-0.07662) -- cycle;
\path[fill=gray!12,draw=black,line width=0.22pt] (0.12077,0.81641) -- (0.11174,0.80937) -- (0.12077,0.80937) -- (0.13343,0.76709) -- (0.12981,0.76357) -- (0.13524,0.76181) -- (0.14066,0.74420) -- (0.15150,0.73010) -- (0.15783,0.72922) -- (0.20483,0.75036) -- (0.21387,0.74684) -- (0.21116,0.75300) -- (0.23194,0.76621) -- (0.25363,0.79087) -- (0.25906,0.79440) -- (0.26810,0.78911) -- (0.26900,0.79528) -- (0.26087,0.79616) -- (0.25996,0.80056) -- (0.27442,0.84283) -- (0.27894,0.84900) -- (0.29159,0.84548) -- (0.29792,0.85164) -- (0.28617,0.85781) -- (0.27894,0.85076) -- (0.27623,0.85340) -- (0.26358,0.89039) -- (0.23375,0.92826) -- (0.20664,0.92650) -- (0.20483,0.93002) -- (0.20302,0.92474) -- (0.18495,0.92298) -- (0.14427,0.90448) -- (0.13524,0.88335) -- (0.12801,0.88159) -- (0.13162,0.87454) -- (0.12439,0.85869) -- (0.11987,0.85428) -- (0.11445,0.85781) -- (0.10812,0.85516) -- (0.11083,0.85076) -- (0.11806,0.85252) -- (0.12077,0.84636) -- (0.12077,0.83403) -- (0.11716,0.83227) -- (0.12077,0.83050) -- cycle;
\path[fill=gray!12,draw=black,line width=0.22pt] (-1.75553,0.11009) -- (-1.74740,0.10745) -- (-1.74469,0.11185) -- (-1.75101,0.11978) -- cycle;
\path[fill=gray!12,draw=black,line width=0.22pt] (-1.26567,-0.24396) -- (-1.26296,-0.25012) -- (-1.25844,-0.24748) -- cycle;
\path[fill=gray!12,draw=black,line width=0.22pt] (-0.89872,0.18407) -- (-0.89059,0.17614) -- (-0.88245,0.17702) -- (-0.88245,0.18583) -- (-0.88878,0.19023) -- (-0.89601,0.19023) -- cycle;
\path[fill=gray!12,draw=black,line width=0.22pt] (0.13162,-0.91505) -- (0.13975,-0.91946) -- (0.14427,-0.90624) -- (0.13433,-0.90713) -- cycle;
\path[fill=gray!12,draw=black,line width=0.22pt] (-0.74327,-0.46765) -- (-0.73784,-0.46589) -- (-0.73875,-0.46149) -- cycle;
\path[fill=gray!12,draw=black,line width=0.22pt] (0.22923,0.94676) -- (0.23104,0.93267) -- (0.23737,0.93002) -- (0.25183,0.93883) -- (0.25815,0.94852) -- (0.24640,0.96173) -- (0.24098,0.96173) -- (0.22923,0.95380) -- cycle;
\path[fill=gray!12,draw=black,line width=0.22pt] (0.12620,-0.78471) -- (0.12077,-0.80937) -- (0.11174,-0.80937) -- (0.11264,-0.81377) -- (0.12077,-0.81465) -- (0.12077,-0.83050) -- (0.11716,-0.83227) -- (0.12077,-0.83403) -- (0.12077,-0.84636) -- (0.11806,-0.85252) -- (0.11083,-0.85076) -- (0.10812,-0.85516) -- (0.11445,-0.85781) -- (0.11987,-0.85428) -- (0.12439,-0.85869) -- (0.13162,-0.87454) -- (0.12801,-0.88159) -- (0.13524,-0.88335) -- (0.14518,-0.90536) -- (0.17591,-0.91946) -- (0.19579,-0.92650) -- (0.20302,-0.92474) -- (0.20483,-0.93002) -- (0.20664,-0.92650) -- (0.23375,-0.92826) -- (0.26358,-0.89039) -- (0.27623,-0.85340) -- (0.27894,-0.85076) -- (0.28617,-0.85781) -- (0.29792,-0.85164) -- (0.29159,-0.84548) -- (0.27894,-0.84900) -- (0.27442,-0.84283) -- (0.25996,-0.80056) -- (0.27081,-0.79175) -- (0.26448,-0.78911) -- (0.25906,-0.79440) -- (0.25363,-0.79087) -- (0.23194,-0.76621) -- (0.21116,-0.75300) -- (0.21387,-0.74684) -- (0.20483,-0.75036) -- (0.17591,-0.73451) -- (0.15241,-0.72922) -- (0.14066,-0.74420) -- (0.13524,-0.76181) -- (0.12981,-0.76357) -- (0.13343,-0.76709) -- cycle;
\path[fill=gray!12,draw=black,line width=0.22pt] (-1.63985,0.24043) -- (-1.63985,0.23515) -- (-1.63171,0.23074) -- (-1.63081,0.23691) -- cycle;
\path[fill=gray!12,draw=black,line width=0.22pt] (-0.71706,-0.46677) -- (-0.71886,-0.46325) -- (-0.73152,-0.46501) -- (-0.73965,-0.45885) -- (-0.74146,-0.43947) -- (-0.72609,-0.42274) -- (-0.70260,-0.42274) -- (-0.69988,-0.44652) -- cycle;
\path[fill=gray!12,draw=black,line width=0.22pt] (0.31419,0.55572) -- (0.32142,0.55925) -- (0.31690,0.56189) -- cycle;
\path[fill=gray!12,draw=black,line width=0.22pt] (-1.92906,-0.05725) -- (-1.93087,-0.06429) -- (-1.92454,-0.06693) -- (-1.92003,-0.06077) -- (-1.92274,-0.05637) -- cycle;
\path[fill=gray!12,draw=black,line width=0.22pt] (0.51484,-0.22634) -- (0.51122,-0.22986) -- (0.51574,-0.23427) -- (0.52116,-0.23603) -- (0.52568,-0.23163) -- cycle;
\path[fill=gray!12,draw=black,line width=0.22pt] (-0.62939,-0.46765) -- (-0.63300,-0.47294) -- (-0.62668,-0.47558) -- (-0.62396,-0.46942) -- cycle;
\path[fill=gray!12,draw=black,line width=0.22pt] (-1.07858,-0.11273) -- (-1.08039,-0.10392) -- (-1.07225,-0.10833) -- cycle;
\path[fill=gray!12,draw=black,line width=0.22pt] (0.23194,-0.65348) -- (0.22833,-0.64468) -- (0.23646,-0.64732) -- cycle;
\path[fill=gray!12,draw=black,line width=0.22pt] (-0.74146,0.44476) -- (-0.73965,0.43595) -- (-0.72609,0.42274) -- (-0.70169,0.42362) -- (-0.69988,0.44652) -- (-0.71615,0.46237) -- (-0.71525,0.46677) -- (-0.73513,0.46325) -- (-0.74327,0.45180) -- cycle;
\path[fill=gray!12,draw=black,line width=0.22pt] (0.36480,-0.52402) -- (0.36661,-0.53459) -- (0.37836,-0.54780) -- (0.39192,-0.53987) -- (0.39553,-0.52754) -- (0.38378,-0.51962) -- (0.36751,-0.51785) -- cycle;
\path[fill=gray!12,draw=black,line width=0.22pt] (0.13072,0.91065) -- (0.13162,0.91681) -- (0.13975,0.91946) -- (0.14427,0.90624) -- cycle;
\path[fill=gray!12,draw=black,line width=0.22pt] (0.22742,0.64908) -- (0.22833,0.64468) -- (0.23646,0.64732) -- (0.23194,0.65348) -- cycle;
\path[fill=gray!12,draw=black,line width=0.22pt] (-0.44591,0.64291) -- (-0.45134,0.63587) -- (-0.48026,0.64291) -- (-0.50466,0.66669) -- (-0.51189,0.68959) -- (-0.52093,0.69135) -- (-0.51822,0.69576) -- (-0.51189,0.69311) -- (-0.50647,0.70897) -- (-0.49382,0.72482) -- (-0.50105,0.73187) -- (-0.49833,0.73979) -- (-0.49201,0.73891) -- (-0.48930,0.72570) -- (-0.46399,0.72394) -- (-0.45495,0.71865) -- (-0.44139,0.69840) -- (-0.43959,0.67902) -- (-0.43416,0.67726) -- (-0.43959,0.67374) -- cycle;
\path[fill=gray!12,draw=black,line width=0.22pt] (-0.83907,0.27742) -- (-0.83274,0.27126) -- (-0.81738,0.27390) -- (-0.81919,0.28623) -- (-0.82913,0.29944) -- (-0.84449,0.28799) -- (-0.84269,0.28271) -- (-0.84630,0.28094) -- cycle;
\path[fill=gray!12,draw=black,line width=0.22pt] (-1.12648,-0.15060) -- (-1.12648,-0.16117) -- (-1.11654,-0.16381) -- (-1.11202,-0.15765) -- cycle;
\path[fill=gray!12,draw=black,line width=0.22pt] (0.36661,0.51873) -- (0.38017,0.51785) -- (0.39463,0.52666) -- (0.39192,0.53987) -- (0.37836,0.54780) -- (0.36661,0.53459) -- cycle;
\path[fill=gray!12,draw=black,line width=0.22pt] (0.22923,-0.95204) -- (0.23556,-0.95997) -- (0.24640,-0.96173) -- (0.25815,-0.94852) -- (0.25183,-0.93883) -- (0.23375,-0.93002) -- cycle;
\path[fill=gray!12,draw=black,line width=0.22pt] (-1.08129,0.11009) -- (-1.08039,0.10392) -- (-1.07225,0.10833) -- (-1.07496,0.11273) -- cycle;
\path[fill=gray!12,draw=black,line width=0.22pt] (0.31600,-0.56101) -- (0.32142,-0.55925) -- (0.31690,-0.55484) -- cycle;
\path[fill=gray!12,draw=black,line width=0.22pt] (-0.84630,-0.28094) -- (-0.84269,-0.29151) -- (-0.82913,-0.29944) -- (-0.81919,-0.28623) -- (-0.81738,-0.27390) -- (-0.83274,-0.27126) -- cycle;
\path[fill=gray!12,draw=black,line width=0.22pt] (0.42445,-0.43066) -- (0.42717,-0.44035) -- (0.43530,-0.43947) -- (0.43440,-0.43155) -- cycle;
\path[fill=gray!12,draw=black,line width=0.22pt] (0.28165,0.61209) -- (0.27894,0.59888) -- (0.29431,0.60504) -- (0.29159,0.61297) -- cycle;
\path[fill=gray!12,draw=black,line width=0.22pt] (-1.05418,0.07310) -- (-1.04966,0.07046) -- (-1.04695,0.07662) -- cycle;
\path[fill=gray!12,draw=black,line width=0.22pt] (-0.30040,0.66845) -- (-0.29588,0.66229) -- (-0.28955,0.66493) -- (-0.28955,0.67550) -- (-0.29769,0.68167) -- (-0.30221,0.67550) -- cycle;
\path[fill=gray!12,draw=black,line width=0.22pt] (-0.89691,-0.18759) -- (-0.88878,-0.19023) -- (-0.88336,-0.18671) -- (-0.88245,-0.17702) -- (-0.89420,-0.17790) -- (-0.89872,-0.18231) -- cycle;
\path[fill=gray!12,draw=black,line width=0.22pt] (-1.20059,-0.22458) -- (-1.19427,-0.23427) -- (-1.18704,-0.23427) -- (-1.18252,-0.22810) -- (-1.18433,-0.21930) -- (-1.19788,-0.21137) -- cycle;
\path[fill=gray!12,draw=black,line width=0.22pt] (-0.30221,-0.67198) -- (-0.29769,-0.68167) -- (-0.28955,-0.67550) -- (-0.28955,-0.66493) -- (-0.29588,-0.66229) -- cycle;
\path[fill=gray!12,draw=black,line width=0.22pt] (-1.57839,0.36902) -- (-1.58019,0.36373) -- (-1.57387,0.36109) -- (-1.57296,0.36725) -- cycle;
\path[fill=gray!12,draw=black,line width=0.22pt] (0.23285,-0.76533) -- (0.24008,-0.76357) -- (0.23556,-0.76093) -- cycle;
\path[fill=gray!12,draw=black,line width=0.22pt] (-0.74327,0.46413) -- (-0.73875,0.46149) -- (-0.73784,0.46589) -- (-0.74236,0.46853) -- cycle;
\path[fill=gray!12,draw=black,line width=0.22pt] (-1.12648,0.15412) -- (-1.12377,0.14972) -- (-1.11202,0.15765) -- (-1.12196,0.16381) -- cycle;
\path[fill=gray!12,draw=black,line width=0.22pt] (-0.45043,0.72482) -- (-0.45314,0.71865) -- (-0.44862,0.72130) -- cycle;
\path[fill=gray!12,draw=black,line width=0.22pt] (-1.92816,0.05637) -- (-1.93087,0.06429) -- (-1.92454,0.06693) -- (-1.92003,0.06077) -- cycle;
\path[fill=gray!12,draw=black,line width=0.22pt] (0.23465,0.76181) -- (0.24008,0.76357) -- (0.23375,0.76621) -- cycle;
\path[fill=gray!12,draw=black,line width=0.22pt] (-1.20059,0.22282) -- (-1.19788,0.21137) -- (-1.18433,0.21930) -- (-1.18252,0.22810) -- (-1.18704,0.23427) -- (-1.19427,0.23427) -- cycle;
\path[fill=gray!12,draw=black,line width=0.22pt] (-1.75553,-0.11185) -- (-1.75101,-0.11978) -- (-1.74469,-0.11009) -- (-1.75101,-0.10745) -- cycle;
\path[fill=gray!12,draw=black,line width=0.22pt] (-1.63894,-0.24131) -- (-1.63985,-0.23515) -- (-1.63171,-0.23074) -- (-1.63081,-0.23691) -- cycle;
\path[fill=gray!12,draw=black,line width=0.22pt] (-0.84992,0.22282) -- (-0.84449,0.22106) -- (-0.84720,0.22546) -- cycle;
\path[fill=gray!12,draw=black,line width=0.22pt] (-1.26476,0.24307) -- (-1.26115,0.25012) -- (-1.25844,0.24572) -- cycle;
\path[fill=gray!12,draw=black,line width=0.22pt] (-0.95205,0.07574) -- (-0.95385,0.08102) -- (-0.94753,0.07662) -- cycle;
\path[fill=gray!12,draw=black,line width=0.22pt] (-0.62848,0.46677) -- (-0.63300,0.47118) -- (-0.63029,0.47558) -- (-0.62396,0.47118) -- cycle;
\path[fill=gray!12,draw=black,line width=0.22pt] (-1.77270,-0.07222) -- (-1.77361,-0.06781) -- (-1.76909,-0.06693) -- (-1.76819,-0.07134) -- cycle;
\path[fill=gray!12,draw=black,line width=0.22pt] (-1.57748,-0.36990) -- (-1.58019,-0.36373) -- (-1.57387,-0.36109) -- (-1.57296,-0.36725) -- cycle;
\path[fill=gray!12,draw=black,line width=0.22pt] (0.39282,0.46677) -- (0.39463,0.47206) -- (0.39915,0.47118) -- cycle;
\path[fill=gray!12,draw=black,line width=0.22pt] (0.28256,-0.61297) -- (0.27804,-0.59976) -- (0.28617,-0.59888) -- (0.29431,-0.60504) -- (0.29159,-0.61297) -- cycle;
\path[fill=gray!12,draw=black,line width=0.22pt] (0.51574,0.22546) -- (0.51122,0.22986) -- (0.52116,0.23603) -- (0.52568,0.23163) -- cycle;
\path[fill=gray!12,draw=black,line width=0.22pt] (0.19941,0.67286) -- (0.20121,0.67814) -- (0.20573,0.67726) -- cycle;
\path[fill=gray!12,draw=black,line width=0.22pt] (-0.76044,0.33995) -- (-0.76315,0.34612) -- (-0.75682,0.34876) -- (-0.75230,0.34083) -- cycle;
\path[fill=gray!12,draw=black,line width=0.22pt] (-0.76044,-0.34876) -- (-0.76315,-0.34259) -- (-0.75321,-0.33995) -- (-0.75230,-0.34436) -- cycle;
\path[fill=gray!12,draw=black,line width=0.22pt] (-1.77270,0.06693) -- (-1.77361,0.07134) -- (-1.76909,0.07222) -- (-1.76819,0.06781) -- cycle;
\path[fill=gray!12,draw=black,line width=0.22pt] (0.20121,-0.67814) -- (0.19941,-0.67286) -- (0.20573,-0.67726) -- cycle;
    \draw[blue!70!black,line width=1.15pt] (0.5,0)--(1.25,0);
    \draw[orange!85!black,line width=1.15pt,domain=-1.57:-1,samples=70]
      plot (\x,{sqrt(max(3*\x*(\x+1),0))});
    \draw[blue!70!black,line width=1.15pt,domain=-1.57:-1,samples=70]
      plot (\x,{-sqrt(max(3*\x*(\x+1),0))});
    \draw[orange!85!black,densely dashed,line width=0.7pt,domain=-1.57:-1,samples=70]
      plot (\x,{-sqrt(max(3*\x*(\x+1),0))});
    \draw[blue!70!black,line width=1.15pt] (-1,0)--(0.5,0);
    \foreach \t in {-1,0,0.5} \fill[black] (\t,0) circle(1.5pt);
    \node[below right] at (-0.96,-0.03) {$-1$};
    \node[below] at (0,-0.05) {$0$};
    \node[below] at (0.5,-0.05) {$1/2$};
    \node[right,blue!70!black] at (1.25,0) {$R_f(0)$};
    \node[left,text=red] at (-1.57,1.64) {$R_f(1/3)$};
    \node[left,text=red] at (-1.57,-1.64) {$R_f(2/3)$};
    \node at (0.12,1.15) {$J_f$};
  \end{scope}
  \node at (6.9,-2.65) {$f(z)=z^3+\frac32z^2$};
\end{tikzpicture}
\caption{{The marking portrait of $f(z)=z^3+\frac32z^2$ and its realization on a numerical approximation of $J_f$. Blue marks the Fatou ray star with angles $\{0,2/3\}$; orange marks the Julia critical ray pair with angles $\{1/3,2/3\}$ landing at $-1$. The ray $R_f(2/3)$ belongs to both markings and is shown in both colors.}}\label{fig:supporting-marking}\par
\end{figure}
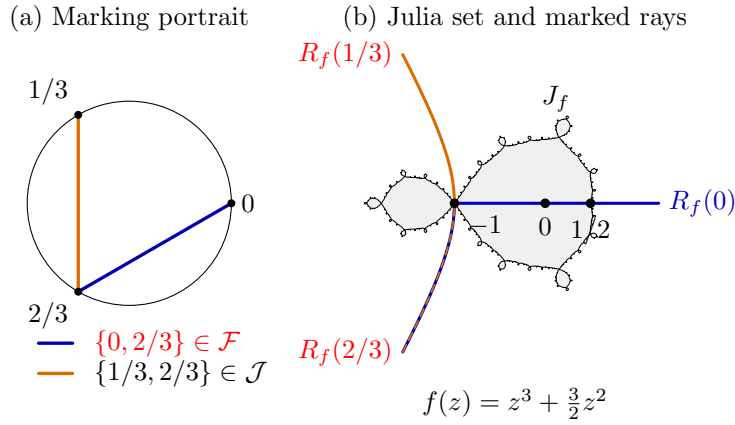

\subsection{Reduction to Siegel non-capture case}

A laminational equivalence relation is called \emph{Siegel non-capture} if its associated lamination has no critical preperiodic gap that eventually maps to a Siegel Fatou gap under $\sigma_d$.

\begin{proposition}\label{pro:reduction}
Let $\sim$ be a laminational equivalence relation, and let
$$\xi=\FFF_A\cup\FFF_S\cup\JJJ$$ be a marking portrait satisfying the condition above.
Then there exists a Siegel non-capture laminational
equivalence relation $\sim_{\rm nc}$ coarser than $\sim$
such that
\begin{equation}\label{eq:-4}
\mathcal L_\sim\subset\mathcal L_{\sim_{\rm nc}}\quad\text{and}\quad
\mathcal L_{\sim_{\rm nc}}\setminus\mathcal L_\sim
\text{ is at most countable}.
\end{equation}
Moreover, the same portrait $\xi$ is a marking portrait
of $\mathcal L_{\sim_{\rm nc}}$, with Fatou marking
$\mathcal F_A$ and Julia marking
$\mathcal J\cup\mathcal F_S$.
\end{proposition}

\begin{proof}
The idea is to insert finitely many leaves in each (pre)critical Siegel gap. The resulting lamination will be the required one.

Let \(G\) be a \emph{last} critical Siegel gap of degree \(k>1\), meaning that \(\sigma^n(G)\) is non-critical for every \(n\ge1\). There is a unique critical \(k\)-gon \(Q_G\in\FFF_S\) that marks \(G\). It partitions \(G\) into \(k\) sub-gaps, each mapped homeomorphically by \(\sigma_d\) onto \(\sigma_d(G)\), while \(\sigma_d(Q_G)\) is a single point on \(\partial\sigma_d(G)\). 

For any gap $G'$ with $\sigma_d^m(G')=G$, let $Q_{G'}$ be the convex hull of the points $$\sigma_d^{-m}(Q_G\cap\mathbb S)\bigcap \partial G'.$$
Since \(Q_G\cap\mathbb{S}\) avoids all leaves on the boundaries of Fatou domains, \(Q_{G'}\) has exactly \(k\cdot\deg(\sigma_d^m|_{G'})\) edges. 

Thus, we obtain a polygon $Q_G$ in each (pre)critical Siegel gap such that these $Q_G$ are pairwise disjoint, and every $Q_G$ is disjoint from the leaves of $\LLL_\sim$.
Let $\EEE_G$ denote the collection of chords on $\partial Q_G$, and define a prelamination
\[\LLL_{\rm nc}=\left(\bigcup_{G}\EEE_G\right)\bigcup \LLL_\sim,\] 
where $G$ ranges over all (pre)critical Siegel gaps. The construction and {the properties} of $Q_G$ {imply}  the following facts:
\begin{enumerate}
\item $\LLL_{\rm nc}$ is closed, and hence a lamination;
\item $\LLL_{\rm nc}$ satisfies all properties of Definition \ref{def:invariant-lamination}, hence invariant.
\item Let \(\sim_{\rm nc}\) be the equivalence relation on \(\mathbb{S}\) induced by \(\LLL_{\rm nc}\). Then each \(\sim_{\rm nc}\)-class is either a \(\sim\)-class, or the vertex set of  \(Q_G\) for some (pre)critical Siegel gap $G$ of $\LLL_\sim$, which is disjoint from all non-trivial \(\sim\)-classes.
\item The relation $\sim_{\rm nc}$ is a laminational equivalence relation and $\LLL_{\sim_{nc}}=\LLL_{\rm nc}$.
\end{enumerate}
Therefore, we obtain the required property  \eqref{eq:-4}. Moreover, the lamination $\LLL_{\rm nc}$ is Siegel non-capture, with the Fatou marking of $\LLL_{\rm nc}$ is $\FFF_A$ and the Julia marking is $\FFF_S\cup\JJJ$. Then the proof is complete.
\end{proof}

In the rest of this section, we always assume that $f$ is Siegel non-capture. This implies that $\FFF_A=\FFF$ in a critical Fatou marking of $f$.

\subsection{Partitions in dynamical plane}
Let \[\xi=\FFF\cup \JJJ=\FFF_A\cup\JJJ=\{Q_1,\ldots,Q_s\}\] be a  marking portrait of \(f\).  
Two points \(x,y\in\mathbb{S}\setminus\bigcup_{i=1}^s Q_i\) are called \emph{\(\xi\)-unlinked on \(\mathbb{S}\)} if they lie in the same component of \(\mathbb{S}\setminus Q_i\) for every \(i=1,\ldots,s\); equivalently, if they lie in the same component of \(\overline{\mathbb{D}}\setminus\cup_{i=1}^s Q_i\). There are \(d\)  \(\xi\)-unlinked classes, and for any $\xi$-class $S$,
\begin{equation}\label{eq:7}
\sigma_d:S\to \mathbb S\setminus\sigma_d(\partial S)\text{ is a homeomorphism}.
\end{equation}

The partition of \(\xi\)-unlinked classes on \(\mathbb{S}\) induces a natural partition of \(\mathbb{C}\) as follows. For \(Q\in\xi\), let \(\Theta(Q)\) denote its vertex set. If \(Q\in\JJJ\), set
\[
\wh{R}_f(Q)=R_f(Q)=\bigcup_{\theta\in\Theta(Q)}\overline{R_f(\theta)}\quad\text{and}\quad C_Q=\wh{R}_f(Q)\cap K_f.
\]
If \(Q\in\FFF\), since $f$ is Siegel non-capture, the landing point \(p_f(\theta)\) of \(R_f(\theta)\) for each \(\theta\in\Theta(Q)\) lies on the boundary of the attracting domain \(\Omega_Q\), where \(\overline{\Omega_Q}=\pi_\sim(G_Q)\). Let \(\wh{R}_f(\theta)\) denote the union of \(R_f(\theta)\) and the internal ray in \(\Omega_Q\) landing at \(p_f(\theta)\). In this case, set
\[
\wh{R}_f(Q)=\bigcup_{\theta\in\Theta(Q)}\wh{R}_f(\theta)
\quad\text{and}\quad
C_Q=\wh{R}_f(Q)\cap\overline{\Omega_Q}=\wh{R}_f(Q)\cap K_f.
\]
We call these $\wh{R}_f(Q)$ \emph{marked ray stars} and $C_Q$ the \emph{core part} of $\wh{R}_f(Q)$.

Two points \(z_1,z_2\in\mathbb{C}\setminus\bigcup_{i=1}^s\wh{R}_f(Q_i)\) are said to be \emph{\(\xi\)-unlinked in the \(f\)-plane} if they lie in the same connected component of \(\mathbb{C}\setminus\wh{R}_f(Q_i)\) for every \(i=1,\ldots,s\).

Considering the circle at infinity, we immediately obtain that two external rays \(R_f(\theta)\) and \(R_f(\eta)\) lie in a common \(\xi\)-unlinked  class in the \(f\)-plane if and only if \(\theta\) and \(\eta\) lie in a common \(\xi\)-unlinked  class on \(\mathbb{S}\). This yields a canonical correspondence between the \(\xi\)-unlinked equivalence classes on \(\mathbb{S}\) and those in the \(f\)-plane.

\begin{lemma}\label{lem:slits}
Fix any unlinked  class $ V$ in the $f$-plane. 
\begin{enumerate}
\item For two distinct points $z_1,z_2\in \ov{ V}\cap K_f$, the regulated arc $[z_1,z_2]\subset \ov{ V}$.
\item The map $f$ sends $ V$ bijectively onto $\C\setminus f(\partial V)$, and 
\begin{equation}\label{eq:8}
f(\partial V\cap K_f)\subset \big(\bigcup_{Q\in\FFF}f(C_Q)\big)\bigcup \big({\rm CritV}(f)\cap J_f\big).
\end{equation}
\end{enumerate}
\end{lemma}

\begin{proof}
(1) For each $i\in\{1,\ldots,s\}$, denote by $Y_i$ the connected component of $\C\setminus\wh{R}_f(Q_i)$ which contains $V$. We obtain that 
$$ V=\bigcap_{i=1}^{s} Y_i \quad\text{and}\quad\ov{ V}=\bigcap_{i=1}^{s} \ov{Y_i}.$$ It follows that {$z_1,z_2$} belong to each $\ov{Y_i}$. Note that $\partial Y_i\cap K_f$ is either a point in $J_f$ or two internal rays of an attracting domain, {so $[z_1,z_2]$} is contained in every $\ov{Y_i}$. Consequently, we have {$[z_1,z_2]\subset \bigcap_{i=1}^{s}\ov{Y_i}=\ov{V}$}.\vspace{3pt}

(2) The image of each marked core $C_Q$ is either a Julia critical value or
one internal ray in an attracting domain. This gives
\textup{(4.3)}. Moreover, $f(\partial V)$ is contained in the union
of finitely many external rays and a finite regulated tree in
$K_f$. Thus every component of
$\mathbb C\setminus f(\partial V)$ is unbounded.

Fix such a component $W$. Since all critical points are on the marked
ray stars, the map $f$ is a local homeomorphism on $V$. The restriction
\[
f:V\cap f^{-1}(W)\longrightarrow W
\]
is proper: a sequence of preimages of a compact subset of $W$
cannot approach $\partial V$, and cannot escape to infinity.
It is therefore a covering map. For a point of $W$ sufficiently
near infinity, \textup{(4.2)} gives exactly one preimage in $V$.
Since $W$ is connected, the covering has degree one. Hence every
point in $\C\setminus f(\partial V)$ has exactly one preimage in $V$.

It remains to show that \(f(V)\cap f(\partial V)=\emptyset\). Suppose \(z\in V\), \(b\in\partial V\), and \(f(z)=f(b)=w\). Choose disjoint neighborhoods \(U_z\subset V\) and \(U_b\) of \(z\) and \(b\), respectively. Since \(f\) is open, \(f(U_z)\) contains a neighborhood of \(w\). Shrinking \(U_b\) if necessary, we may assume \(f(U_b)\subset f(U_z)\). The set \(f(U_b\cap V)\) is nonempty and open. As \(f(\partial V)\) has no interior points, \(f(U_b\cap V)\) contains a point \(w'\notin f(\partial V)\). Then \(w'\) has a preimage in \(U_z\) and another preimage in \(U_b\cap V\), contradicting the uniqueness just proved. Hence \(f(V)\cap f(\partial V)=\emptyset\), and therefore
\[
f:V\longrightarrow\mathbb{C}\setminus f(\partial V)
\]
is a homeomorphism. 
\end{proof}

\subsection{Reduced dynamics}\label{sec:reduce}
To simplify the discussion below, we define the \emph{collapsing forest}
\[T_f(\xi):=\bigcup_{Q\in\FFF}\bigcup_{n\geq0}f^n(C_Q).\]
The construction of $\xi$ yields that $T_f(\xi)$ is an $f$-invariant finite forest, and it consists of closed internal rays in closures of attracting domains.

\begin{lemma}\label{lem:forest}
There exists \(m_0\ge 0\) such that every component of \(f^{m_0}(T_f(\xi))\) meets \(J_f\) in exactly one point.
\end{lemma}
\begin{proof}
Since all internal rays in \(T_f(\xi)\) are (pre)periodic, we can choose \(m_0\) so that every internal ray in \(f^{m_0}(T_f(\xi))\) is periodic. By the defining property of \(\FFF_A\), each periodic attracting domain contains exactly one internal ray from \(f^{m_0}(T_f(\xi))\). Hence every component of \(f^{m_0}(T_f(\xi))\) meets \(J_f\) in exactly one point.
\end{proof}

Collapsing each component of $T_f(\xi)$ to a point, we obtain a quotient map $\phi:\C\to\C$. 
Since $f(T_f(\xi))\subset T_f(\xi)$, the  map $\phi$ descends the topological polynomial $f$ to a \emph{reduced map}
$\overline{f}:\C\to\C$ such that
\begin{equation}\label{eq:12}
\phi \circ f(z)=\overline{f}\circ \phi(z),\ \forall\,z\in\C.
\end{equation} 
Denote
\[K_{\overline{f}}:=\phi(K_f),\ {\rm Crit}(\overline{f}):=\phi({\rm Crit}(f))\text{\ \ and\ \ } q_{\overline{f}}:=\phi\circ p_f.
\]
It follows that $q_{\overline{f}}\circ\sigma_d=\overline{f}\circ q_{\overline{f}}$ on $\mathbb S$.
For any $a,b\in\mathbb{S}$, denote
\begin{equation}\label{eq:arcs}
 I(a,b)=\phi([p_f(a),p_f(b)]_{K_f}).
\end{equation}
By Lemma \ref{lem:regulated}, this is an arc (maybe degenerate) in $K_{\overline{f}}$ joining $q_{\overline{f}}(a)$ and $q_{\overline{f}}(b)$.

\begin{lemma}\label{lem:nonseparated}
If $a,b$ are not separated by $\xi$, then
\[
 \overline{f}:I(a,b)\longrightarrow I(\sig (a),\sig (b))
\]
is a monotone surjection.
\end{lemma}

\begin{proof}
Write
\[
x=p_f(a),\quad y=p_f(b),\quad
\gamma=[x,y]_{K_f}.
\]
We consider the map
\[
H=\phi\circ f|_\gamma
  =\overline f\circ\phi|_\gamma.
\]

We first locate the part of $\gamma$ outside the marked ray stars.
For each $Q\in\xi$, the marked core  $C_Q=\wh{R}_f(Q)\cap K_f$
 is either a Julia critical point or a regulated tree in the closure of a critical Fatou domain.
By Lemma~\ref{lem:regulated}(2),
$\gamma\cap C_Q$ is connected. Hence $\gamma\setminus C_Q$
has at most two components, each lying in a component of
$\mathbb C\setminus\widehat R_f(Q)$.
If both components occur, they contain $x$ and $y$, respectively.
Since $a,b$ are not separated by $Q$, the rays $R_f(a)$ and
$R_f(b)$, and therefore these two components, lie {in} the same
component of $\mathbb C\setminus\widehat R_f(Q)$.
It follows that all points of
\[
\gamma\setminus\bigcup_{Q\in\xi}C_Q
\]
belong to a single $\xi$-unlinked class $V$, whenever this
set is nonempty.

Let
\[
E=\phi(T_f(\xi))\cup \phi\bigl(\operatorname{CritV}(f)\cap J_f\bigr).
\]
Then \(E\) is finite, and \(H(C_Q\cap\gamma)\subset E\) for every \(Q\in\xi\).
We claim that every point outside \(E\) has at most one preimage under \(H\). Indeed, suppose
\[
H(u)=H(v)=w\notin E,\qquad u,v\in\gamma.
\]
Since \(w\notin\phi(T_f(\xi))\), the fiber \(\phi^{-1}(w)\) is a singleton, so \(f(u)=f(v)\). Neither \(u\) nor \(v\) lies in a marked core \(C_Q\), hence both lie in \(V\). Lemma~\ref{lem:slits}(2) now yields \(u=v\).

By Lemma~\ref{lem:regulated}(3), \(f(\gamma)\) is a finite regulated tree. By Lemma~\ref{lem:regulated}(2), its intersection with each component of \(T_f(\xi)\) is connected. Hence
\[
Y:=H(\gamma)=\phi(f(\gamma))
\]
is a finite tree, obtained by collapsing finitely many subtrees of \(f(\gamma)\).

We claim that every fiber of \(H\) is connected. Suppose \(u,t,v\) occur in this order on \(\gamma\), with
\[
H(u)=H(v)=w,\qquad H(t)\ne w.
\]
Choose \(z\notin E\) in the open arc of \(Y\) joining \(w\) to \(H(t)\). Both connected sets \(H([u,t]_\gamma)\) and \(H([t,v]_\gamma)\) contain this arc, so \(z\) has two distinct preimages under \(H\), contradicting the claim. Therefore \(H\) is monotone, and \(Y\) is an arc (possibly degenerate) with endpoints \(\phi(f(x))\) and \(\phi(f(y))\).

The regulated arc \([f(x),f(y)]_{K_f}\) lies in \(f(\gamma)\), and its \(\phi\)-image is a connected subset of \(Y\) containing both endpoints. Hence
\[
Y=\phi([f(x),f(y)]_{K_f})=I(\sigma_d(a),\sigma_d(b)).
\]
Finally, \(\phi(\gamma)=I(a,b)\) and \(H=\overline f\circ\phi|_\gamma\). For every \(w\in Y\),
\[
\bigl(\overline f|_{I(a,b)}\bigr)^{-1}(w)=\phi\bigl(H^{-1}(w)\bigr),
\]
which is connected. Thus the lemma is proved.
\end{proof}

\section{Separation graph and path estimation}
\label{sec:paths}

{We recall the separation graphs of \cite{GT,Ti} and study their
path counts as abstract labeled wedges. The estimates below control
escaping paths and localize closed-path growth at one vertex.
Section~\ref{sec:dimension} gives these paths their geometric meaning
as branches and returns on postcritical arcs.}
\subsection{Separation graph and growth rate}

Fix integers $d\geq2$ and $1\leq s\leq d-1$.  The set of pairs
$$\Xi_s=\left\{ (y_k(i),y_l(j))  \ : \  k,l \in\{1,\ldots,s\}, 1 \leq i \leq j \right\}$$
is called the {\it wedge} (of size $s$), and the pairs $(y_k(i),y_l(j))$ are called \emph{vertices}.
We  define a \emph{labeling} of this wedge  as an assignment to each pair
$(y_k(i),y_l(j))$ of a label, which can be either $\emptyset$ or
$$(\alpha_1,\ldots,\alpha_r)$$
with $r\geq 1$ and $\alpha_1,\ldots,\alpha_r\in\{1,\ldots,s\}$ with the $\alpha_i$'s pairwise distinct.
We call the wedge with such a labeling a {\it labeled wedge} of size $s$ and denote it by $\WWW$.
A vertex of $\WWW$ is  \emph{non-separated} if its label is $\emptyset$, and \emph{separated} otherwise.

For each labeled wedge $\WWW$, we construct a separation graph $\Gamma=\G(\WWW)$: its vertices are the elements of \(\Xi_s\); for each vertex \(v=(y_k(i),y_l(j))\), the edges with source \(v\) are determined as follows:
\begin{enumerate}
\item
if  $v$ is labeled $\emptyset$, then there is only one outgoing edge, namely
$$( y_k(i), y_l(j) ) \to ( y_k(i+1), y_l(j+1) );$$
\item
If $v$ is labeled \((\alpha_1,\ldots,\alpha_r)\neq\emptyset\), then there are exactly \(r+1\leq d\) edges outgoing from \(v\), and the targets of these edges, in order, are
\begin{equation}\label{eq:split}
 \begin{split}
 &\big(y_{\alpha_1}(1),y_k(i+1)\big),\\
 &(y_{\alpha_1}(1),y_{\alpha_2}(1)),\ldots,
   (y_{\alpha_{r-1}}(1),y_{\alpha_r}(1)),\\
 &(y_{\alpha_r}(1),y_l(j+1)).
 \end{split}
\end{equation}
The first and last are called  \emph{outer edges}; the others are \emph{middle edges}.
 We call $i (=\min\{i,j\})$ the \emph{height} of $v$, and $j$ the \emph{width} of $v$.
\end{enumerate}

A length-\(n\) path in $\G$ based at a vertex \(v\) is a sequence
\[
v=v_0 \xrightarrow{e_0} v_1 \xrightarrow{e_1} \cdots \xrightarrow{e_{n-1}} v_n
\]
of edges, where \(v_0,\ldots,v_n\) are vertices. If \(v_n=v_0\), the path is called \emph{closed}. A path is specified by its edge sequence, so parallel edges are counted separately; vertices and edges may repeat. Closed paths with different starting vertices are considered distinct.

We use \(P_n(v)\) and \(C_n(v)\) to denote the number of length-\(n\) paths and closed paths in \(\G\) based at \(v\), respectively, and set
\[
C_n=\sum_{v\in \Xi_s} C_n(v).
\]

\begin{lemma}[\cite{GT},\,Proposition 4.1]\label{lem:confinement}
A closed path of length $n$ uses  vertices of height at most $n$ and
width at most $2n$. In particular, we have
\begin{equation}\label{eq:loop-bound}
 C_n\le4s^2n^2d^n.
\end{equation}
\end{lemma}

The \emph{(logarithmic) growth rate} of  $\G$, or the labeled wedge $\WWW$, is defined as
\[h(\G)=h(\WWW):=\limsup_{n\to\infty}\log\max\left\{1,\sqrt[n]{C_n}\right\}\]
It follows from Lemma \ref{lem:confinement} that $h(\G)\leq\log d$ for all $\G$.\vspace{3pt}

A sequence $(\WWW_n)_{n\geq1}$ of labeled wedges of size $s$ is said to \emph{converge} to $\WWW$ if, for every finite set of vertices $S\subset\Xi_s$, there is an $N$ such that for every $n\geq N$ and every $v\in S$, the label of $v$ in $\WWW_n$ agrees with its label in $\WWW$.

\begin{lemma}[\cite{GT},\,Lemma 4.4]\label{lem:convergence-of-growth-rate}
If a sequence of labeled wedges $(\WWW_n)_{n\geq1}$ converges to $\WWW$, then the growth rate of $\WWW_n$ converges to that of $\WWW$.
\end{lemma}

\subsection{Core entropy and Thurston's algorithm}\label{sec:graph}
We now see how to associate to each critical
portrait a labeled wedge.
Let $$\xi=\{Q_1,\ldots,Q_s\}$$ be a critical portrait.
For $i\ge1$ and $k\in\{1,\ldots,s\}$, set $x_k(i):=\sig^i(Q_k)$.

The portrait $\xi$ induces a labeled wedge $\WWW_\xi$ of size $s$ as follows:  
for any vertex $(y_k(i),y_l(j))$ of $\Xi_s$, 
 it is labeled $\emptyset$ if the angles $x_k(i),x_l(j)$ are not separated by $\xi$;
and labeled $(\alpha_1,\ldots,\alpha_r)$ ($r\geq1$) if 
\begin{enumerate}
\item the chord  $\overline{x_k(i)x_l(j)}$ successively crosses $Q_{\alpha_1},\ldots,Q_{\alpha_r}$; and
\item no other element of $\xi$ separates $x_k(i)$ from $x_l(j)$.
\end{enumerate}

We denote  by $\G_\xi$ the separation graph associated to $\WWW_\xi$. The growth rate  $h(\G_\xi)=h(\WWW_\xi)$  is simply denoted by $h(\xi)$.

\begin{definition}[core entropy]
Let $\xi$ be a critical portrait. Then its \emph{core entropy}   is defined to be  $h(\xi)$.
\end{definition}

Thurston's algorithm computes the core entropy of a postcritically
finite polynomial from a weak critical marking [GT, Theorem~3.4].
Tiozzo developed an infinite-graph formulation in the quadratic
case; Gao and Tiozzo extended this construction to arbitrary
degree. It is well-known that rational critical portraits correspond to postcritically finite polynomials, but the general relation between critical portraits and polynomials remains unclear.

\subsection{Counting closed paths}
We show that the exponential growth of all closed paths is already attained at a single vertex.

Fix an integer $N>0$ and define a subset
\begin{equation}\label{eq:10}
S_N=\{v\in \Xi_s:{\rm width}(v)\le N\}
\end{equation} 
of the wedge $\Xi_s$. Every closed path either avoids $S_N$ or meets it. We first estimate the number of paths of the former kind.

\begin{lemma}\label{lem:escape}
 For $v\notin S_N$, denote $P_n^{\rm out}(v)$ the number of
length-$n$ paths which start from $v$ and whose vertices avoid  $S_N$. Then
\begin{equation}\label{eq:13}
P_n^{\rm out}(v)\leq 2^{1+n/N}.
\end{equation}
\end{lemma}
\begin{proof}
For any vertex $w\in \Xi_s$, the two target indices of the middle edges starting from $w$ are both \(1\). So middle edges all enter \(S_N\), and at most two outer edges remain outside. If the source index is \(i\le j\), both outer edges avoid \(S_N\) if and only if \(i+1>N\), i.e., the height of $w$ is at least \(N\).

The target height of any separating edge is \(1\). Before the next binary choice, the height must increase from \(1\) to at least \(N\), and each step increases it by at most \(1\). Hence two consecutive binary choices are separated by at least \(N\) steps. On a path of length \(n\), the number of binary choices is at most \(1+\lfloor (n-1)/N\rfloor\). Then the estimation \eqref{eq:13} holds.
\end{proof}

By Lemmas \ref{lem:confinement} and \ref{lem:escape}, the number $C_n^{\rm out}$ of length-$n$ closed paths avoiding $S_N$ satisfies
\[
C_n^{\rm out}\leq 4s^2n^2\,2^{1+n/N}.
\]
Consequently,
\begin{equation}\label{eq:14}
\limsup_{n\to\infty}\frac{1}{n}\log\max\{C_n^{\rm out},1\}\leq \frac{\log 2}{N}.
\end{equation}

\begin{lemma}\label{lem:meeting}
There exist an integer $N>0$ and a vertex $v\in S_N$ such that
\begin{equation}\label{eq:15}
h(\G)=\limsup_{n\to\infty}\frac{1}{n}\log \max\{C_n(v),1\}.
\end{equation}

\end{lemma}
\begin{proof}

Set $h:=h(\G)$. If $h=0$, since $C_n(v)\leq C_n$ for every vertex $v\in\Xi_s$, equation \eqref{eq:15} clearly hold for every $v\in S_N$.

Suppose $h>0$, and choose $N$ so that $\log 2/N<h$. Let $C_n^{\rm meet}$ denote the number of length-$n$ closed paths that meet $S_N$. Then $C_n=C_n^{\rm out}+C_n^{\rm meet}$, and \eqref{eq:14} implies
\[
\limsup_{n\to\infty}\frac1n\log\max\{C_n^{\rm meet},1\}=h.
\]
Rotating a closed path that meets $S_N$ to one of its visits to a vertex in $S_N$ gives
\[
C_n^{\rm meet}\leq n\sum_{u\in S_N} C_n(u).
\]
Because $C_n(u)\leq C_n$ and $S_N$ is finite, some $u\in S_N$ satisfies
\[
\limsup_{n\to\infty}\frac1n\log\max\{C_n(u),1\}=h.
\]
This proves \eqref{eq:15}.
\end{proof}

\section{Core entropy and biaccessibility dimension}
\label{sec:dimension}

Fix a laminational equivalence relation $\sim$ and set $\LLL=\LLL_\sim$.  Let $$\xi=\{Q_1,\ldots,Q_s\}$$ be a marking portrait of  $\LLL$. Recall that $B(\LLL)$ is the set of biaccessible angles (i.e., endpoints of non-degenerate leaves) of $\LLL$. The aim of this section is to prove Theorem \ref{thm:equal} restated as follows.
\begin{theorem}\label{thm:equal1}
The core entropy  $h(\xi)$ coincides with $(\log d)\Hd B(\LLL)$.
\end{theorem}

By Proposition \ref{pro:reduction}, it suffices to assume that $\LLL$ is Siegel non-capture. In this case $\xi$ contains no Siegel marking so that $\FFF=\FFF_A$.\vspace{3pt}

Throughout this section, simply write $B=B(\LLL)$, $h=h(\xi)$, and
$f=f_{\sim}$. We always assume that $f$ is Siegel non-capture, and retain the notation of Section~\ref{sec:reduce}:
$\overline{f}$ is the reduced map and $\phi f=\overline{f}\phi$. For brevity, we write
\[
p=p_f,\quad q=q_{\overline f}=\phi\circ p,\quad c_k=q(\Theta(Q_k)),\quad J_{\overline{f}}=\phi(J_f)=\overline{J}.
\]
For convenience, we refer to $c_1,\ldots,c_s$
(not necessarily distinct) as the  critical
points of $\overline f$, even though $\overline f$
may collapse some arcs to points.

For a formal vertex $v=(y_k(i),y_l(j))$ of $\G_\xi$, set
$$I_v=I(x_k(i),x_l(j))=\phi\big([p(x_k(i)),p(x_l(j))]_{K_f}\big),$$ called a \emph{postcritical arc}. Distinct vertices may therefore determine the same arc, and an arc may be degenerate.

\subsection{Graph paths and  dynamics on postcritical arc}
\label{sec:coding}
The proof of Theorem~\ref{thm:equal1} compares three types of orbits: an edge path in the separation graph,
an orbit of $\overline f$ on postcritical arcs, and a $\sigma_d$-itinerary on $\mathbb S$. Here we establish the correspondence between the first two. Proposition~\ref{prop:arc-rule}
is the one-step statement; successive pullbacks give the
correspondence for finite and infinite paths.

The maps from the preceding sections fit into the commuting diagram
\[
\begin{CD}
\Sone @>{\sig}>> \Sone\\
@V{p}VV @VV{p}V\\
J_f @>{f}>> J_f\\
@V{\phi}VV @VV{\phi}V\\
\overline{J} @>{\overline{f}}>>\overline{J}.
\end{CD}
\]
In particular, $\overline{f}\circ q=q\circ \sig$. The graph will record the successive
branches along an orbit after an initial postcritical arc has
been chosen.

\begin{proposition}\label{prop:arc-rule}
Let $a,b\in\mathbb S$ have their separation vector $(\alpha_1,\ldots,\alpha_r)$. Set
\[
b_0=q(a),\quad b_j=c_{\alpha_j}\ (1\le j\le r),
\quad b_{r+1}=q(b).
\]
Choose $\theta_j\in\Theta(Q_{\alpha_j})$ for $1\le j\le r$,
and set $\theta_0=a$, $\theta_{r+1}=b$.
Then the points $b_0,\ldots,b_{r+1}$ occur in this order on $I(a,b)$,
and, for $0\le j\le r$,
\begin{equation}\label{eq:local-arc}
\overline f:[b_j,b_{j+1}]
\longrightarrow I(\sigma_d(\theta_j),\sigma_d(\theta_{j+1}))
\end{equation}
is a monotone surjection. Repeated points and degenerate subarcs
are allowed.
\end{proposition}
\begin{proof}
Let \(\g = [p(a), p(b)]\) be the regulated arc in $K_f$. Each ray-star $\wh{R}_f(Q_{\alpha_j})$ separates the external rays with angles \(a\) and \(b\). Thus \(\g\) meets these stars in the same order as the chord from \(a\) to \(b\) crosses the corresponding critical polygons.

The quotient map \(\phi\) is monotone on \(\g\) and sends its intersection with $\wh{R}_f(Q_{\alpha_j})$ to \(c_{\alpha_j}=b_j\). It preserves this order, although some of the images may coincide. Thus \(b_0, \ldots, b_{r+1}\) occur in this order on \(I(a,b) = \phi(\g)\). The consecutive subarcs \( [b_j, b_{j+1}] \subset I(a,b)\) consequently have pairwise disjoint relative interiors, although some may be points.

Between consecutive entries there is no further separating critical component. So each pair $\{\theta_j,\theta_{j+1}\}$ is non-separated. Then \eqref{eq:local-arc} follows from  Lemma \ref{lem:nonseparated}.
\end{proof}

\smallskip\noindent
\emph{Edges and finite paths.}
Retain the subarc $L_e\subset I_v$ associated by
Proposition~\ref{prop:arc-rule} to each edge $e:v\to w$.
Thus $\overline{f}:L_e\to I_w$ is a monotone surjection, and the outgoing
edges at $v$ give an ordered subdivision of $I_v$. The interior
subdivision points belong to ${\rm Crit}(\overline{f})$. Repeated subdivision points are allowed
and give degenerate subarcs. See Figure~\ref{fig:three-codings} for a local circle--arc--graph correspondence.

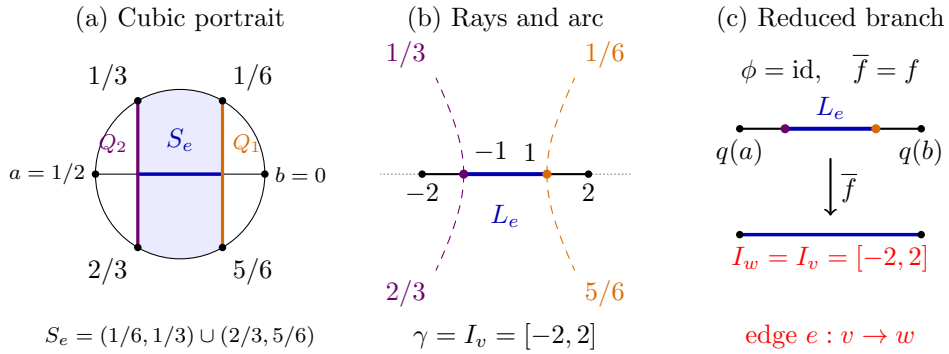
\begin{figure}[htbp]
\centering
\begin{tikzpicture}[x=1cm,y=1cm,line cap=round,line join=round,font=\small]
  \begin{scope}[shift={(1.8,0)}]
    \node at (0,2.05) {(a) Cubic portrait};
    \fill[blue!8] (60:1.12) arc(60:120:1.12) --
      (240:1.12) arc(240:300:1.12) -- cycle;
    \draw[black] (0,0) circle(1.12);
    \draw[black,thin] (-1.12,0)--(1.12,0);
    \draw[blue!70!black,line width=1.3pt] (-0.56,0)--(0.56,0);
    \draw[orange!85!black,line width=1.2pt] (60:1.12)--(300:1.12);
    \draw[violet!85!black,line width=1.2pt] (120:1.12)--(240:1.12);
    \foreach \t in {0,60,120,180,240,300}
      \fill[black] (\t:1.12) circle(1.4pt);
    \node[right,font=\scriptsize] at (0:1.12) {$b=0$};
    \node[left,font=\scriptsize] at (180:1.12) {$a=1/2$};
    \node[above right] at (60:1.12) {$1/6$};
    \node[above left] at (120:1.12) {$1/3$};
    \node[below left] at (240:1.12) {$2/3$};
    \node[below right] at (300:1.12) {$5/6$};
    \node[orange!85!black,right,font=\scriptsize] at (0.56,0.4) {$Q_1$};
    \node[violet!85!black,left,font=\scriptsize] at (-0.56,0.4) {$Q_2$};
    \node[blue!70!black] at (0,0.45) {$S_e$};
    \node[font=\scriptsize] at (0,-2.15)
      {$S_e=(1/6,1/3)\cup(2/3,5/6)$};
  \end{scope}
  \begin{scope}[shift={(6.1,0)}]
    \node at (0,2.05) {(b) Rays and arc};
    \draw[black,thick] (-1.1,0)--(1.1,0);
    \draw[blue!70!black,line width=1.5pt] (-0.55,0)--(0.55,0);
    \draw[orange!85!black,dashed,variable=\r,domain=1:3,samples=40]
      plot ({0.275*(\r+1/\r)},{0.55*sqrt(3)*(\r-1/\r)/2});
    \draw[orange!85!black,dashed,variable=\r,domain=1:3,samples=40]
      plot ({0.275*(\r+1/\r)},{-0.55*sqrt(3)*(\r-1/\r)/2});
    \draw[violet!85!black,dashed,variable=\r,domain=1:3,samples=40]
      plot ({-0.275*(\r+1/\r)},{0.55*sqrt(3)*(\r-1/\r)/2});
    \draw[violet!85!black,dashed,variable=\r,domain=1:3,samples=40]
      plot ({-0.275*(\r+1/\r)},{-0.55*sqrt(3)*(\r-1/\r)/2});
    \draw[gray,densely dotted] (-1.1,0)--(-1.65,0) (1.1,0)--(1.65,0);
    \node[above right,orange!85!black] at (0.917,1.27) {$1/6$};
    \node[above left,violet!85!black] at (-0.917,1.27) {$1/3$};
    \node[below left,violet!85!black] at (-0.917,-1.27) {$2/3$};
    \node[below right,orange!85!black] at (0.917,-1.27) {$5/6$};
    \fill[black] (-1.1,0) circle(1.4pt) (1.1,0) circle(1.4pt);
    \fill[violet!85!black] (-0.55,0) circle(1.7pt);
    \fill[orange!85!black] (0.55,0) circle(1.7pt);
    \node[below] at (-1.1,0) {$-2$};
    \node[below] at (1.1,0) {$2$};
    \node[above right] at (-0.55,0.04) {$-1$};
    \node[above left] at (0.55,0.04) {$1$};
    \node[blue!70!black] at (0,-0.55) {$L_e$};
    \node at (0,-2.15) {$\gamma=I_v=[-2,2]$};
  \end{scope}
  \begin{scope}[shift={(10.4,0)}]
    \node at (0,2.05) {(c) Reduced branch};
    \node at (0,1.4) {$\phi=\mathrm{id},\quad \overline f=f$};
    \draw[black,thick] (-1.2,0.6)--(1.2,0.6);
    \draw[blue!70!black,line width=1.5pt] (-0.6,0.6)--(0.6,0.6);
    \fill[black] (-1.2,0.6) circle(1.4pt) (1.2,0.6) circle(1.4pt);
    \fill[violet!85!black] (-0.6,0.6) circle(1.7pt);
    \fill[orange!85!black] (0.6,0.6) circle(1.7pt);
    \node[above,blue!70!black] at (0,0.6) {$L_e$};
    \node[below] at (-1.2,0.6) {$q(a)$};
    \node[below] at (1.2,0.6) {$q(b)$};
    \draw[->,thick] (0,0.2)--(0,-0.55)
      node[midway,right] {$\overline f$};
    \draw[blue!70!black,line width=1.3pt] (-1.2,-0.8)--(1.2,-0.8);
    \fill[black] (-1.2,-0.8) circle(1.4pt) (1.2,-0.8) circle(1.4pt);
    \node[below,text=red] at (0,-0.8) {$I_w=I_v=[-2,2]$};
    \node[text=red] at (0,-2.15) {edge $e:v\to w$};
  \end{scope}
\end{tikzpicture}
\caption{Circle, arc and graph descriptions for $f(z)=z^3-3z$. The critical leaves $Q_1=\overline{1/6,5/6}$ and $Q_2=\overline{1/3,2/3}$ correspond to the dashed ray pairs landing at $1$ and $-1$. The angles $a,b$ are separated by $Q_1,Q_2$, and the central unlinked class $S_e$ labels the branch $L_e=[-1,1]$. Here $\phi=\mathrm{id}$, and the decreasing homeomorphism $\overline f:L_e\to I_w=[-2,2]$ gives the edge $e:v\to w$, where $v=(y_1(1),y_2(1))$ and $w=(y_2(1),y_1(1))$. These formal vertices are distinct, although $I_w=I_v$. The sector label $S_e$ is made precise for regular angles in Section~\ref{sec:regular}\textcolor{red}{.}}\label{fig:three-codings}\par
\end{figure}

A  length-$n$ path of $\G_\xi$ means a directed edge sequence
\begin{equation}\label{eq:coding-path}
 \omega=(v=v_0\xrightarrow{e_0}v_1\xrightarrow{e_1}\cdots
              \xrightarrow{e_{n-1}}v_n).
\end{equation}
We also write $\omega(j)=v_j$. The vertex $v_1$ is a \emph{successor} of $v$, and $v_j,j\geq 1$ are called {\emph{descendants}} of $v$.
The \emph{full inverse interval} of $\omega$ is
\begin{equation}\label{eq:coding-full-interval}
 A_\omega=\{x\in I_{v_0}:\overline{f}^j(x)\in L_{e_j}
                         \text{ for }0\le j<n\}.
\end{equation}
Equivalently, these intervals are defined recursively by
\begin{equation}\label{eq:coding-extension}
 A_\varnothing=I_{v_0},\qquad
 A_{\omega e}=(\overline{f}^k|_{A_\omega})^{-1}(L_e),\quad k=|\omega|.
\end{equation}
Induction using Proposition \ref{prop:arc-rule} shows that $A_\omega$ is a
nonempty closed subarc, possibly a point, and that
\begin{equation}\label{eq:coding-monotone}
 \overline{f}\,^n:A_\omega\longrightarrow I_{v_n}
 \quad\text{is a monotone surjection.}
\end{equation}
Moreover, the length-$n$ paths
starting at a fixed vertex $v$ give a finite cover of $I_v$.

This cover need not be a partition into intervals with disjoint
interiors. A subdivision point can pull back to a whole plateau (see Figure \ref{fig:full-trimmed-intervals})\textcolor{red}{.}
There is, however, a precise description of the possible overlap:
if $\omega\ne\omega'$ have the same initial vertex $v$ and
the same length $n$, then
\begin{equation}\label{eq:coding-overlap}
 A_\omega\cap A_{\omega'}
 \subset I_v\cap\bigcup_{j=0}^{n-1}\overline{f}^{-j}({\rm Crit}(\overline{f})).
\end{equation}

\smallskip\noindent
\emph{Trimmed intervals and closed paths.}
When the target $I_{v_n}$ is non-degenerate, we also use the
smallest subarc $L_\omega\subset A_\omega$ still mapping onto
$I_{v_n}$, which is  called the \emph{trimmed interval}. For an increasing monotone surjection $T:L\to[a,b]$,
the smallest such inverse interval is
\[
 [\max T^{-1}(a),\min T^{-1}(b)];
\]
for a decreasing map the endpoint order is reversed (see Figure \ref{fig:full-trimmed-intervals}). Making this choice at
each successive pullback, gives
\begin{equation}\label{eq:coding-trimmed}
 \overline{f}^{\,n}:L_\omega\longrightarrow I_{v_n}
 \quad\text{monotone and onto.}
\end{equation}
For distinct paths of the same length from the same vertex,
the trimmed intervals with non-degenerate targets have disjoint
relative interiors.  

\begin{figure}[htbp]
\centering
\begin{tikzpicture}[x=1cm,y=1cm,line cap=round,line join=round,font=\small]
  \draw[->,gray] (-0.6,0)--(6.7,0);
  \draw[->,gray] (-0.6,0)--(-0.6,2.85);
  \draw[gray,densely dashed] (-0.6,0.65)--(1.5,0.65);
  \draw[gray,densely dashed] (-0.6,2.35)--(6.15,2.35);
  \node[left] at (-0.6,0.65) {$a$};
  \node[left] at (-0.6,2.35) {$b$};
  \draw[black,thick] (0,0.65)--(1.5,0.65);
  \draw[blue!70!black,line width=1.35pt] (1.5,0.65)--(4.5,2.35);
  \draw[black,thick] (4.5,2.35)--(6,2.35);
  \node[above left] at (4.5,2.35) {$T=\overline f^{\,n}$};
  \draw[gray,densely dashed] (0,0.65)--(0,-0.7);
  \draw[gray,densely dashed] (6,2.35)--(6,-0.7);
  \draw[blue!70!black,densely dashed] (1.5,0.65)--(1.5,-1.5);
  \draw[blue!70!black,densely dashed] (4.5,2.35)--(4.5,-1.5);
  \draw[black] (0,-0.7)--(6,-0.7);
  \draw[black] (0,-0.6)--(0,-0.8) (6,-0.6)--(6,-0.8);
  \node[below] at (3,-0.7) {full inverse interval $A_\omega$};
  \draw[blue!70!black,line width=1.2pt] (1.5,-1.5)--(4.5,-1.5);
  \draw[blue!70!black] (1.5,-1.4)--(1.5,-1.6) (4.5,-1.4)--(4.5,-1.6);
  \node[below,blue!70!black] at (3,-1.5) {trimmed interval $L_\omega$};
  \node[above,font=\scriptsize] at (1.5,0) {$\max T^{-1}(a)$};
  \node[above,font=\scriptsize] at (4.5,0) {$\min T^{-1}(b)$};
\end{tikzpicture}
\caption{Full and trimmed inverse intervals for an increasing monotone surjection $T:A_\omega\to[a,b]$. The full interval retains the endpoint plateaux. Trimming removes their interiors and gives $L_\omega=[\max T^{-1}(a),\min T^{-1}(b)]$, the smallest closed interval still mapping onto $[a,b]$.}\label{fig:full-trimmed-intervals}\par
\end{figure}
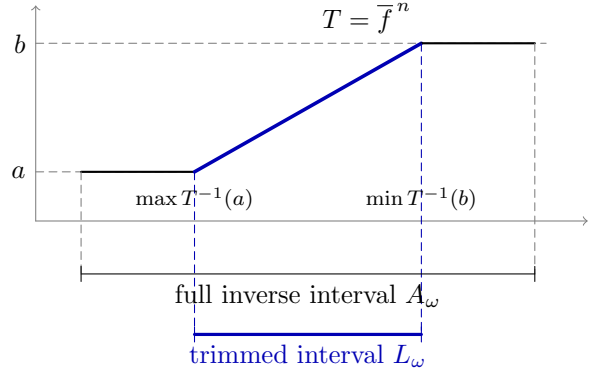

If $\omega$ is closed and $I_v$ is non-degenerate, then
\begin{equation}\label{eq:coding-return}
 L_\omega\subset I_v\text{\ \ and\ \ } \overline{f}^{\,n}(L_\omega)=I_v.
\end{equation}
{This is a \emph{monotone return branch to $I_v$ of time $n$}.
Several distinct closed paths of the same length based at $v$ give
return branches with disjoint relative interiors.

We shall use the following correspondence throughout the rest of
the proof: an edge specifies a monotone arc branch and, as explained
in Section~\ref{sec:regular}, an unlinked angular class; a finite path
specifies a full inverse interval and, for a non-degenerate target,
its trimmed interval; a closed path gives a return to the initial
arc. The associated angular itinerary then determines successive
inverse branches of $\sigma_d$.}

\subsection{Degenerate arcs and horseshoes}
Before extracting return branches from the graph, we must control the
paths whose postcritical arcs collapse to points.

The \emph{zero-arc vertices} of $\G_\xi$ {are} defined as
\begin{equation}\label{eq:1}
Z(f)=Z=\{v:I_v\text{ is a point}\}.
\end{equation}
By
Proposition~\ref{prop:arc-rule}, $Z$ is forward invariant, i.e., if $v\in Z$, then its successors all lie in $Z$. The number of closed paths with vertices in $Z$ is estimated as follows.

\begin{lemma}\label{lem:zero}
There is a constant $K_0$ depending only on $f$, such that
$$P_n(u)\leq K_0,\ \forall n\geq1,\,u\in Z;\quad\text{and}\quad C_n^Z\leq 4s^2K_0n^2,$$
where $C_n^Z$ counts length-$n$ closed paths that meet $Z$.

\end{lemma}
\begin{proof}
Fix a vertex $u\in Z$. It suffices to show that $P_n(u)\leq K_0$.
{Indeed,}
Lemma~\ref{lem:confinement} allows at most $4s^2n^2$ base vertices for
a closed path of length $n$. Summing the bound over these vertices gives
the required inequality.

For a vertex $v$, let $x,y$ be its associated angles and write
\[
\g_v=[p(x),p(y)]_{K_f}.
\]
If $v\in Z$, then either $\g_v$ is a point or $\g_v$ is contained in a
component of the collapsing forest $T_f(\xi)$. Proposition \ref{prop:arc-rule}, before taking the quotient,
shows that
\[
\g_w\subset f(\g_v)
 \quad\text{whenever $w$ is a successor of $v$.}
\]
Indeed,  a separating {marking ray-star} meets $\g_v$, and the image of that
intersection contains the common image of the {landing points} of this ray-star.

Suppose first that $\g_u$ is contained in the forest. By
Lemma~\ref{lem:forest}, after $m_0$ iterates it lies in a forest
component meeting $J_f$ in a single point. The endpoints of every
descendant postcritical arc lie in $J_f$, so every descendant arc at that time is a
point. There are at most $d^{m_0}$ paths of length $m_0$. It remains to
bound the number of continuations from a vertex whose two rays already
land at one point $z$. If $\g_u$ was a point initially, we start directly
with this second part.

Every descendant pair then lands at the corresponding iterate of \(z\). Its chord is therefore either a leaf or contained in a finite gap of \(\LLL_\sim\). Hence the chord can cross neither a Fatou marking polygon, since it lies inside a Fatou gap, nor a Julia marking belonging to a different class. Thus branching can occur only when the orbit of \(z\) meets a Julia critical point. Such a point cannot be periodic. Therefore the orbit of \(z\) visits each such point at most once. Since there are at most \(d-1\) such points and each encounter yields at most \(d\) successors, these encounters contribute at most \(d^{d-1}\) length-\(n\) continuations, uniformly in \(n\).

{Combining these bounds gives}
$P_n(u)\le d^{m_0}d^{d-1}=d^{m_0+d-1}=:K_0$.
\end{proof}

\begin{corollary}\label{coro:many-returns}
If $h>0$, then for every $\epsilon>0$ there are a
vertex $v\not\in Z$, an integer $m\ge1$, and $M\ge2$ closed paths of length $m$
based at $v$, such that
\[
 \frac{\log M}{m}>h-\epsilon.
\]
\end{corollary}
\begin{proof}
By Lemma \ref{lem:meeting}, we can find $v,m$ and $M=C_m(v)$ such that {the required inequality holds}.
 By Lemma \ref{lem:zero}, the number of closed paths meeting $Z$ has polynomial growth. So $v\not\in Z$.
\end{proof}

\begin{definition}[Interval horseshoe]
Let \(I_1,\ldots,I_N\) be non-degenerate subintervals of an interval \(I\) with pairwise disjoint interiors. A continuous map
\[
T:\bigcup_{j=1}^N I_j\to I
\]
that sends each \(I_j\) monotonically (resp.\ homeomorphically) onto \(I\) is called a \emph{monotone (resp.\ homeomorphic) interval (\(N\)-)horseshoe}.
The set 
\[
 X=\bigcap_{k\ge0}T^{-k}\left(\bigcup_{i=1}^N I_i\right)
\]
is   the \emph{non-escaping set} of this horseshoe. It is  compact and $T$-invariant.
\end{definition}

\begin{lemma}[Paths to horseshoe]\label{lem:returns}
Suppose that $v\notin Z$ supports $N$ distinct closed paths of length $n$. Then
$I_v$ contains $N$ non-degenerate subarcs $L_1,\ldots,L_N$ with disjoint relative interiors,
each mapped monotonically onto $I_v$ by $\overline{f}^{\,n}$.
\end{lemma}
\begin{proof}
 A closed path based at $v\notin Z$ never visits $Z$, since $Z$ is
forward invariant. All postcritical arcs along the path are therefore non-degenerate.

Apply the trimmed pullback construction of Section~\ref{sec:coding}
to each of the $N$ paths. Equations~\eqref{eq:coding-trimmed}
and~\eqref{eq:coding-return} give a non-degenerate source interval
$L_\omega\subset I_v$ with $\overline{f}^{\,n}(L_\omega)=I_v$. Distinct paths
give intervals with disjoint relative interiors. These are the required monotone return
branches. 
\end{proof}

\subsection{Regular biaccessible set}\label{sec:regular}
We remove a countable exceptional set from $B$ so that each
remaining angle belongs to a two-point class and its forward orbit
avoids branch points, Fatou boundaries, and marked landing classes.

Let $\Upsilon_0$ be the set of angles $\theta$ such that $p(\theta)$ is one of the following cases:
\begin{itemize}
\item a branched point of $J_f$;
\item a cut point on the boundary of a Fatou domain, or
\item  the landing point of a marking angle for $\xi$ (including ${\rm Crit}(f)\cap J_f$).
\end{itemize}
Define
\[
\Upsilon
=
\bigcup_{m,n\geq0}
\sigma_d^{-m}\bigl(\sigma_d^n(\Upsilon_0)\bigr).
\]
 We call $B_{\rm reg}:=B\setminus \Upsilon$ the \emph{regular biaccessible set} of $\LLL$.
\begin{proposition}\label{pro:exceptional}
The set $B_{\rm reg}$ is  invariant under $\sigma_d$ and every angle in
$B_{\rm reg}$ belongs to a two-point $\sim$-class. The difference $B\setminus B_{\rm reg}$ is countable.
\end{proposition}
\begin{proof}
 The landing class of a branched point contains at least three angles, so its convex hull has nonempty interior. These interiors are pairwise disjoint and each contains a point of \(\mathbb{Q}^2\); hence there are at most countably many such classes. Since every \(\sim\)-class is finite, the angles landing at branched points form a countable set.

Next, fix a Fatou domain \(U\) and its gap \(G\). If a cutpoint of \(J_f\) on \(\partial U\) has exactly two landing angles, the chord joining them is an edge of \(G\). The edges of \(G\) correspond to complementary intervals of \(G\cap\mathbb{S}\), and hence form a countable family. Since there are only countably many Fatou domains, the angles landing at cutpoints on Fatou boundaries are also countable.

Finally, the marking angles and the \(\sim\)-class of each marking {angle} are both finite. Thus \(\Upsilon_0\), and hence \(\Upsilon\), is countable. In particular, \(B\setminus B_{\rm reg}\) is countable.

The remaining conclusions follow immediately from the choice of  $\Upsilon$
\end{proof}

By this proposition, proving Theorem~\ref{thm:equal1} reduces to checking that
\[h=(\log d)\Hd B_{\rm reg}.\] To this end, we need to relate the set \(B_{\rm reg}\) to the paths
in \(\G_\xi\). The next result provides a starting point for establishing such a connection.

\begin{lemma}\label{lem:coverage}
Every angle of $B_{\rm reg}$ eventually lands, under $q$, on
a non-degenerate arc $I_v\subset K_{\overline{f}}$ for some vertex $v$ of $\G_\xi$.
\end{lemma}

Indeed, this lemma is a consequence of \cite[Lemma 3.12]{BOPT1}.  Following \cite{BOPT1}, an \emph{atom} of \(J_f\) is either a point or the boundary of a Fatou domain. It is a \emph{critical atom} if the Fatou domain or the point defining it is critical.

\begin{proposition}[\cite{BOPT1},\,Lemma 3.12]\label{pro:atom}
For every cutpoint $x$ of $J_f$, there exists an integer $k\geq0$ such that $f^k(x)$ either belongs to a critical atom, or separates two
critical atoms, or separates a critical atom from its image.
\end{proposition}

\begin{proof}[Proof of Lemma \ref{lem:coverage}]
Let \(\theta\in B_{\rm reg}=B\setminus\Upsilon\) and set \(z=p_f(\theta)\). Then \(z\) is a cutpoint of \(J_f\), and by the choice of \(\Upsilon\), its orbit avoids all critical atoms of \(f\). By Proposition~\ref{pro:atom}, after iterating \(z\) if necessary, we may assume that \(z\) separates either two critical atoms or a critical atom from its image, of $f$.

Since \(z\) avoids the closure of every Fatou domain, we can choose a regulated arc \(\gamma\subset K_f\) whose endpoints are either two critical points or a critical point and a critical value of \(f\), and whose interior contains \(z\). Although \(\gamma\) may meet \(T_f(\xi)\), the point \(z\) does not. Hence \(q(\theta)\) is an interior point of the non-degenerate arc \(\phi(\gamma)\), whose endpoints are critical points or critical values of \(\overline{f}\).

Apply Proposition~\ref{prop:arc-rule} to the arc \(\phi(\gamma)\). Its image under \(\overline{f}\) is a finite union of arcs with postcritical endpoints of \(\overline{f}\), so \(\overline{f}(q(\theta))\) enters some \(I_v\) for \(v\in\Xi_s\). Since \(f(z)\) also avoids \(T_f(\xi)\), it follows that \(\phi(f(z))=\overline{f}(q(\theta))\) is an interior point of \(I_v\). Thus \(I_v\) is non-degenerate.
\end{proof}

We next establish the codings and covers of $B_{\rm reg}$.
We begin by recalling the same unlinked classes in three spaces: the unit circle, the $f$-plane and the $\overline{f}$-plane (see Section \ref{sec:partition}).
\vspace{3pt}

Let $\SSS$ be the set of $\xi$-unlinked classes
on the unit circle. For $ S\in\SSS$, let $V_S$
be the corresponding unlinked class in the $f$-plane: an
external ray has angle in $S$ precisely when it lies in $V_S$.
The corresponding reduced class in $\overline{f}$-plane is
\[
 \widehat V_S=\phi(V_S\setminus T_f(\xi)).
\]
Removing the forest before projection is essential:
different sector boundaries can be identified at a collapsed
forest point.
For regular angles, there is no such identification,
and the three descriptions agree:
\begin{equation}\label{eq:three-unlinked-classes}
 \theta\in S
 \quad\Longleftrightarrow\quad p(\theta)\in V_S
 \quad\Longleftrightarrow\quad q(\theta)\in\widehat V_S,
 \qquad \theta\in B_{\mathrm{reg}}.
\end{equation}

Suppose that $v=v_0\in \Xi_s$ and $\theta\in B_{\rm reg}$ with $q(\theta)\in I_v$. Regularity excludes subdivision points and their preimages; in particular, $\theta$ avoids
$\bigcup_{k\ge0}\sigma_d^k\bigl(\bigcup_{Q\in\xi}\Theta(Q)\bigr)$. Thus, by Proposition \ref{prop:arc-rule}, there exists a unique edge $e=e_0:v_0\to v_1$ and a unique unlinked class $S_{e}\in \SSS$ such that
\[q(\theta)\in L_e,\quad \theta\in S_{e},\quad \overline{f}\circ q(\theta)=q\circ\sigma_d(\theta)\in I_{v_1}.\]

 We claim that this label works
for every regular angle landing on $L_e$, i.e.,
\begin{equation}\label{eq:edge-sector}
 B_v(e):=B_{\mathrm{reg}}\cap q^{-1}(L_e)\subset S_e.
\end{equation}
To verify the claim, choose  endpoint angles
$\alpha_e,\beta_e\in\overline{ S_e}$ such that
\[
 L_e=\phi(\g_e),\qquad
 \g_e=[p(\alpha_e),p(\beta_e)]_{K_f}
                   \subset\overline{V_{S_e}},
\]
where the inclusion follows from Lemma~\ref{lem:slits}.
If $q(\theta)\in L_e$ and $\theta$ is regular, its original
landing point $p(\theta)$ is the unique
point of its $\phi$-fiber, hence belongs to $\g_e$.
It also avoids the marking ray-stars, so it lies in $V_{S_e}$, not on the boundary. Property~\eqref{eq:three-unlinked-classes}
gives $\theta\in S_e$.

Since $q(\sigma_d(\theta))\in I_{v_1}$ and $\sigma_d(\theta)$ is still regular, we can iterate the process above.
Set
\[\text{$\theta_j=\sig^j(\theta)$\quad and\quad $x_j=q(\theta_j)=\overline{f}^{\,j}(q(\theta))$},\quad \forall\, j\geq0.\]
 Let $v_{j+1}$ be the  target of $e_j$. Then
\[
 x_j\in L_{e_j}\subset I_{v_j},\quad\text{and}\quad
 x_{j+1}=\overline{f}(x_j)\in \overline{f}(L_{e_j})=I_{v_{j+1}}.
\]
Therefore, $\theta$ is coded by a unique infinite path $\omega(\theta)$ in $\G_\xi$
\[\omega(\theta): v=v_0\overset{e_0}{\to} v_1\overset{e_1}{\to}v_2\overset{e_2}{\to}\cdots,\]
by an orbit in $\overline{f}$-plane:
\[q(\theta)=x_0\overset{\overline{f}}{\to} x_1\overset{\overline{f}}{\to} x_2\overset{\overline{f}}{\to}\cdots,\quad x_j\in L_{e_j}\subset I_{v_j},\ x_j\in \wh{V}_{S_{e_j}},\]
 and by the corresponding angular itinerary in $\mathbb{S}$:
\[
 S_{e_0},\ S_{e_1},\ S_{e_2},\ldots,
 \qquad \theta_j\in S_{e_j}.
\]
Overlaps between different target arcs cause no ambiguity:
the selected edge has already specified the next formal vertex.

Conversely, given a length-$n$ path $\omega$ in $\G_\xi$ starting at $v$, define its set
of coded regular angles by
\begin{equation}\label{eq:coded-regular-angles}
 B_v(\omega):=B_{\mathrm{reg}}\cap q^{-1}(A_\omega),
\end{equation}
where $A_\omega$ is the full inverse interval determined by $\omega$, see \eqref{eq:coding-full-interval}.
If $B_v(\omega)$
is nonempty, inductively applying \eqref{eq:edge-sector} shows  that
\begin{equation}\label{eq:16}
\omega(\theta)|_n=\omega\quad\text{and}\quad\sigma^j_d(\theta)\in S_{e_j},\ j=0,\ldots,n-1; \quad\forall\, \theta\in B_v(\omega),
\end{equation}
where \(\omega(\theta)|_n\) is the length-\(n\) prefix of \(\omega(\theta)\), and \(e_j\) is the \((j+1)\)-th edge of $\omega$.

For fixed $v$ and $n$, these sets $B_v(\omega)$ partition
$B_{\mathrm{reg}}\cap q^{-1}(I_v)$ by
\eqref{eq:coding-overlap} and regularity.
The corresponding \emph{(angular) cylinder to $\omega$} is
\begin{equation}\label{eq:angular-cylinder}
 S_v(\omega)=\bigcap_{j=0}^{n-1}\sig^{-j}(S_{e_j}),
 \quad \text{with }\ B_v(\omega)\subset S_v(\omega).
\end{equation}

Summarizing the discussion above, we obtain the following result.

\begin{lemma}\label{lem:coverage1}
If $\theta\in B_{\rm reg}$ and $q(\theta)\in I_v$ for some vertex  $v$ of $\G_\xi$,
then there is a unique infinite  path $\omega$ starting at $v$
such that, for every $n\ge1$,
\[
 \theta\in B_v(\omega|_n)\subset S_v(\omega|_n).
\]
\end{lemma}

Note that the sets $A_\omega,L_\omega$ are subarcs of $I_v\subset K_{\overline f}$, whereas
$B_v(\omega),S_v(\omega)$ are subsets of $\mathbb S$.
The vertex $v$ specifies the arc, and the edge $e$ specifies its branch $L_e$ and sector label $S_e$.

\subsection{Upper bound on dimension}
This subsection is devoted to  the upper bound estimation on the Hausdorff dimension:
\begin{proposition}\label{prop:dimension-upper}
One has $\Hd B\le h/\log d$.
\end{proposition}

\begin{lemma}\label{lem:cylinders}
For every vertex $v$ of $\G_\xi$ and
every length-$n$ path $\omega$ starting at $v$, the set $B_v(\omega)$ defined in \eqref{eq:coded-regular-angles} is covered by
 at most $2dn$ circle intervals of length
at most $d^{-n}$.
\end{lemma}
\begin{proof}
 By  \eqref{eq:angular-cylinder}, it suffices to show that the cylinder $S_v(\omega)$ consists of  at most $2dn$ intervals, each of which has length at most $d^{-n}$.

We use the inverse representation of $\sigma_d$. For $S\in\SSS$, the restriction $\sig|_S$ is
injective. Write its inverse, with its actual domain, as
\[
 \beta_S:\sig(S)\longrightarrow S.
\]
Denote the edges of \(\omega\) by $(e_j)_{j=0}^{n-1}$. By \eqref{eq:16}, the path $\omega$ determines the partially defined composition
\begin{equation}\label{eq:itinerary-inverse}
 \beta_\omega=
 \beta_{S_{e_0}}\circ\beta_{S_{e_1}}\circ\cdots
                  \circ\beta_{S_{e_{n-1}}}.
\end{equation}
Its domain consists of terminal angles for which every
intermediate inverse is defined. On this domain,
\[
S_v(\omega)=\beta_\omega(\operatorname{dom}\beta_\omega),
 \quad \theta=\beta_\omega(\sig^n\theta)
            \ (\theta\in S_v(\omega)).
\]

Each inverse $\beta_S$
is piecewise affine of slope $1/d$. Its discontinuities and
missing endpoints are controlled by the marked angles.
There are at most $2d-2$ marked vertices, so the number of
cuts needed for one inverse map is bounded by $2d-2$.

In the composition~\eqref{eq:itinerary-inverse}, the number
of pieces grows at most linearly. When an inverse map is
adjoined on the right, each old cut has at most one preimage,
because that inverse map is globally injective on its domain.
Only its own, at most $2d-2$, cuts can be added. Induction gives
at most $(2d-2)n+1\leq 2dn$ continuous pieces. Each has slope $d^{-n}$
on an interval of length at most one.
\end{proof}

\begin{proof}[Proof of Proposition \ref{prop:dimension-upper}]
For each vertex $v$, set
\[
 B_v=B_{\mathrm{reg}}\cap q^{-1}(I_v).
\]
It follows from Lemma~\ref{lem:coverage} that
\[
 B_{\mathrm{reg}}\subset
 \bigcup_{m\ge0}\ \bigcup_{v\in \Xi_s}\sig^{-m}(B_v).
\]
For every $m$, taking the $\sigma_d^m$-preimage
preserves Hausdorff dimension, and
the set $\Xi_s$ is countable, so it suffices to prove
$\Hd B_v\le h/\log d$ for each fixed $v$.

For $\theta\in B_v$, let $\omega_v(\theta)$ be the infinite
path coding $\theta$ given in Lemma~\ref{lem:coverage1}.
It either visits some vertex infinitely often
or eventually leaves every finite set of vertices. We estimate
the corresponding angle sets separately.

We shall also use the
compatibility with a change of starting time: for $\theta\in B_v$, let
$u=\omega_v(\theta)(k)$ be the vertex reached by its actual coding at time $k$.
Then $\sigma_d^k(\theta)\in B_u$, and its coding relative to $u$ is the tail of $\omega_v(\theta)$ starting at time $k$.\vspace{2pt}

For a vertex $u$, define
\[
 R_u=\{\theta\in B_u:\omega_u(\theta)
                 \text{ returns to }u\text{ infinitely often}\}.
\]
Let $\mathcal C_u(n)$ be the set of closed  paths of length
$n$ based at $u$.
 If $\theta\in R_u$ and $n$ is a return time,
then $\omega_u(\theta)|_n\in\mathcal C_u(n)$, and
$\theta\in B_u(\omega_u(\theta)|_n)$ by Lemma \ref{lem:coverage1}. Since these return times are
unbounded, for every integer $N\ge1$ we have
\begin{equation}\label{eq:17}
 R_u\subset\bigcup_{n\ge N}\ \bigcup_{\omega\in\mathcal C_u(n)}
                       B_u(\omega)\subset\bigcup_{n\ge N}\ \bigcup_{\omega\in\mathcal C_u(n)}
                       S_u(\omega).
\end{equation}

Given $t>h/\log d$, choose $\epsilon>0$ with
$h+\epsilon<t\log d$. By the definition of $h$, the total
number $C_n$ of length-$n$ closed paths satisfies
$C_n\le e^{(h+\epsilon)n}$ for all sufficiently large $n$.
Since $C_n(u)\le C_n$, Lemma~\ref{lem:cylinders} and \eqref{eq:17}
implies that $R_u$ has an interval cover
of mesh at most $d^{-N}$ and total $t$-content at most
\[
 \sum_{n\ge N}2dn C_nd^{-nt}
 \le\sum_{n\ge N}2dn e^{-n(t\log d-h-\epsilon)}.
\]
The last sum tends to zero as $N\to\infty$, proving
$\Hd R_u\le h/\log d$.

If the coding of an angle in $B_v$
visits $u$ infinitely often, shift to any visit to $u$.
The shifted angle belongs to $R_u$. All angles in $B_v$ of
this type therefore belong to
\[
 \bigcup_{k\ge0}\ \bigcup_{u\in \Xi_s}\sig^{-k}(R_u),
\]
whose dimension is at most $h/\log d$.

Let $X_v$ be the set of angles in $B_v$ whose coding eventually
leaves every finite set of vertices. Fix $N\ge1$. For each
$u$ avoiding $S_N$ (defined in \eqref{eq:10}), let
\[
 D_{u,N}=\{\theta\in B_u:\omega_u(\theta)
                    \text{ avoids }S_N\text{ at every time}\}.
\]
Lemma~\ref{lem:escape} bounds the number of possible length-$n$
prefixes of these codings by $2^{1+n/N}$. By
Lemma~\ref{lem:cylinders}, $D_{u,N}$ is therefore covered by
at most $2dn 2^{1+n/N}$ intervals of length at most $d^{-n}$.
For $t>(\log2)/(N\log d)$, their total $t$-content is at most
\[
 2\cdot 2dn\cdot e^{-n(t\log d-(\log2)/N)},
\]
which tends to zero as $n$ tends to infinity. Hence
$$\Hd D_{u,N}\le(\log2)/(N\log d).$$

For every $\theta\in X_v$, there is a time $k$ after which its
coding avoids $S_N$. If $u$ is the vertex at that time, we have $\sig^k(\theta)\in D_{u,N}$.
Consequently
\[
 X_v\subset\bigcup_{k\ge0}\ \bigcup_{u\notin S_N}
                      \sig^{-k}(D_{u,N}).
\]
This is a countable union, so
$\Hd X_v\le(\log2)/(N\log d)$. Since $N$ can be arbitrary large,
it follows that $\Hd X_v=0$.

The two cases exhaust $B_v$ and give $\Hd B_v\le h/\log d$.
\end{proof}

\subsection{Lower bound on dimension}
\label{sec:topological-returns}
{To prove the lower bound, we lift compact horseshoe sets to
the circle. Lemma~\ref{lem:returns} supplies a horseshoe with a
high-entropy non-escaping set. Its lift under $q$, together with the
entropy--dimension formula, gives the required dimension estimate.}\vspace{3pt}

For a compact metric space \(X\) and a continuous map \(T:X\to X\), the topological entropy \(h_{\rm top}(T)\) is defined as usual (see, e.g., \cite{Bow}). We shall use the following basic properties.

\begin{proposition}\label{pro:entropy}
Let $T:X\to X$ be a continuous map on a compact metric space $X$. The following statements hold.
\begin{enumerate}
\item If $X=X_1\cup X_2$, with both $X_1$ and $X_2$ compact and $T$-invariant, then
$$h_{\rm top}(T)=\max\big\{\,h_{\rm top}(T|_{X_1}),h_{\rm top}(T|_{X_2})\,\big\}.$$
\item Let $E$ be a closed and $T$-invariant subset of $X$. Suppose that
\[ {\rm dist}(T^n(x),E)\to 0 \quad\text{as}\quad n\to\infty\]
locally and uniformly  in $X\setminus E$. Then $h_{\rm top}(T)=h_{\rm top}(T|_E)$.\vspace{2pt}

\item Suppose that  $Y$ is a compact metric space, $R:Y\to Y$ is a continuous map and $\pi:Y\to X$ is a continuous surjection, such that
$$\pi\circ R=T\circ \pi\quad\text{on}\quad Y.$$ Then $h_{\rm top}(T)\leq h_{\rm top}(R)$. Furthermore, if $\#\pi^{-1}(x)<\infty$ for every $x\in X$, we obtain the equality $h_{\rm top}(T)= h_{\rm top}(R)$.
\end{enumerate}
\end{proposition}
\begin{proof}
{Statements (1) and (2) are Propositions~2 and~3 in
\cite[Section~1.1, p.~68]{Do}, respectively. Statement (3) follows
from \cite[Theorem~17]{Bow}.}
\end{proof}

In literature, if $\pi:Y\to X$ is a semiconjugacy from $R:Y\to Y$ to $T:X\to X$, we usually say that $(T,X)$ is a \emph{factor} of $(R,Y)$ via a \emph{factor map} $\pi$.

\begin{lemma}\label{lem:fatou-removal}
Let $\g\subset K_f$ be a regulated arc joining
two Julia points, and suppose the arc $I=\phi(\g)\subset K_{\overline{f}}$ is  non-degenerate. Suppose that $X\subset I$ is compact, with
$\overline{f}^{\,n}(X)\subset X$ for some $n\ge1$. Set
\[
 X_{\overline{J}}=X\cap \overline{J},\quad A_0=q^{-1}(X_{\overline{J}}),\quad
 A=A_0\cap B_{\mathrm{reg}}.
\]
Then the following results hold.
\begin{enumerate}
\item Both $X_{\overline{J}}$ and $A_0$ are compact, invariant under
$\overline{f}^{\,n}$ and $\sig^n$, respectively.\vspace{2pt}
\item  $h_{\mathrm{top}}(\overline{f}^{\,n}|_X)=h_{\mathrm{top}}(\overline{f}^{\,n}|_{X_{\overline{J}}})=h_{\rm top}(\sigma_d^n|_{A_0})$.\vspace{3pt}
\item The set  $A_0\setminus A$ is
countable, so that $\Hd A=\Hd A_0$.
\end{enumerate}
\end{lemma}
\begin{proof}
(1) Since \(X\) and \(\overline{J}\) are compact and \(\overline{f}^{\,n}\)-invariant, so is \(X_{\overline{J}}\). Note that \(q:\mathbb{S}\to\overline{J}\) semiconjugates \((\mathbb S,\sigma_d)\) to \((\overline{J},\overline{f})\). Hence \(A_0\) is compact and \(\sigma_d^n(A_0)\subset A_0\).\vspace{3pt}

(2) By definition, {the fibers of the maps $\phi|_{J_f}$ and $p$ are finite}. Then
 Proposition \ref{pro:entropy}\,(3) shows  that the factor map
\[
 q:(A_0,\sig^n)\longrightarrow(X_{\overline{J}},\overline f^{\,n})
\]
preserves topological entropy. Then the second equality holds.

To prove the first equality, lift \(X\) to the \(f\)-plane and set
\[
Y=\phi^{-1}(X),\qquad Y_J=Y\cap J_f,\qquad g=f^n|_Y.
\]
Indeed, \(Y\) lies in the union of \(\g\) and the components of \(T_f(\xi)\) whose images lie in \(X\). By Lemma~\ref{lem:regulated}, \(Y\) is contained in a finite regulated tree. These sets are compact and forward invariant under \(g\). By a similar argument, we have
\[
h_{\mathrm{top}}(g|_{Y_J})=h_{\mathrm{top}}(\overline{f}^{\,n}|_{X_{\overline{J}}}).
\]

It suffices to prove  $h_{\mathrm{top}}(g|_{Y_J})
  =h_{\mathrm{top}}(g)$: if this {holds}, combining the equality in the last paragraph and Proposition \ref{pro:entropy}\,(3), it follows that
  \[
 \begin{aligned}
 h_{\mathrm{top}}(\overline{f}^{\,n}|_X)
 &\le h_{\mathrm{top}}(g)
  =h_{\mathrm{top}}(g|_{Y_J})
 &=h_{\mathrm{top}}(\overline{f}^{\,n}|_{X_{\overline{J}}})
 \le h_{\mathrm{top}}(\overline{f}^{\,n}|_X).
 \end{aligned}
\]
This proves the required equality.

Let \(y\in Y\setminus Y_J=Y\cap F_f\). We claim that \(y\) is either contained in an attracting domain or eventually maps into a periodic Siegel disk. Suppose \(y\) lies in no attracting domain. Then, after iteration, we may assume \(y\) belongs to a periodic Siegel disk \(D\). If \(y\) is not the center of \(D\), then the closure of its orbit \(\{g^k(y):k\ge0\}\) contains a Jordan curve. This contradicts the fact that \(Y\) is contained in a regulated tree.

Let $\widehat{Y}_J$ be the union of $Y_J$ and all
attracting and Siegel periodic points of $f$ in $Y$.
Every $y\in Y\setminus\widehat{Y}_J$ either lies in an
attracting domain or eventually falls on a Siegel periodic
point. Since convergence to attracting cycles is locally
uniform on attracting domains, we obtain
\[
\operatorname{dist}\bigl(g^k(y),\widehat{Y}_J\bigr)
\longrightarrow 0
\]
locally uniformly on $Y\setminus\widehat{Y}_J$.
Proposition~6.13(2) then yields
\[
h_{\rm top}(g)=h_{\rm top}(g|_{\widehat{Y}_J})
=h_{\rm top}(g|_{Y_J}),
\]
where the second equality follows from the fact that
adding finitely many periodic points does not change
topological entropy.

(3) Outside the finite set $\phi(T_f(\xi))\cup\partial I$,
each $x\in X_{\overline{J}}$ has a unique lift $y\in J_f\cap\operatorname{int}(\g)$.
Such a point $y$ is a cut point of $J_f$, and hence is biaccessible.
Since $q$ has finite fibers, $A_0\setminus B$ is finite.
Together with the countability of $B\setminus B_{\mathrm{reg}}$,
this proves the last assertion.
\end{proof}

\begin{lemma}\label{lem:entropy}
Let $T:\bigcup_{j=1}^M I_j\to I$ be a monotone interval $M$-horseshoe. Then
\[h_{\rm top}(T|_X)\geq \log M,\]
where $X$ is the non-escaping set of this horseshoe.
\end{lemma}
This is a classical entropy estimate for interval horseshoes, see e.g.~\cite[Chapter VIII, Proposition 14]{BC} and
\cite[Proposition 4.6]{Ru}.

\begin{lemma}[Entropy-dimension formula on the circle]\label{lem:formula}
\label{lem:circle-dimension}
If $A\subset\Sone$ is compact and $\sig(A)\subset A$, then
\[
 h_{\mathrm{top}}(\sig|_A)=\log d\cdot\Hd A.
\]
\end{lemma}
This is a classical entropy--dimension formula, see e.g.~\cite[Proposition III.1]{Fur}.

\begin{proposition}\label{prop:dimension-lower}
One has $\Hd B\ge h/\log d$.
\end{proposition}
\begin{proof}
Assume \(h>0\) and fix \(0<\epsilon<h\). By Corollary~\ref{coro:many-returns}, there exists a non-zero-arc vertex \(v\) supporting \(M\ge2\) closed paths of  length \(m\) such that \[\log M/m>h-\epsilon.\] 
Lemma~\ref{lem:returns} realizes these paths as a monotone interval \(M\)-horseshoe on the postcritical arc \(I_v\), yielding a compact \(\overline{f}^{\,m}\)-invariant set \(X\subset I_v\). Applying Lemmas~\ref{lem:fatou-removal} and~\ref{lem:entropy}, we obtain a compact \(\sig^m\)-invariant set \(A_0\) with 
 \[h_{\rm top}(\sigma_d^m|_{A_0})\geq \log M,\] together with a subset \(A=A_0\cap B_{\mathrm{reg}}\) of the same Hausdorff dimension. Thus
\[
\begin{aligned}
\Hd B&\ge\Hd A=\Hd A_0
   =\frac{h_{\mathrm{top}}(\sig^m|_{A_0})}{m\log d}\\
   &\ge\frac{\log M}{m\log d}
   >\frac{h-\epsilon}{\log d},
\end{aligned}
\]
where the second equality is by Lemma~\ref{lem:formula}. This completes the proof.
\end{proof}

\begin{proof}[Proof of Theorem \ref{thm:equal1}]
It follows from Propositions \ref{pro:reduction}, \ref{prop:dimension-upper} and \ref{prop:dimension-lower}.
\end{proof}

\section{Continuity of biaccessibility dimension for polynomials}\label{sec:continuity1}
\label{sec:continuity}

{We prove Theorem~\ref{thm:main} by passing from actual smooth
landing data to a perfect lamination and its dendrite model.
This model provides no global landing map for the original
polynomial. We therefore realize only the finite periodic classes
needed to bound conformal return systems. Uniform pruning controls
their return times and boundary angles; the resulting systems give
the two semicontinuity estimates.}

\subsection{Biaccessible lamination and its perfect part}\label{sec:biaccessible}

Recall that $\Sigma_0(f)$ consists of all arguments whose external radii terminate at critical points of $f$. Then 
\begin{equation}\label{eq:5}
\#\Sigma_0(f)\leq \mu_d:=(d-1)2^{d-1},
\end{equation}
and the (smooth) external ray $R_f(\theta)$ is well-defined if and only if $\theta$ avoids 
$$\Sigma(f)=\bigcup_{j=0}^{\infty}\sigma_d^{-j}(\Sigma_0(f)).$$

For each point \(z\in J_f\), let the \emph{landing class} of \(z\) be the set of smooth angles whose external rays land at \(z\). For every such finite landing class containing at least two angles, take the edges of its convex hull. The resulting collection, together with points in $\mathbb S$, is a prelamination $\LLL_{\rm bi}^\circ(f)$. Taking its closure  in the space of chords yields a lamination $\LLL_{\rm bi}(f)$, called the \emph{biaccessible lamination} of \(f\). Its perfect part is $\LLL_{\rm bi}^{\rm p}(f)$. 

It is precisely the following two facts that relate \(\LLL_{\rm bi}^{\rm p}(f)\) to the dynamics of \(f\).

By a \emph{ray pair} of \(f\), we mean the union of two smooth external rays and their common landing point.
If $\ell=\overline{ab}$ and the rays at $a,b$ are smooth and land together, $R_f(\ell)$ denotes this ray pair.
\begin{proposition}\label{pro:fact}
 Each leaf of $\LLL_{\rm bi}^\circ(f)$ corresponds to a ray pair of $f$; and
each leaf of $\LLL_{\rm bi}^{\rm p}(f)$ can be approximated by leaves in $\LLL_{\rm bi}^\circ(f)$.
\end{proposition}

\begin{proposition}\label{pro:perfect}
The lamination $\LLL_{\rm bi}^{\rm p}(f)$ is invariant and perfect.
\end{proposition}

\begin{proof}
We first locate the possible failures of the usual ray-lifting
argument. Recall
\[
 \Sigma_1(f)=\sig(\Sigma_0(f)).
\]
By Lemma \ref{lem:lift}, every lift
 of a smooth angle outside $\Sigma_1(f)$ is smooth.

First, a finite landing class cannot map into an infinite one: the latter has infinitely many angles avoiding \(\Sigma_1(f)\), whose rays lift smoothly to land at every preimage of the landing point, yielding infinitely many angles in the source class.

Second, only finitely many finite landing classes meet $\Sigma_1(f)$. Mark their edges together with the edges of their finitely many preimage classes. Outside this finite family, forward images and sibling lifts of leaves in $\LLL_{\rm bi}^\circ (f)$ obey the usual invariance rules in Definition \ref{def:invariant-lamination}.

We now consider the derived set \(\LLL_{\rm bi}'(f)\) of non-isolated leaves of \(\LLL_{\rm bi}(f)\). Every non-degenerate leaf of \(\LLL_{\rm bi}'(f)\) is a limit of distinct unmarked leaves in \(\LLL_{\rm bi}^{\circ}(f)\). Forward images and sibling collections pass to the limit. Hence, by the previous discussion, \(\LLL_{\rm bi}'(f)\) is an invariant lamination. Applying Lemma~\ref{lem:perfect} to \(\LLL_{\rm bi}'(f)\), we see that its perfect part \((\LLL_{\rm bi}')^{\rm p}(f)\) is invariant and perfect. Since \((\LLL_{\rm bi}')^{\rm p}(f)=\LLL_{\rm bi}^{\rm p}(f)\), the assertion follows.
\end{proof}

\begin{lemma}\label{lem:infinite-landings}
The smooth angles in infinite landing classes of $f$ have Hausdorff dimension zero.
Consequently,
\begin{equation}\label{eq:actual-upper-perfect}
 \Hd B(f)\le\Hd B(\LLL_{\rm bi}^{\rm p}(f)).
\end{equation}
\end{lemma}

\begin{proof}
Let \(\Theta(z)\) be an infinite landing class at \(z\in J_f\). We first claim that \(z\) is preperiodic. Otherwise,  choose a set $A$ of \(2^d+1\) angles in this infinite landing class.  {The convex hulls of its forward images form a wandering polygon.} It follows from Theorem \ref{thm:kiwi} that $\# A\leq 2^d$, a contradiction. 

It suffices now to consider a periodic point: if $f^k(z)=w$, then
$\sigma_d^k(\Theta(z))\subset\Theta(w)$, and taking an inverse
 under $\sigma_d^k$ preserves Hausdorff dimension.

Suppose that $f^m(z)=z$, and set $A=\overline{\Theta(z)}$.
Since  $f^m$ is locally
univalent at $z$, its action on the smooth rays landing at $z$ is
injective and preserves cyclic order. By continuity, $\sigma_d^m(A)\subset A$, and
the restriction to $A$ preserves weak cyclic order.
Hence $\sigma_d^m:A\to A$ extends to
a continuous monotone circle map 
\[\sigma_d^m:\partial\, {\rm CH}(A)\to \partial\,{\rm CH}(A)\]
of degree
one. Such a circle map has
zero topological entropy. Consequently,
\[
 h_{\rm top}(\sigma_d^m|_A)=0,
 \qquad \Hd\Theta(z)\le\Hd A=0,
\]
where the last equality follows from Lemma~\ref{lem:formula}.
As preperiodic points are countable, the union of all infinite landing classes has Hausdorff dimension zero.

Every angle in a finite non-singleton landing class is an endpoint of a leaf of \(\LLL_{\rm bi}(f)\). By Proposition~\ref{pro:countable}, passing to the perfect part removes only countably many leaves. Together with the preceding paragraph, this proves~\eqref{eq:actual-upper-perfect}.
\end{proof}

\subsection{\texorpdfstring{The dendrite model of $\LLL_{\rm bi}^{\rm p}(f)$}{The dendrite model of the perfect biaccessible lamination}}\label{sec:dendrite-model}
Set \(\LLL_{\rm bi}^{\rm p}(f)=\LLL_{\rm bi}^{\rm p}\). By Lemma~\ref{lem:proper}, the lamination \(\LLL_{\rm bi}^{\rm p}\) is proper and hence generates a laminational equivalence relation \(\sim=\sim_{\LLL_{\rm bi}^{\rm p}}.\) Let \(J_\sim,g_\sim,\pi_\sim\) be the corresponding Julia set, topological polynomial, and projection. Then \((g_\sim,J_\sim)\) is a factor of \((\sigma_d,\mathbb{S})\) via \(\pi_\sim\). Note that, by Lemma \ref{lem:perfect}, all Fatou domains of $g_\sim$ are attracting.\vspace{2pt}

Choose a marking portrait $\xi$ of $g_\sim$. 
As in Section~\ref{sec:reduce}, collapsing the components of the collapsing forest of \(g_\sim\) to points yields the reduced dynamics \(\overline{g}_\sim:J_{\overline{g}_\sim}\to J_{\overline{g}_\sim}\), which is a factor of \(g_\sim:J_\sim\to J_\sim\) through the projection \(\phi\).\vspace{2pt}

Since \(\LLL_{\rm bi}^{\rm p}\) is perfect, the closures of the Fatou domains of \(g_\sim\) are pairwise disjoint by Lemma \ref{lem:proper}. Collapsing the closure of each Fatou domain to a point yields a dendrite Julia set \(J_\approx\) and a topological polynomial \(g_\approx\). In this case, every component of the collapsing forest of \(g_\sim\) lies in the closure of a unique Fatou domain. Hence \((g_\approx, J_\approx)\) is a factor of \((\overline{g}_\sim, J_{\overline{g}_\sim})\), with projection denoted by \(\wp\). The successive quotients and their commuting diagram are given in Figure \ref{fig:successive-quotients}. 
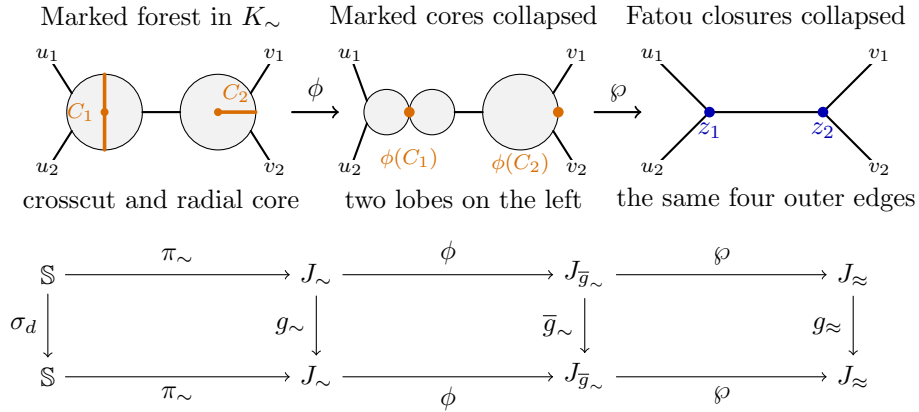
\begin{figure}[htbp]
\centering
\begin{tikzpicture}[x=1cm,y=1cm,line cap=round,line join=round,font=\small]
  \begin{scope}[shift={(1.65,0)}]
    \node at (0,1.25) {Marked forest in $K_\sim$};
    \fill[gray!10] (-0.75,0) circle(0.5) (0.75,0) circle(0.5);
    \draw[black] (-0.75,0) circle(0.5) (0.75,0) circle(0.5);
    \draw[black,thick] (-0.25,0)--(0.25,0);
    \draw[black,thick] (-1.183,0.25)--(-1.5,0.75)
      (-1.183,-0.25)--(-1.5,-0.75)
      (1.183,0.25)--(1.5,0.75) (1.183,-0.25)--(1.5,-0.75);
    \draw[orange!85!black,line width=1.35pt] (-0.75,-0.5)--(-0.75,0.5);
    \draw[orange!85!black,line width=1.35pt] (0.75,0)--(1.25,0);
    \fill[orange!85!black] (-0.75,0) circle(1.5pt) (0.75,0) circle(1.5pt);
    \node[orange!85!black,left,font=\scriptsize] at (-0.75,0) {$C_1$};
    \node[orange!85!black,above,font=\scriptsize] at (0.99,0) {$C_2$};
    \foreach \xx/\yy/\lab in {-1.5/0.75/u_1,-1.5/-0.75/u_2,1.5/0.75/v_1,1.5/-0.75/v_2}
      \node[font=\scriptsize,fill=white,inner sep=1pt] at (\xx,\yy) {$\lab$};
    \node at (0,-1.15) {crosscut and radial core};
  \end{scope}
  \draw[->,thick] (3.38,0)--(4.0,0) node[midway,above] {$\phi$};
  \begin{scope}[shift={(5.65,0)}]
    \node at (0,1.25) {Marked cores collapsed};
    \fill[gray!10] (-1.02,0) circle(0.3) (-0.42,0) circle(0.3)
      (0.75,0) circle(0.5);
    \draw[black] (-1.02,0) circle(0.3) (-0.42,0) circle(0.3)
      (0.75,0) circle(0.5);
    \draw[black,thick] (-0.12,0)--(0.25,0);
    \draw[black,thick] (-1.280,0.15)--(-1.5,0.75)
      (-1.280,-0.15)--(-1.5,-0.75)
      (1.183,0.25)--(1.5,0.75) (1.183,-0.25)--(1.5,-0.75);
    \fill[orange!85!black] (-0.72,0) circle(2pt) (1.25,0) circle(2pt);
    \node[orange!85!black,font=\scriptsize] at (-0.72,-0.62) {$\phi(C_1)$};
    \node[orange!85!black,font=\scriptsize] at (0.75,-0.68) {$\phi(C_2)$};
    \foreach \xx/\yy/\lab in {-1.5/0.75/u_1,-1.5/-0.75/u_2,1.5/0.75/v_1,1.5/-0.75/v_2}
      \node[font=\scriptsize,fill=white,inner sep=1pt] at (\xx,\yy) {$\lab$};
    \node at (0,-1.15) {two lobes on the left};
  \end{scope}
  \draw[->,thick] (7.38,0)--(8.0,0) node[midway,above] {$\wp$};
  \begin{scope}[shift={(9.65,0)}]
    \node at (0,1.25) {Fatou closures collapsed};
    \draw[black,thick] (-0.75,0)--(0.75,0);
    \draw[black,thick] (-0.75,0)--(-1.5,0.75)
      (-0.75,0)--(-1.5,-0.75)
      (0.75,0)--(1.5,0.75) (0.75,0)--(1.5,-0.75);
    \fill[blue!70!black] (-0.75,0) circle(2pt) (0.75,0) circle(2pt);
    \node[blue!70!black,below] at (-0.75,0) {$z_1$};
    \node[blue!70!black,below] at (0.75,0) {$z_2$};
    \foreach \xx/\yy/\lab in {-1.5/0.75/u_1,-1.5/-0.75/u_2,1.5/0.75/v_1,1.5/-0.75/v_2}
      \node[font=\scriptsize,fill=white,inner sep=1pt] at (\xx,\yy) {$\lab$};
    \node at (0,-1.15) {the same four outer edges};
  \end{scope}
  \begin{scope}[shift={(0.15,-2.15)}]
    \node (s0) at (0,0) {$\mathbb S$};
    \node (s1) at (3.55,0) {$J_\sim$};
    \node (s2) at (7.1,0) {$J_{\overline g_\sim}$};
    \node (s3) at (10.65,0) {$J_\approx$};
    \node (t0) at (0,-1.35) {$\mathbb S$};
    \node (t1) at (3.55,-1.35) {$J_\sim$};
    \node (t2) at (7.1,-1.35) {$J_{\overline g_\sim}$};
    \node (t3) at (10.65,-1.35) {$J_\approx$};
    \draw[->] (s0)--(s1) node[midway,above] {$\pi_\sim$};
    \draw[->] (s1)--(s2) node[midway,above] {$\phi$};
    \draw[->] (s2)--(s3) node[midway,above] {$\wp$};
    \draw[->] (t0)--(t1) node[midway,below] {$\pi_\sim$};
    \draw[->] (t1)--(t2) node[midway,below] {$\phi$};
    \draw[->] (t2)--(t3) node[midway,below] {$\wp$};
    \draw[->] (s0)--(t0) node[midway,left] {$\sigma_d$};
    \draw[->] (s1)--(t1) node[midway,left] {$g_\sim$};
    \draw[->] (s2)--(t2) node[midway,left] {$\overline g_\sim$};
    \draw[->] (s3)--(t3) node[midway,left] {$g_\approx$};
  \end{scope}
\end{tikzpicture}
\caption{Collapsing the crosscut $C_1$ produces a figure eight, with two closed disks joined at $\phi(C_1)$; collapsing the radial core $C_2$ leaves a disk with $\phi(C_2)$ on its boundary. The map $\wp$ then collapses both lobes together to $z_1$ and the other disk to $z_2$. }\label{fig:successive-quotients}\par
\end{figure}

Set \(\pi_\approx := \wp \circ \phi \circ \pi_\sim\). Then \(\pi_\approx\) is a factor map from \((\sigma_d, \mathbb{S})\) to \((g_\approx, J_\approx)\). For each \(x \in J_\approx\), the fiber \(\pi_\approx^{-1}(x)\) is called a \emph{$\approx$-class}. It is either a finite $\sim$-class, or the vertex set of an infinite gap of $\LLL_\sim$.

We remark that the actual polynomial is always denoted by $f$, while $g_\sim$, $\overline g_\sim$ and $g_\approx$ denote its models.
A finite landing class is a finite set of smooth angles whose rays of $f$ land together; while
a finite $\approx$-class belongs to the quotient model and is not assumed realized before Lemma~\ref{lem:finite-fibers}.
Figure~\ref{fig:successive-quotients} records the projection maps and their dynamical compatibility.

\begin{lemma}\label{lem:project-returns}
Let $\g$ be a regulated arc in $K_\sim$ such that $I=\phi(\g)$ is non-degenerate and  $\overline{g}_\sim^m:\bigcup_{j=1}^M I_j\to I$ is a monotone $M$-horseshoe with $M\geq2$.
 The projected arcs in $J_\approx$ form a homeomorphic $M$-horseshoe.
\end{lemma}
\begin{proof}
We first show that $\wp(I)$ is non-degenerate. Otherwise,
$\gamma\subset\overline U$ for some Fatou domain $U$.
The return relation forces $g_\sim^m(U)=U$.
Let $c$ be the image under $\phi$ of the center of $U$.
Since $\gamma$ is regulated, $I$ either lies in the image
of one internal ray or consists of two arms joined at $c$.
On each internal ray, the induced return is either injective
or constant, and its image lies in a single internal ray.
In the first case, there cannot be two branches onto $I$.
In the second, every branch onto $I$ must contain $c$ in
its interior, contradicting the disjointness of the branch
interiors. Thus $\wp(I)$ is non-degenerate.

Set $A=\wp(I)$ and $A_j=\wp(I_j)$. By semiconjugacy,
\(
g_\approx^m(A_j)=A,
\)
so every $A_j$ is non-degenerate. Since $\wp|_I$ is monotone,
the arcs $A_j$ have pairwise disjoint interiors, and each
 map $g_\approx^m:A_j\to A$ is a monotone surjection.
Its fibers are connected and finite, hence singletons.
Thus every branch is a homeomorphism.
\end{proof}

\subsection{\texorpdfstring{Realization of finite periodic $\approx$-classes}{Realization of finite periodic classes}}\label{sec:finite-realization}

{We now realize each finite periodic $\approx$-class avoiding
$\Sigma(f)$ by smooth rays landing at one periodic point of $f$.
These are precisely the classes needed for the boundaries of the
return strips below; no global landing map is asserted.
The following separation lemma is the main tool.}

\begin{lemma}\label{lem:actual-separators}
Let $\Theta$ be a finite $\approx$-class and let $\beta\notin \Theta$.
Then for every $r\geq 1$, there are $r$ leaves in $\LLL_{\rm bi}^\circ(f)$  which
strictly separate the angles of $\Theta$ from $\beta$ and the corresponding ray pairs of $f$ have pairwise
distinct landing points.
\end{lemma}

\begin{proof}
First, $x=\pi_\approx(\Theta)$ and
$y=\pi_\approx(\beta)$ are distinct points of $J_\approx$.
Since there are only countably many $\approx$-classes
containing at least three angles, we may choose
$2r$ points $z_1,\ldots,z_{2r}\in(x,y)$ such that
\[
\ell_j:=\CH\bigl(\pi_\approx^{-1}(z_j)\bigr)
\]
is a leaf of $\LLL_{\rm bi}^{\rm p}(f)$ separating
$\Theta$ from $\beta$ for each $j$.
By Proposition~\ref{pro:fact}, these leaves can be
approximated by $2r$ distinct leaves of
$\mathcal L_{\rm bi}^{\circ}(f)$, preserving strict
separation.

Fix $\theta\in\Theta$. Each approximating leaf
crosses $\overline{\theta\beta}$. Since the boundary
of a convex polygon crosses this chord at most twice,
at most two approximating leaves can belong to the
same finite landing class. Thus there are at least
$r$ distinct landing points. Choosing one leaf at
each of these points completes the proof.
\end{proof}

\begin{lemma}[Realization of periodic classes]\label{lem:finite-fibers}
Let $\Theta$ be a finite periodic $\approx$-class disjoint from
$\Sigma(f)$. Then all rays with angles in $\Theta$
land at a common repelling or parabolic periodic point,
and distinct such classes have distinct landing points. 

\end{lemma}

\begin{proof}
Since angles in $\Theta$ are periodic, their smooth rays land at 
repelling or parabolic points by Lemma \ref{lem:PZ1}.
It suffices to show these landing points coincide. 

We first claim that the set
\[
 X_\Theta:=\bigcup_{\theta\in\Theta}\Imp_f(\theta)
\]
is connected. If $\Theta$ is a singleton, the claim holds as impressions are continua by Lemma~\ref{lem:star-impressions}. Otherwise,
let 
 $\overline{ab}$ be a leaf of $\LLL_{\rm bi}^{\rm p}(f)$ with $a,b\in \Theta$. By Proposition \ref{pro:fact}, we can approximate $\overline{ab}$
by leaves $\overline{a_jb_j}$ of $\LLL_{\rm bi}^{\circ}(f)$, where $R_f(a_j)$ and $R_f(b_j)$ land together. A limit of their common landing
points belongs to $\Imp_f(a)\cap\Imp_f(b)$ by upper
semicontinuity. The impressions are continua, and successive ones meet.
Thus $X_\Theta$ is connected, and the claim is proved.

By the claim, \(X_\Theta\) lies in a component \(K\) of \(K_f\). Since \(K\) contains a periodic landing point, it is periodic. If \(K\) is a singleton, the rays already land together.

Suppose that $K$ is non-degenerate. {Its return admits a polynomial-like restriction; see the
discussion preceding \cite[Theorem~2.2]{PZ1}. Straighten this restriction} to a polynomial $P$ with connected filled Julia set,
using a hybrid conjugacy $\varphi$ defined near $K$ (see \cite{DH2}).
Let $I_K$ be the compact set of angles of smooth or broken
rays accumulating on $K$. By \cite[Theorem~2.2]{PZ1}, there is
a continuous degree-one monotone surjection
$\Pi:I_K\to\Sone$ intertwining the return angle maps, with
\begin{equation}\label{eq:component-ray-correspondence}
 R_f(\theta)\text{ lands at }z\in K
 \ \Longleftrightarrow\
 R_P(\Pi(\theta))\text{ lands at }\varphi(z),
 \quad \theta\in I_K\setminus\Sigma(f).
\end{equation}
Extend $\Pi$ continuously and monotonically to the whole circle, constantly
across each complementary interval of $I_K$;
see \cite[Lemma~4.5]{PZ2}.
When $K_f$ is connected, take $P=f$ and
$\varphi=\Pi=\mathrm{id}$.
Set
\[
 \Lambda=\Pi(\Theta),\qquad Y_\Lambda=\bigcup_{\alpha\in \Lambda}\Imp_P(\alpha).
\]
The set $\Lambda$ is finite and consists of periodic angles.
We verify connectedness and separation for $Y_\Lambda$, then apply
\cite[Lemma~37]{BCO}.

The angles of a ray pair of $f$ landing outside $K$ lie in one component of $\R/\Z\setminus I_K$, hence has equal $\Pi$-images.
 Consequently, if a leaf $\overline{ab}$ of $\LLL_{\rm bi}^{\rm p}(f)$
 satisfies $a,b\in\Theta$ and $\Pi(a)\ne\Pi(b)$, then all
sufficiently close approximating landing pairs land in $K$.
Their projected rays also land together by
\eqref{eq:component-ray-correspondence}. Upper semicontinuity
gives
\[
 \Imp_P(\Pi(a))\cap\Imp_P(\Pi(b))\ne\varnothing.
\]
Following the cyclic order and omitting repeated images proves
that $Y_\Lambda$ is connected; the singleton case is immediate.

Fix \(\beta'\notin\Lambda\) and choose \(\beta\in I_K\) with \(\Pi(\beta)=\beta'\). Choose three ray pairs that strictly separate \(\Theta\) from \(\beta\) and have distinct landing points \(z_1,z_2,z_3\). We claim that the three points lie in \(K\). Indeed, a ray pair landing outside \(K\) would place the connected set \(K\), and hence all angles in \(I_K\), on one side, so it could not separate \(\Theta\) from \(\beta\), a contradiction.

By \textup{(7.3)}, the projected rays land at the distinct points \(\varphi(z_j),j=1,2,3\). Note that none of their angles belongs to \(\Lambda\). Otherwise, \(\Pi(\alpha)=\Pi(\theta)\) for some separating angle \(\alpha\) and some \(\theta\in\Theta\). The ray of angle \(\theta\) already lands by Lemma~2.1, and the ray correspondence would force it to land together with the ray of angle \(\alpha\), contradicting  Lemma~7.5.

A degenerate projected pair is constant on one angular side. That side cannot contain \(\Theta\), since no projected angle lies in \(\Lambda\) (as shown above). Hence it must contain \(\beta\), and its image is \(\beta'\). Among the three projected pairs with distinct landing points, at most one has endpoint \(\beta'\). The other two strictly separate \(\Lambda\) from \(\beta'\) at distinct landing points. For either pair, impressions on opposite sides can meet only at its landing point, so these two landing points are distinct. It follows that
\[
Y_\Lambda\cap\Imp_P(\beta')=\varnothing.
\]
{Choose $r\ge1$ such that $\sigma_{\deg P}^{r}$ fixes
$\Lambda$ pointwise. Equivariance of impressions gives
$P^r(Y_\Lambda)=Y_\Lambda$, so $Y_\Lambda$ is the periodic continuum
required in \cite[Lemma~37]{BCO}.}
Thus, Lemma~37 of~\cite{BCO} yields
\begin{equation}\label{eq:straightened-class-point}
Y_\Lambda=\{\widetilde z\}.
\end{equation}
By \eqref{eq:component-ray-correspondence}, all rays in \(\Theta\) land at \(\varphi^{-1}(\widetilde z)=:z_\Theta\).

For another such class $\Theta'$, choose $\beta\in \Theta'$ and two leaves of $\LLL_{\rm bi}^{\circ}(f)$
 separating $\Theta$ from $\beta$, such that their corresponding ray pairs of $f$ land at distinct 
points, by Lemma~\ref{lem:actual-separators}. If $z_\Theta=z_{\Theta'}$,
this point would be the landing point of both separating
pairs, which is impossible. 
\end{proof}

\subsection{From horseshoes to conformal strip systems}\label{sec:strip-systems}
 A \emph{(cut-ray) strip} of \(f\) is a simply connected domain bounded by two disjoint ray pairs of \(f\).
It intersects the circle at infinity in two open arcs with disjoint closures, called \emph{channels}.
 Explicitly, if the angles of the ray pairs are $\{\theta,\eta\}$ and $\{\alpha,\beta\}$, 
we can write this strip as 
\begin{equation*}
U_f(\overline{\theta\eta},\overline{\alpha\beta}),
\end{equation*}
and $\alpha,\beta,\theta,\eta$ are called its \emph{boundary angles}.

\begin{definition}\label{def:strip}
By an \emph{$f^n$-conformal strip system of size $N$}, we mean a map
\begin{equation}\label{eq:21}
f^n:\bigcup_{j=1}^N U_j\to U
\end{equation}
with  $U$ and $(U_j)_{j=1}^N$ strips of $f$, such that
\begin{itemize}
\item each $\overline{U_j}$ lies in $U$ and separates the two components of $\partial U$; 
\item the closures of $U_1,\ldots,U_N$ are pairwise disjoint;
\item $f^n:U_j\to U$ is a conformal onto map for every $j=1,\ldots,N$.
\end{itemize}
\end{definition}

 For a strip $U$, set $U(t)=\{z\in U:G_f(z)<t\}$ and denote its two angular channels by $\Gamma^+(t),\Gamma^-(t)$.
Compact containment will be used only after truncating at finite potential, as in Figure~\ref{fig:conformal-strips}.

\begin{figure}[htbp]
\centering
\begin{tikzpicture}[x=1.05cm,y=1.05cm,line cap=round,line join=round,font=\small]
  \fill[gray!9] (-2,0)..controls(-2,1.0)..(-2.35,1.9)
    ..controls(0,2.2)..(2.35,1.9)..controls(2,1.0)..(2,0)
    ..controls(2,-1.0)..(2.35,-1.9)..controls(0,-2.2)..(-2.35,-1.9)
    ..controls(-2,-1.0)..cycle;
  \draw[black,thick] (-2,0)..controls(-2,1.0)..(-2.35,1.9);
  \draw[black,thick] (-2,0)..controls(-2,-1.0)..(-2.35,-1.9);
  \draw[black,thick] (2,0)..controls(2,1.0)..(2.35,1.9);
  \draw[black,thick] (2,0)..controls(2,-1.0)..(2.35,-1.9);
  \draw[black,dashed,->] (-2.35,1.9)--(-2.55,2.55);
  \draw[black,dashed,->] (2.35,1.9)--(2.55,2.55);
  \draw[black,dashed,->] (-2.35,-1.9)--(-2.55,-2.55);
  \draw[black,dashed,->] (2.35,-1.9)--(2.55,-2.55);
  \draw[blue!70!black,thick] (-2.35,1.9)..controls(0,2.2)..(2.35,1.9);
  \draw[blue!70!black,thick] (-2.35,-1.9)..controls(0,-2.2)..(2.35,-1.9);
  \node at (0,2.45) {$\Gamma^+(t)$};
  \node at (0,-2.45) {$\Gamma^-(t)$};
  \node[left] at (-2.55,2.55) {$\alpha$};
  \node[left] at (-2.55,-2.55) {$\beta$};
  \node[right] at (2.55,2.55) {$\theta$};
  \node[right] at (2.55,-2.55) {$\eta$};
  \fill[black] (-2,0) circle(1.7pt);
  \fill[black] (2,0) circle(1.7pt);
  \node[left] at (-2,0) {$z_a$};
  \node[right] at (2,0) {$z_b$};
  \foreach \x/\j in {-1.0/1,1.0/2}{
    \fill[blue!13] (\x-0.26,0)..controls(\x-0.26,0.7)..(\x-0.34,1.15)
      ..controls(\x,1.24)..(\x+0.34,1.15)..controls(\x+0.26,0.7)..(\x+0.26,0)
      ..controls(\x+0.26,-0.7)..(\x+0.34,-1.15)
      ..controls(\x,-1.24)..(\x-0.34,-1.15)..controls(\x-0.26,-0.7)..cycle;
    \draw[blue!70!black,thick] (\x-0.26,0)..controls(\x-0.26,0.7)..(\x-0.34,1.15)
      ..controls(\x,1.24)..(\x+0.34,1.15)..controls(\x+0.26,0.7)..(\x+0.26,0)
      ..controls(\x+0.26,-0.7)..(\x+0.34,-1.15)
      ..controls(\x,-1.24)..(\x-0.34,-1.15)..controls(\x-0.26,-0.7)..cycle;
    \draw[blue!70!black,dashed] (\x-0.34,1.15)--(\x-0.47,2.35);
    \draw[blue!70!black,dashed] (\x+0.34,1.15)--(\x+0.47,2.35);
    \draw[blue!70!black,dashed] (\x-0.34,-1.15)--(\x-0.47,-2.35);
    \draw[blue!70!black,dashed] (\x+0.34,-1.15)--(\x+0.47,-2.35);
    \fill[blue!70!black] (\x-0.26,0) circle(1.5pt);
    \fill[blue!70!black] (\x+0.26,0) circle(1.5pt);
    \node[blue!70!black] at (\x,0.48) {$U_{\j}$};
  }
  \node at (0,-0.6) {$U(t)$};
  \draw[blue!70!black,->] (3.3,1.9)--(2.4,1.9);
  \node[anchor=west] at (3.35,1.9) {$G_f=t$};
  \draw[blue!70!black,->] (3.3,1.15)--(1.45,1.15);
  \node[anchor=west] at (3.35,1.15) {$G_f=t/d^n$};
  \node[align=left,anchor=west] at (3.15,-0.4)
    {$f^n:U_j(t/d^n)\to U(t)$\\[3pt]conformal and onto\\[3pt]$\overline{U_j(t/d^n)}\subset U(t)$};
\end{tikzpicture}
\caption{A conformal strip system with two source strips.}\par
\label{fig:conformal-strips}
\end{figure}
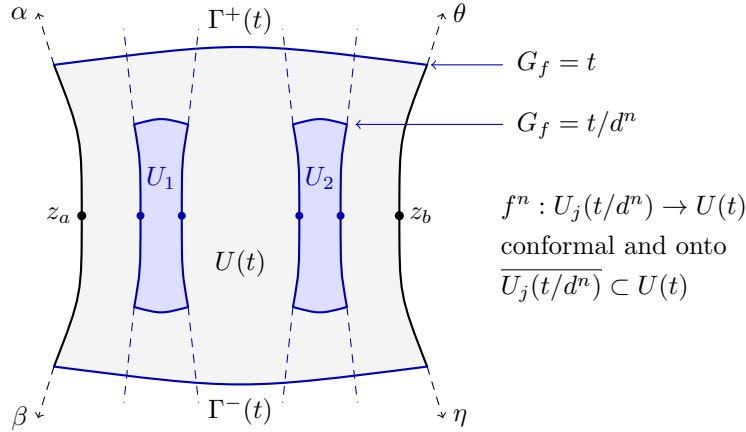

\begin{proposition}\label{pro:strip}
Suppose that $f^n:\bigcup_{j=1}^N U_j\to U$ is a conformal strip
system of size $N$. There is a compact set $\Theta\subset B(f)$ with
$\sigma_d^n(\Theta)=\Theta$ and
\begin{equation}\label{eq:strip-entropy-dimension}
 h_{\rm top}(\sigma_d^n|_\Theta)=\log N,
 \qquad \Hd\Theta=\frac{\log N}{n\log d}.
\end{equation}
\end{proposition}
\begin{proof}
Let \(h_j:U\to U_j\) be the inverse branches.  For sufficiently large \(t\), the set
$
U(t)=\{z\in U:G_f(z)<t\}
$
is a Jordan domain obtained by truncating the four boundary rays at potential \(t\) and joining them by the two equipotential arcs. Choose \(t\) so large that the same description applies to \(U_j(t/d^n)\). The Green-function identity then gives
\[
h_j(U(t))=U_j(t/d^n)\Subset U(t).
\]
Thus the finitely many inverse maps share a common contraction constant less than one with respect to the Poincaré metric on \(U(t)\).

For a word \(w=j_1\cdots j_r\), put
\[
h_w=h_{j_1}\circ\cdots\circ h_{j_r},\qquad
U_w=h_w(U),\qquad V_w=h_w(U(t)).
\]
The closures of successive \(V_w\) are nested, and their diameters tend to zero uniformly as \(r\to\infty\). Hence, for every infinite word \(\omega\in\{1,\ldots,N\}^{\mathbb N}\),
\[
\bigcap_{r\ge1}\overline{V_{\omega|_r}}=\{z_\omega\}\subset J_f.
\]

Let $\Gamma=\Gamma^+\cup\Gamma^-$ be the union of the two channels
of $U$, and let $\Gamma_j$ be the union of the channels of $U_j$.
The map $\sigma_d^n:\Gamma_j\to\Gamma$ maps the two component
intervals homeomorphically onto those of $\Gamma$, possibly
interchanging them. Its inverse $\beta_j:\Gamma\to\Gamma_j$ is
affine of slope $d^{-n}$ on each component. For a finite word $w$,
write $\Gamma_w=\beta_{j_1}\circ\cdots\circ\beta_{j_r}(\Gamma)$.
These are the two channels of $U_w$.

Note that all boundaries of $U_w$ consist of smooth ray pairs.  Moreover,
$\overline{\Gamma_j}\subset\Gamma$, and each component of
$\Gamma_w$ has length at most $d^{-nr}$. Thus
\[
 \bigcap_{r\ge1}\overline{\Gamma_{\omega|_r}}
   =\{\theta_\omega^+,\theta_\omega^-\},
\]
with one angle in each channel of $U$. Both angles are interior
points of every corresponding finite-level channel.

We claim that their external rays are smooth. Since otherwise, the
external radius at one of the two angles crashes at a point $c$ of
positive potential. Then the
crash point belongs to $U_{\omega|_r}$ for every $r\geq1$. But $c$ is an escaping
precritical point of $f$, so it is a critical point of $f^{nr}$ for
all sufficiently large $r$. This contradicts the conformality of
$f^{nr}:U_{\omega|_r}\to U$. The claim follows.

Note that the tail of either ray below potential $t/d^{nr}$ lies in
$V_{\omega|_r}$. Uniform shrinking
therefore shows that both rays land at $z_\omega$. In particular,
$z_\omega\in J_f$, and
\[
 \Theta:=\{\theta_\omega^+,\theta_\omega^-:
                  \omega\in\{1,\ldots,N\}^{\mathbb N}\}
 \subset B(f).
\]
This set is compact, since
\[
 \Theta=\bigcap_{r\ge1}\ \bigcup_{|w|=r}\overline{\Gamma_w},
\]
and each union is finite. The map $\sigma_d^n$ removes the first
symbol of the word and permutes the two channel labels according
to the first inverse branch. Consequently $\sigma_d^n$ maps $\Theta$ onto
itself. 

The coding identifies this system with a two-state symbolic system in which each state has one outgoing edge for each symbol \(1,\ldots,N\), and the terminal state is determined by the channel permutation of that branch. There are exactly \(2N^r\) paths of length \(r\), so the entropy is \(\log N\). Lemma~\ref{lem:formula} then yields \eqref{eq:strip-entropy-dimension}.
\end{proof}

Call a finite $\approx$-class \emph{realized} if the external rays of $f$ with these angles  are
smooth and land together. 

A point \(a\in J_\approx\) is called \emph{finitely biaccessible} if its class \(\Theta_a=\pi_\approx^{-1}(a)\) is finite and contains at least two angles. For two such distinct points \(a\ne b\), let \(\ell_a\) be the edge of \(\CH(\Theta_a)\) separating \(\Theta_a\) from \(\Theta_b\), and define \(\ell_b\) analogously. If both classes are realized, their landing points are distinct by Lemma~\ref{lem:actual-separators}. We denote by \[U_f(\ell_a,\ell_b)\] the strip bounded by the external ray pairs of $f$ corresponding to $\ell_a$ and $\ell_b$. \vspace{2pt}

Recall that $\Sigma_N(f):=\bigcup_{k=1}^N\sigma_d^k(\Sigma_0(f))$ for each $N\geq1$.


\begin{lemma}\label{lem:analytic}
Let $I=[a,b]\subset J_\approx$, where $a,b$ are finitely biaccessible
$g_\approx^n$-fixed points whose classes avoid
$\Sigma(f)\cup\Sigma_n(f)$. Suppose that
$I_1,\ldots, I_N$ are pairwise disjoint closed
subarcs of $(a,b)$ and that $g_\approx^n:I_j\to I$ is a
homeomorphism. If the closure of the component of
$J_\approx\setminus\{a,b\}$ containing $(a,b)$ is disjoint from the critical
values of $g_\approx^n$, then
\begin{equation}\label{eq:analytic-returns}
 f^n:\bigcup_{j=1}^N U_f(\ell_{a_j},\ell_{b_j})
       \longrightarrow U_f(\ell_a,\ell_b)
\end{equation}
is a conformal strip system, where $[a_j,b_j]=I_j$.
\end{lemma}
\begin{proof}
By Lemma~\ref{lem:finite-fibers}, the \(\approx\)-classes \(\Theta_a,\Theta_b\) of \(a,b\) are realized, so the strip \(U_f(\ell_a,\ell_b)\) is well-defined. Since \(a,b\) are not critical values of \(g_\approx^n\), the maps 
\[\sigma_d^n:\Theta_{a_j}\to\Theta_a\quad \text{and} \quad\sigma_d^n:\Theta_{b_j}\to\Theta_b\] 
are bijections for every \(j=1,\ldots,N\).
We first claim that \(\Theta_{a_j}\) (resp.\ \(\Theta_{b_j}\)) is realized as well for each \(j\). Thus, the strips $U_f(\ell_{a_j},\ell_{b_j}),j=1,\ldots,N$, are well-defined.\vspace{1pt}

Since \(\Theta_a\) avoids \(\Sigma(f)\cup\Sigma_n(f)\), by Lemma \ref{lem:lift}, every angle in
$$
\sigma_d^{-n}(\Theta_a)\supset \Theta_{a_j}
$$
is smooth. Choose \(\eta\in \Theta_{a_j}\), and let \(z_j\) be the landing point of \(R_f(\eta)\). By Lemma~\ref{lem:actual-separators}, no smooth  ray with angle outside \(\Theta_{a_j}\) can land at \(z_j\).

On the other hand, local lifting at \(z_j\) supplies at least one smooth ray above each angle of \(\Theta_a\). All these angles must therefore belong to \(\Theta_{a_j}\). Since \(\#\Theta_{a_j}=\#\Theta_a\), they exhaust \(\Theta_{a_j}\). Hence \(\Theta_{a_j}\) is realized. The claim is proved.\vspace{1pt}

Using the separation lemma (Lemma~\ref{lem:actual-separators}) again, we obtain that the ray pairs \[R_f(\ell_{a_j}),\ R_f(\ell_{b_j}), \quad j=1,\ldots,N\] are pairwise disjoint. Moreover, the relative positions of the intervals \(I_j\) and \(I\) imply that the strips \(U_f(\ell_{a_j},\ell_{b_j})\) and \(U_f(\ell_a,\ell_b)\) satisfy items 1 and 2 of Definition~\ref{def:strip}. Thus, to check that \eqref{eq:analytic-returns} is a conformal strip system, it suffices to verify
\begin{equation}\label{eq:conformal}
f^n:U_f(\ell_{a_j},\ell_{b_j})\to U_f(\ell_a,\ell_b)
\end{equation}
is a conformal map for each $j=1,\ldots,N$.

In the $g_\approx$-plane, the leaves $\ell_a,\ell_b$ (resp.\,$\ell_{a_j},\ell_{b_j}$) induce disjoint external ray pairs of $g_\approx$, which bound a strip $U_{\approx}$ (resp.\, $U_{\approx,j}$).
Since the component of \(J_\approx\setminus\{a,b\}\) containing \((a,b)\) is assumed to avoid the critical values of \(g^n_\approx\), the domain \(U_{\approx}\) contains no critical values of \(g^n_\approx\). Consequently, for each \(j\), the component of \(g_\approx^{-n}(U_{\approx})\) containing \((a_j,b_j)\) is exactly \(U_{\approx,j}\), and
\[
g^n_\approx: U_{\approx,j}\to U_{\approx}\text{ is a homeomorphism}.
\]
Let $\Gamma$ and $\Gamma_j$ be the unions of the two angular channels of
$U_\approx$ and $U_{\approx,j}$, respectively.
The action at infinity gives
\begin{equation}\label{eq:22}
\sigma_d^n:\Gamma_j\longrightarrow\Gamma
\quad\text{a homeomorphism}.
\end{equation}
It maps each component interval homeomorphically onto a target interval.
The corresponding actual strips have the same channel angles.

For simplicity, write
\[
U=U_f(\ell_a,\ell_b),\qquad U_j=U_f(\ell_{a_j},\ell_{b_j}).
\]
Choose \(t\) so large that
$
U(t)$ and 
$U_j(t/d^n)=\{z\in U_j:G_f(z)<t/d^n\}
$
are Jordan domains, each bounded by four ray segments and two equipotential arcs (see Figure \ref{fig:conformal-strips}). It follows from \eqref{eq:22} that $f^n$ maps $\partial U_j(t/d^n)$ homeomorphically onto $\partial U(t)$. Its winding number about a point of $U(t)$ is \(+1\). The Argument Principle then shows 
\[
f^n:U_j(t/d^n)\to U(t)
\]
is conformal and onto. Letting \(t\) increase proves that \(f^n:U_j\to U\) is conformal and onto. Combined with the separation and inclusion properties already established, this yields the desired conformal strip system.
\end{proof}

\subsection{Uniform pruning of interval horseshoes}\label{sec:pruning}
The preceding lemma realizes a conformal strip system from an interval horseshoe. For what follows, we need a periodic bound on the endpoints of $\ell_a,\ell_b$ independent of the polynomial \(f\). The following uniform pruning argument provides such a bound.

\begin{lemma}\label{lem:pruning}
Suppose $I\subset J_\approx$ has $M_0\ge2$ internally disjoint homeomorphic
 return branches of time $m_0$. For any finite set  $E\subset\Sone$ and any $\epsilon>0$, there exist  integers 
$m>m_0, M>M_0$, 
 an arc $[a,b]\subset{\rm int}(I)$
with $M$ pairwise disjoint homeomorphic return branches in $(a,b)$ of time $m$, such that:
\begin{enumerate}
\item $a,b$ are finitely biaccessible $g_\approx^m$-fixed points, such that
their $\approx$-classes avoid $\Sigma(f)\cup\Sigma_m(f)$, and every angle
in these classes is fixed by $\sigma_d^{2m}$;\vspace{2pt}
\item the closure of the component of $J_\approx\setminus\{a,b\}$
containing $(a,b)$ avoids the critical values of $g_\approx^m$
and the set $\pi_\approx(E)$;\vspace{2pt}
\item $\log M/m>\log M_0/m_0-\epsilon$, and $m,M$ depend only on
$m_0,M_0,\epsilon,\#E,d$.
\end{enumerate}
\end{lemma}
\begin{proof}
Let $\rho_I:J_\approx\to I$ be the natural
retraction: for $x\notin I$, the arc from $x$ to $I$ first meets
$I$ at $\rho_I(x)$. 

Fix a large integer $r\ge 2$.
Composing the original returns $r$ times gives $M_0^r$
homeomorphic full returns of time $m_r:=rm_0$, with pairwise disjoint 
interiors. 

We define a finite set $S_r$ of some special points in $I$ including: 
\begin{itemize}
\item the endpoints of $I$;
\item the $g_\approx^{m_r}$-fixed points in $I$ which correspond to infinite $\approx$-classes;
\item the $g_\approx^{m_r}$-fixed points in $I$ whose $\approx$-classes intersecting $\Sigma_{m_r}(f)$;
\item the points in $\rho_I({\rm CritV}(g_\approx^{m_r}))\bigcup \rho_I(\pi_\approx(E))$.
\end{itemize}
Since ${\rm CritV}(g_\approx^{m_r})=\bigcup_{k=1}^{m_r} g_\approx^k({\rm Crit}(g_\approx))$, we have
\[
\#\operatorname{CritV}(g^{m_r}_\approx)\le(d-1)m_r.
\]
Since every cycle of infinite classes contains a gap of degree greater than one, there are at most \((d-1)m_r\) such \(g_\approx^{m_r}\)-fixed points. Finally, \eqref{eq:5} yields
$
\#\Sigma_{m_r}(f)\le \mu_d m_r.
$
It follows that
\[
 \#S_r\le s_r:=(2+2^{d-1})(d-1)m_r+\#E+2.
\]

Since each special point lies in at most two return intervals, discarding all return intervals that meet \(S_r\) leaves at least \(M_0^r-2s_r\) of them. These remaining intervals lie in at most \(s_r+1\) components of \(I\setminus S_r\). Hence a component \(L\) contains at least
\[
\kappa_r=\left\lfloor\frac{M_0^r-2s_r}{s_r+1}\right\rfloor
\]
return intervals. Choose \(r\) large enough so that \(\kappa_r>M_0+2\), and select exactly \(\kappa_r\) of these intervals in their order along \(I\).

Each selected interval lies in \(\operatorname{int}(I)\) and maps homeomorphically onto \(I\), so it contains a fixed point of \(g^{m_r}_\approx\). Choose one such fixed point in each interval, and let \(a\) and \(b\) be those in the first and last intervals. For each of the remaining \(\kappa_r-2\) intervals, pull back \([a,b]\) by the return homeomorphism. These pullbacks are strictly contained in their source intervals, hence are pairwise disjoint closed subarcs of \((a,b)\) mapping homeomorphically onto \([a,b]\) by $g_\approx^{m_r}$, see Figure \ref{fig:interval-pruning}.
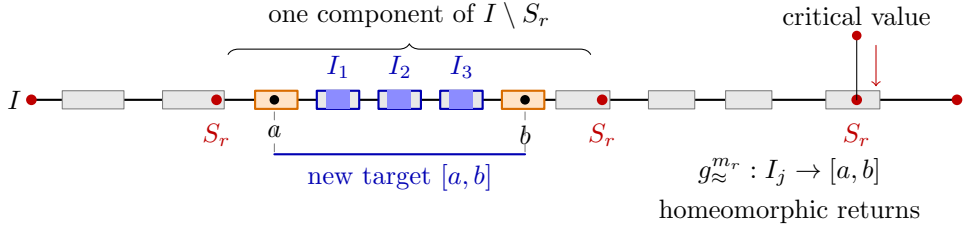
\begin{figure}[htbp]
\centering
\begin{tikzpicture}[x=1.02cm,y=1cm,line cap=round,line join=round,font=\small]
  \draw[black,thick] (0,0)--(12,0);
  \node[left] at (0,0) {$I$};
  \foreach \x/\y in {0.4/1.2,1.7/2.5,2.9/3.45,3.7/4.25,4.5/5.05,5.3/5.85,6.1/6.65,6.8/7.5,8.0/8.6,9.0/9.6,10.3/11.0}{
    \fill[gray!18] (\x,-0.12) rectangle (\y,0.12);
    \draw[gray] (\x,-0.12) rectangle (\y,0.12);
  }
  \foreach \x/\y in {2.9/3.45,6.1/6.65}{
    \fill[orange!22] (\x,-0.12) rectangle (\y,0.12);
    \draw[orange!85!black,thick] (\x,-0.12) rectangle (\y,0.12);
  }
  \foreach \x/\y in {3.7/4.25,4.5/5.05,5.3/5.85}{
    \draw[blue!70!black,thick] (\x,-0.12) rectangle (\y,0.12);
  }
  \foreach \x/\y in {3.82/4.12,4.62/4.92,5.42/5.72}{
    \fill[blue!45] (\x,-0.12) rectangle (\y,0.12);
  }
  \foreach \x in {0,2.4,7.4,10.7,12}{
    \fill[red!75!black] (\x,0) circle(2pt);
  }
  \draw[black] (10.7,0)--(10.7,0.85);
  \fill[red!75!black] (10.7,0.85) circle(1.8pt);
  \draw[red!75!black,->] (10.95,0.72)--(10.95,0.18);
  \node[above] at (10.7,0.85) {critical value};
  \node[red!75!black,below] at (2.4,-0.18) {$S_r$};
  \node[red!75!black,below] at (7.4,-0.18) {$S_r$};
  \node[red!75!black,below] at (10.7,-0.18) {$S_r$};
  \draw[decorate,decoration={brace,amplitude=4pt}] (2.55,0.6)--(7.25,0.6)
    node[midway,above=6pt] {one component of $I\setminus S_r$};
  \fill[black] (3.15,0) circle(1.8pt);
  \fill[black] (6.4,0) circle(1.8pt);
  \node[below] at (3.15,-0.2) {$a$};
  \node[below] at (6.4,-0.2) {$b$};
  \draw[blue!70!black,thick] (3.15,-0.72)--(6.4,-0.72);
  \draw[gray,dashed] (3.15,-0.16)--(3.15,-0.72) (6.4,-0.16)--(6.4,-0.72);
  \node[below,blue!70!black] at (4.78,-0.72) {new target $[a,b]$};
  \node[blue!70!black,above] at (3.98,0.15) {$I_1$};
  \node[blue!70!black,above] at (4.78,0.15) {$I_2$};
  \node[blue!70!black,above] at (5.58,0.15) {$I_3$};
  \node[align=center,anchor=west] at (8.0,-1.2)
    {$g_\approx^{m_r}:I_j\to[a,b]$\\[3pt]homeomorphic returns};
\end{tikzpicture}
\caption{Pruning the interval returns. Red points are the finite obstruction set $S_r$. In one complementary component, the first and last selected returns (orange) supply fixed points $a,b$. Pulling $[a,b]$ back in the interior returns leaves the smaller blue branches. Five selected returns and three retained branches are shown.}\label{fig:interval-pruning}\par
\end{figure}

The $\approx$-classes of $a,b$ are finite and avoid $\Sigma_{m_r}(f)$, by the
choice of $S_r$. They also avoid $\Sigma(f)$ by Lemma \ref{lem:lift}. 

The retraction marks show that the closure of the component of \(J_\approx\setminus\{a,b\}\) containing \((a,b)\) misses \(\operatorname{CritV}(g^{m_r}_\approx)\cup\pi_\approx(E)\). In particular, \(a,b\) are not critical points of \(g^{m_r}_\approx\). At each of them, the local map permutes the two directions along \(I\), so its second iterate fixes both. On the circle, the induced permutation of the finite \(\approx\)-class preserves cyclic order and fixes the complementary interval for that direction; hence it fixes every vertex. Therefore \(\sigma_d^{2m_r}\) fixes every angle in the \(\approx\)-classes of \(a,b\). Since \(a,b\) are interior points of an arc in the dendrite, their classes contain at least two angles. This proves (1) and (2).

Finally, observe that $\log \kappa_r / m_r \to \log M_0 / m_0$ as $r\to\infty$. Thus we can choose $r$ sufficiently large, depending only on $m_0,M_0$, $\#E$, $d$, and $\epsilon$, such that $$\log M / m > \log M_0 / m_0 - \epsilon,$$ where $M := \kappa_r - 2$ and $m = m_r = m_0 r$. This proves the lemma.
\end{proof}

Let \(\xi_f=\xi\) be a  marking portrait of \(\sim = \sim_{\LLL_{\rm bi}^{\rm p}(f)}\), with the induced separation graph \(\G_{\xi}\).
Combining the previous results in this section, we can establish a link between the number of closed paths in \(\G_{\xi}\) starting from a fixed vertex and a certain subset of the biaccessible set of \(f\). This link will play a key role in the proof of the continuity of biaccessible dimension.\vspace{2pt}

 Recall that \(C_n(v)=C_n(v,\G_{\xi})\) denotes the number of length-\(n\) closed paths in \(\G_{\xi}\) starting from \(v\); and the zero-arc vertices are defined in \eqref{eq:1}.

\begin{lemma}\label{lem:key}
Let $v$ be a non-zero-arc vertex of $\Gamma_\xi$, and fix
$m_0\ge1$ and $2\le M_0\le C_{m_0}(v)$.
Then for every $\epsilon>0$ and finite $E\subset\Sone$, there are
integers $M,m>0$, a conformal strip system
\[
 f^m:\bigcup_{j=1}^M U_j\longrightarrow
 U=U_f(\overline{\alpha\beta},\overline{\theta\eta}),
\]
and a compact $\sigma_d^m$-invariant set $\Theta\subset B(f)$ such that:
\begin{enumerate}
\item $M,m$ depend only on $M_0,m_0,\epsilon,\#E,d$;
\item $\alpha,\beta,\theta,\eta\in\operatorname{Fix}(\sigma_d^{2m})$;
\item  the two channels of $U$  and
the $\approx$-classes of $\alpha$ and $\theta$  all avoid $E$;
\item the dimension of \(\Theta\) satisfies 
\[
 \Hd\Theta=\frac{\log M}{m\log d}
          >\frac{\log M_0}{m_0\log d}-\epsilon.
\]
\end{enumerate}
\end{lemma}
\begin{proof}
Select exactly $M_0$ closed paths of length $m_0$ based at $v$. By Lemma~\ref{lem:returns}, we obtain a postcritical arc \(I_v\) that admits $M_0$ internally disjoint monotone return branches of time \(m_0\). By Lemma~\ref{lem:project-returns}, this yields a non-degenerate arc  \(L\subset J_\approx\) which admits \(M_0\) internally disjoint homeomorphic return branches of time \(m_0\).

Applying Lemma~\ref{lem:pruning} to \(L\), \(E\), and \(\epsilon\log d\), we obtain an arc \([a,b]\subset L\) with \(M\) homeomorphic return branches of time \(m\) satisfying
\[
\frac{\log M}{m}>\frac{\log M_0}{m_0}-\epsilon\log d.
\]
Its endpoints satisfy the period and avoidance conditions of that lemma. By Lemma~\ref{lem:analytic}, the resulting horseshoe is realized as the stated conformal strip system. The two channels correspond to the central component cut off by \(a,b\), and the avoidance of its closure yields assertion (3). Finally, Proposition~\ref{pro:strip} provides \(\Theta\) together with its dimension. The displayed inequality gives (4), while the uniform choice in Lemma~\ref{lem:pruning} gives (1).
\end{proof}

\subsection{The continuity of biaccessible dimension}
For a fixed polynomial \(f\), recall that \(\LLL_{\rm bi}(f)\) denotes the biaccessible lamination of \(f\), and \(\LLL_{\rm bi}^{\rm p}(f)\) denotes its perfect part (see Section~\ref{sec:biaccessible}). The sets \(B(f)\) and \(B(\LLL_{\rm bi}^{\rm p}(f))\) are the biaccessible sets of \(f\) and \(\LLL_{\rm bi}^{\rm p}(f)\), respectively.

Since \(\LLL_{\rm bi}^{\rm p}(f)\) is perfect, Lemma~\ref{lem:proper} implies that it is proper. Hence it induces a laminational equivalence relation \(\sim_f\). By Corollary~\ref{cor:differ}, the invariant laminations \(\LLL_{\rm bi}^{\rm p}(f)\) and \(\LLL_{\sim_f}\) differ in only countably many leaves. Let \(\xi_f=\xi\) be a marking portrait of \(\sim_f\), and let \(\G_{\xi}\) be its induced separation graph.

\begin{proposition}\label{prop:pointwise}
For any polynomial $f\in \PPPP_d$, we have
\[
 (\log d)\Hd B(f)=h(\xi).
\]
\end{proposition}

\begin{proof}
By Theorem~\ref{thm:equal1} and Corollary \ref{cor:differ}, we have
\[
 h:=h(\xi)=(\log d)\Hd B\bigl(\LLL_{\rm bi}^{\rm p}(f)\bigr).
\]
Lemma~\ref{lem:infinite-landings} then yields the upper bound
\[
(\log d)\Hd B(f)\le h.
\]
If \(h=0\), this bound together with nonnegativity gives equality. Henceforth, assume that the core entropy is positive, and fix \(\epsilon>0\) sufficiently small.

By Corollary~\ref{coro:many-returns}, there exists a non-zero-arc vertex \(v\) of \(\G_{\xi}\) such that
\[
\frac{\log C_{m_0}(v)}{m_0}>h-\epsilon\log d.
\]
By Lemma~\ref{lem:key} (with \(M_0=C_{m_0}(v)\), \(E=\varnothing\)), we obtain a set \(\Theta\subset B(f)\) satisfying
\[
\Hd B(f)\geq \Hd(\Theta)>\frac{\log C_{m_0}(v)}{m_0\log d}-\epsilon>\frac{h}{\log d}-2\epsilon.
\]
Letting \(\epsilon\downarrow0\), we obtain the lower bound and  this proposition is proved.
\end{proof}

\begin{proposition}\label{prop:lsc}
If $f_n\to f$ in $\PPPP_d$, then
\[
\Hd B(f)\le\liminf_{n\to\infty}\Hd B(f_n).
\]
\end{proposition}
\begin{proof}
If $\Hd B(f)=0$, the conclusion is obvious.
Otherwise fix $\epsilon>0$ sufficiently small. 
Combining 
Proposition~\ref{prop:pointwise} and  Corollary~\ref{coro:many-returns}, there exists a non-zero-arc vertex \(v\) of \(\G_{\xi}\) such that
\begin{equation}\label{eq:lsc-loops}
\frac{\log C_{m_0}(v)}{m_0}>(\log d)\Hd B(f)-\epsilon\log d.
\end{equation}

Let \(E=E_f\) be the set of periodic angles whose smooth rays land at parabolic points or at points of the critical orbit of \(f\). Then \(E\) is finite. Applying Lemma~\ref{lem:key} with $M_0=C_{m_0}(v),m_0$, error $\epsilon$, and exceptional set $E$, we obtain \(M,m>0\) and a conformal strip system
\begin{equation}\label{eq:2}
f^m:\bigcup_{j=1}^M U_f(\overline{\theta_j\eta_j},\overline{\alpha_j\beta_j})\to U_f(\overline{\theta\eta},\overline{\alpha\beta})
\end{equation}
in the \(f\)-plane,  such that 
\begin{itemize}
\item the $\approx$-classes containing $\alpha$ and $\theta$, respectively,  both avoid \(E\);
\item the numbers $M,m$ satisfy
\begin{equation}\label{eq:lsc-pruning-rate}
\frac{\log M}{m\log d}>\frac{\log C_{m_0}(v)}{m_0\log d}-\epsilon.
\end{equation}
\end{itemize}

The first item implies that
the landing points of $R_f(\theta)$ and $R_f(\alpha)$ are repelling and are not critical values of \(f^m\).
Thus, by Lemma \ref{lem:stable}, the conformal strip system \eqref{eq:2} in the $f$-plane can be perturbed to the one in the $f_n$ plane
\begin{equation*}
f_n^m:\bigcup_{j=1}^M U_{f_n}(\overline{\theta_j\eta_j},\overline{\alpha_j\beta_j})\to U_{f_n}(\overline{\theta\eta},\overline{\alpha\beta})
\end{equation*}
for every large $n$. Proposition \ref{pro:strip} gives a subset $\Theta_n\subset B(f_n)$ such that
\[\Hd B(f_n)\geq \Hd (\Theta_n)=\frac{\log M}{m\log d}\]
Combining inequalities \eqref{eq:lsc-pruning-rate} and \eqref{eq:lsc-loops}, we obtain
\[\Hd B(f_n)>\Hd B(f)-2\epsilon\]
for every sufficiently large $n$.
Taking the lower limit and then letting $\epsilon\downarrow0$
proves the proposition.
\end{proof}

{For lower semicontinuity, one fixed conformal return system
and its continuation suffice. For upper semicontinuity, the systems
belong to varying polynomials, so their return times must first be
bounded uniformly. This is the role of Lemma~\ref{lem:key}: its
uniform realization places all boundary angles in a single finite
set, from which we can select a constant subsequence.}

\begin{proposition}\label{prop:usc}
If $f_n\to f$ in $\PPPP_d$, then
\[
 \limsup_{n\to\infty}\Hd B(f_n)\le \Hd B(f).
\]
\end{proposition}

\begin{proof}
Pass to a subsequence such that
\[
 \Hd B(f_n)\longrightarrow \kappa=\limsup_{n\to\infty}\Hd B(f_n).
\]
If $\kappa=0$, the assertion is immediate. Suppose $\kappa>0$ and fix
$0<\epsilon<\kappa/3$. Choose a  marking portrait $\xi_n$ for the laminational equivalence relation $\sim_n$ induced by 
$\LLL_{\rm bi}^{\rm p}(f_n)$ for each $n$. By
Proposition~\ref{prop:pointwise},
\[
 \frac{h(\G_{\xi_n})}{\log d}=\Hd B(f_n)\longrightarrow \kappa.
\]

Passing to a subsequence, we may assume that all \(\xi_n\) have the same size and, by a diagonal selection from the finitely many labels at each vertex, that their labeled wedges $\WWW_n$ converge to a labeled wedge \(\WWW_\infty\).
Lemma \ref{lem:convergence-of-growth-rate} gives
$$h(\WWW_\infty)=\kappa\cdot \log d.$$ 
This limit is used only to select
finitely many paths; it need not be the labeled wedge of a marking portrait of the limit
polynomial $f$.

By Lemma~\ref{lem:meeting}, there are $M_0\ge2$ closed paths
based at one vertex $v$ of $\G(\WWW_\infty)$, all of length $m_0$, with
\begin{equation}\label{eq:usc-loop-rate}
 \frac{\log M_0}{m_0\log d}>\kappa-\epsilon.
\end{equation}
Only finitely many labels occur on these paths. Hence the same
paths occur in $\G_{\xi_n}$ for all sufficiently large $n$.
For each such $n$, $v$ is not a zero-arc vertex: otherwise
Lemma~\ref{lem:zero} would bound the number of paths starting
at $v$ independently of their length, whereas concatenating
the selected loops gives $M_0^r$ paths of length $rm_0$.

We choose the exceptional angles using the limit polynomial $f$.
Let $\operatorname{Per}(\sigma_d)$ denote the periodic angles, and set
\[
 E_f^{\rm esc}=\operatorname{Per}(\sigma_d)\cap
       \bigcup_{k\ge1}\sigma_d^k(\Sigma_0(f)).
\]
This set is finite: each of the finitely many forward orbits
starting in $\Sigma_0(f)$ contains at most one finite periodic
cycle.

Let $E_f^{\rm pp}$ be the set of smooth periodic angles whose external rays land at parabolic points or at points of ${\rm Post}(f)$ defined in \eqref{eq:-3}.
There are only finitely many such
points, and each receives only finitely many periodic rays. Thus
\[\text{$E_f:=E_f^{\rm esc}\cup E_f^{\rm pp}$\ \ is finite}.\]

Apply Lemma~\ref{lem:key} in each $f_n$-plane with the vertex $v$, the selected $M_0$ paths of length $m_0$, error $\epsilon$, and set $E_f$.
Its uniformity gives common integers $M,m$ and conformal strip
\begin{equation*}
f_n^m:\bigcup_{j=1}^M U_{f_n,j}\to U_{f_n}(\overline{\theta_n\eta_n},\overline{\alpha_n\beta_n})
\end{equation*}
for each $f_n$, such that
\begin{itemize}
\item the angles $\alpha_n,\beta_n,\theta_n$ and $\eta_n$ are $\sigma_d^{2m}$-fixed;
\item the $\approx$-classes containing $\alpha_n$ and $\theta_n$, respectively,  both avoid \(E_f\);
\item the numbers $M,m$ satisfy
\begin{equation}\label{eq:4}
\frac{\log M}{m\log d}>\frac{\log M_0}{m_0\log d}-\epsilon.
\end{equation}
\end{itemize}

By the first assertion, the boundary angles of \(U_{f_n}\) and \(U_{f_n,j}\) for all \(j, n\) lie in the finite set \(\sigma_d^{-m}(\operatorname{Fix}(\sigma_d^{2m}))\). Passing to a subsequence, we may assume that all boundary angles of \(U_{f_n}\) and \(U_{f_n,j}\)  are independent of \(n\). In particular, for every \(n\ge N\), we may write $(\alpha_n,\beta_n,\theta_n,\eta_n)=(\alpha,\beta,\theta,\eta)$.

By the second assertion, the rays of \(f\) with arguments \(\alpha,\beta,\theta,\eta\) are smooth, and their landing points are repelling and disjoint from the critical orbits. The choice of \(E_f\) and Lemma~\ref{lem:lift} further give smoothness for arguments in \[\sigma_d^{-m}(\{\alpha,\beta,\theta,\eta\}).\] Hence, by Lemma~\ref{lem:stable}, the smooth ray pairs of \(f_n\) on the boundaries of \(U_{f_n,j}\) and \(U_{f_n}\) converge to those of \(f\) with the same arguments. Therefore, we  obtain a conformal strip system for \(f\)
\[
f^m:\bigcup_{j=1}^M U_{f,j}\to U_{f}(\overline{\theta\eta},\overline{\alpha\beta}).
\]

By Proposition~\ref{pro:strip}, there exists a set \(\Theta\subset B(f)\) with $\Hd\Theta=\log M/(m\log d)$. Combining this with \eqref{eq:4} and \eqref{eq:usc-loop-rate}, we obtain
\[
\Hd B(f)\geq\Hd\Theta>\kappa-2\epsilon.
\]
Letting \(\epsilon\to0\) completes the proof of the proposition.
\end{proof}

\begin{proof}[Proof of Theorem \ref{thm:main}]
It follows directly from Propositions \ref{prop:lsc} and \ref{prop:usc}.
\end{proof}

\end{document}